\documentclass[11pt]{amsart}
\usepackage{amsmath,amssymb,a4wide,scalerel}
\usepackage{mathabx}
\usepackage{color}
\usepackage{xcolor,graphicx}
\usepackage{mathrsfs}
\usepackage{tikz,pgfplots}
\usepackage{subcaption}
\usepackage[normalem]{ulem}
\usepackage{hyperref}
\usepackage{dsfont}
\usepackage{comment}
\usepackage{enumitem}
\usepackage{ctable} % for \specialrule command

\newtheorem{theorem}{Theorem}
\newtheorem{lemma}[theorem]{Lemma}
\newtheorem{proposition}[theorem]{Proposition}
\newtheorem{remark}[theorem]{Remark}
\newtheorem{corollary}[theorem]{Corollary}

\newtheorem{assumption}{Assumption}
\newtheorem{example}{Example}

\newcommand{\dt}{\tau}

\newcommand{\R}{{\mathbb R}}
\newcommand{\N}{{\mathbb N}}
\newcommand{\E}{{\mathbb E}}

\newcommand*\dd{\mathop{}\!\mathrm{d}}

\newcommand{\sI}{{\upshape DP(0.5;1)}}
\newcommand{\sII}{{\upshape DP(0.5;2)}}
\newcommand{\hosI}{{\upshape DP(1.0;1)}}
\newcommand{\hosII}{{\upshape DP(1.0;2)}}

\usepackage{color}

\author[C.-E. Br\'ehier]{Charles-Edouard Br\'ehier}
              \address{Universite de Pau et des Pays de l'Adour, E2S UPPA, CNRS, LMAP,Pau, France}
              \email{charles-edouard.brehier@univ-pau.fr}

\author[D. Cohen]{David Cohen}
              \address{Department of Mathematical Sciences,
              Chalmers University of Technology and University of Gothenburg, 41296~Gothenburg, Sweden}
              \email{\tt david.cohen@chalmers.se}

\begin{document}

\title[Explicit domain preserving schemes for SDE\MakeLowercase{s}]{Explicit domain preserving numerical schemes for a class of stochastic differential equations}

\begin{abstract}
We construct and analyze numerical schemes for systems of stochastic differential equations, which preserve almost surely a given hypercube of arbitrary dimension. We propose a new general class of explicit schemes, such that for any choice of the time-step size the numerical solution takes values in the hypercube. We prove strong and weak convergence results for this general class of domain preserving numerical schemes, with strong order $1/2$ and weak order $1$ in general. We also construct a variant of the scheme which achieves strong order $1$ when the stochastic differential equation is driven by a one-dimensional Brownian motion. The convergence results are illustrated with numerical experiments. 
\end{abstract}

\maketitle
{\small\noindent


\section{Introduction}

Stochastic differential equations whose solutions are constrained almost surely to take values in a bounded domain arise in many areas of applied mathematics, biology, or epidemiology for instance. A prominent example is the stochastic Susceptible-Infected-Susceptible (SIS) epidemic model, see e.g. \cite{MR2821582}. The solution to this problem represent population proportions and thus must remain inside a bounded interval. For any given time-step size, standard numerical methods such as the Euler–Maruyama or Milstein schemes almost surely do not preserve bounded domains.

To address domain preservation, various numerical techniques have been proposed in the literature, which are briefly surveyed below. For the scalar stochastic SIS epidemic model, numerical schemes with strong order $1$ are proposed in the articles \cite{MR4274899,MR4220738,MR4790879} applying the Lamperti transform which allows to obtain a stochastic differential equation driven by additive noise, whereas an alternative numerical scheme with strong order $1$ using a logarithm transformation and the Milstein scheme is studied in the article~\cite{MR4792751}. Finally, the article~\cite{MR4444727} shows strong rate $1/2$ of convergence of a split-step numerical scheme with truncated noise.

Further domain preserving numerical schemes for other classes of stochastic differential equations have been proposed and analyzed in the articles \cite{MR2341800,MR3248050,MR4888024,MR4737060,erdogan2025}. The article~\cite{MR2341800} applies a domain preserving splitting scheme with strong order $1$ for a class of stochastic differential equations with regular enough coefficients. The Lamperti transform is applied in the articles~\cite{MR3248050,MR4888024,MR4737060,ulander2026} for the design and analysis of domain preserving schemes for scalar stochastic differential equations with strong order $1, 1.5$ and $2$. Finally, domain preserving schemes of strong order $1/2$ for systems of stochastic differential equations with solutions in an open hypercube are given in \cite{erdogan2025}. The numerical schemes proposed in \cite{erdogan2025} are based on positivity preserving numerical schemes for stochastic problems, a subject studied extensively, see for instance \cite{MR4795576,MR4729657,MR4780408,MR2186814, MR2898556, MR3433041, MR2367990, MR3732573, MR3082312, MR3331648, MR4242953, 
MR4847179,MR5004937,djurdjevac2026,bc26o1}. Observe that many existing domain preserving schemes are either tailored to scalar stochastic differential equations or restricted to specific classes of diffusion coefficients. 

In this article, we introduce and analyze an original and versatile class of explicit domain preserving numerical schemes for a class of stochastic differential equations. Precisely, almost surely the exact and numerical solutions, for any time-step size, preserve the hypercube $\mathcal{D}=[-1,1]^d$ in arbitrary dimension $d$. The stochastic differential equations are driven by multi-dimensional noise and may contain some drift component. We mainly consider stochastic differential equations with It\^o interpretation of the noise, however it is also shown how to adapt the proposed domain preserving schemes when the noise is interpreted in the Stratonovich sense.

To make the contributions of this article clear, in this section for pedagogical reasons we consider scalar stochastic differential equations
\begin{equation}\label{eq:SDEintro}
\left\lbrace
\begin{aligned}
&\dd X(t)=g(X(t))\dd B(t),\qquad t\ge 0,\\
&X(0)=x_0\in[-1,1]
\end{aligned}
\right.
\end{equation}
driven by a one-dimensional Brownian motion denoted by $\bigl(B(t)\bigr)_{t\ge 0}$, and with no drift component. Assuming that the function $g$ is Lipschitz continuous and satisfies the conditions $g(-1)=g(1)=0$, for any initial value $x_0\in[-1,1]$, almost surely one has  $X(t)\in[-1,1]$, for all $t\ge 0$, for the solution to~\eqref{eq:SDEintro}. Given a final time $T\in(0,\infty)$ and a time-step size $\dt=T/N$ with $N\in\N$, we construct a domain preserving numerical scheme with solution denoted by $\bigl(X_n\bigr)_{0\le n\le N}$ as follows. First, assuming that the mapping $g$ is of class $\mathcal{C}^2$ on $[-1,1]$, the diffusion coefficient $g$ can be decomposed as a product $g=f\sigma$, where $\sigma(x)=(x-1)(x+1)$ for all $x\in[-1,1]$ and $f$ is continuous, see Lemma~\ref{lem:g2f}. Second, if $t_n=n\dt$, on each interval $[t_n,t_{n+1}]$, the solution $X$ can be approximated by solving auxiliary stochastic differential equations 
\[
\left\lbrace
\begin{aligned}
&\dd X_n(t)=f(X_n)\sigma(X_n(t))\dd B(t),\qquad t\in[t_n,t_{n+1}],\\
&X_n(t_n)=X_n,
\end{aligned}
\right.
\]
where the factor $f$ in the decomposition $g=f\sigma$ is frozen to the value $X_n(t_n)=X_n$ at time $t=t_n$. Unfortunately, the exact solution $X_n(t)$ is not known. We propose a domain preserving scheme by considering its Stratonovich formulation
\[
\dd X_n(t)=-\frac{f(X_n)^2}{2}\bigl(\sigma'\sigma\bigr)(X_n(t))\dd t+f(X_n)\sigma(X_n(t))\circ\dd B(t).
\]
and by applying a Lie--Trotter splitting scheme using the flows $\phi$ and $\varphi$ associated to the mappings $-\frac12\sigma'\sigma$ and $\sigma$ (see Section~\ref{sec:schemes} for their expressions~\eqref{eq:flowphi} and~\eqref{eq:flowvarphi}). We then obtain the numerical scheme
\[
X_{n+1}=\varphi\bigl(f(X_n)\delta B_n,\phi(f(X_n)^2\dt)\bigr).
\]
It is straightforward to see that the proposed scheme preserves the domain $[-1,1]$, for any choice of the time-step size: if $X_0=x_0\in[-1,1]$, then almost surely $X_n\in[-1,1]$ for all $n\in\{1,\ldots,N\}$. We refer to Section~\ref{sec:schemes} for more details on the derivation of the numerical scheme above, for other examples and for its generalization for systems and noise in arbitrary dimension.

For the scalar problem~\eqref{eq:SDEintro}, we consider a general class of schemes, which can be written as
\begin{equation}\label{eq:schemeintro}
X_{n+1}=\Phi\left(f(X_n)^2,f(X_n)\delta B_n,X_n\right)
\end{equation}
The scheme above is only of strong order $1/2$. We also consider the following scheme
\begin{equation}\label{eq:scheme1intro}
X_{n+1}=\Phi\left(f(X_n)^2,f(X_n)\delta B_n+\frac12\bigl(gf'\bigr)[\delta B_n^2-\dt],X_n\right).
\end{equation}
which is shown to be of strong order $1$. Compared with~\eqref{eq:schemeintro}, the scheme~\eqref{eq:scheme1intro} contains a correction term, like in the definition of the Milstein scheme compared with the standard Euler--Maruyama scheme.

We refer to Section~\ref{sec:schemes} for the formulation of the considered scheme~\eqref{eq:schemeX} for systems of stochastic differential equations in arbitrary dimension $d$, driven by multi-dimensional noise and with non-zero drift. Appropriate conditions on the integrator $\Phi$ to ensure the preservation of the domain and the convergence of the scheme~\eqref{eq:schemeintro} to the solution of~\eqref{eq:SDEintro} are given in Section~\ref{sec:schemes}.

Let us now present the main results of this article, for the schemes~\eqref{eq:schemeintro} and~\eqref{eq:scheme1intro} applied to the It\^o stochastic differential equation~\eqref{eq:SDEintro}. We refer to Section~\ref{sec:main} for general and rigorous statements for the schemes presented in Section~\ref{sec:schemes}.
\begin{itemize}
\item In Proposition~\ref{propo:dpX}, we prove that the schemes~\eqref{eq:schemeintro} and~\eqref{eq:scheme1intro} are domain preserving, for any time-step size $\dt$: if $X_0=x_0\in\mathcal{D}=[-1,1]$, then for all $n\in\{1,\ldots,N\}$ one has almost surely $X_n\in\mathcal{D}=[-1,1]$.  
\item In Theorems~\ref{theo:strong}~and~\ref{theo:weak}, we prove that the scheme~\eqref{eq:schemeintro} converges to the solution to~\eqref{eq:SDEintro} with strong order $1/2$ and weak order $1$. Precisely, under appropriate regularity conditions on the diffusion coefficient $g$, the integrator $\Phi$ and the test function $\theta$, one has strong and weak error estimates
\[
\sup_{0\leq n\leq N}\left( \E\left[ |X_n-X(t_n)|^2 \right]\right)^{\frac12}\lesssim\dt^{\frac12},\quad |\E\left[ \theta(X_N) \right]-\E\left[ \theta(X(T))\right]|\lesssim\dt.
\]
\item In Theorem~\ref{theo:stronghigherorder}, we prove that the scheme~\eqref{eq:scheme1intro} converges to the solution to~\eqref{eq:SDEintro} with strong order $1$. Precisely, under additional conditions on the integrator $\Phi$ and appropriate regularity conditions on the diffusion coefficient $g$, one has strong estimates
\[
\sup_{0\leq n\leq N}\left( \E\left[ |X_n-X(t_n)|^2 \right]\right)^{\frac12}\lesssim\dt.
\]
\end{itemize}

All the main results are illustrated with numerical experiments, in order to validate them in various situations where the dimension $d$ of the system, the dimension $M$ of the Brownian motion change, and when a drift is added.

The presentation and the analysis of the domain preserving schemes are performed for stochastic differential equations where the noise is interpreted in the It\^o sense. In Section~\ref{sec:Strato}, we transfer and illustrate our results when the noise is interpreted in the Stratonovich sense, by considering the equivalent It\^o formulation. In particular, we illustrate the importance of considering the It\^o formulation when applying the freezing technique in the construction of the domain preserving scheme to ensure its consistency. We state convergence results, Corollaries~\ref{cor:strongStrato},~\ref{cor:weakStrato} and~\ref{cor:stronghigherorderStrato}, without proofs.

We provide detailed proofs of the convergence results, Theorems~\ref{theo:strong},~\ref{theo:weak} and~\ref{theo:stronghigherorder}.

Note that the domain preserving schemes introduced in this article are applied and studied in the context of parabolic semilinear stochastic partial differential equations in~\cite{bcc26}. In addition, a variant of the scheme~\eqref{eq:scheme1intro} is proposed and analyzed in~\cite{bc26o1} to obtain a positivity preserving scheme with strong order $1$ for a class of scalar stochastic differential equations.

This article is organized as follows. Section~\ref{sec:setting} presents the considered class of It\^o stochastic differential equations and the assumptions on drift and diffusion coefficients.
We explain the derivation of the abstract classes~\eqref{eq:schemeX} and~\eqref{eq:schemeX1} of domain preserving schemes in Section~\ref{sec:schemes}, with several examples implemented in practice for the numerical experiments. Then, in Section~\ref{sec:main}, we state and illustrate the main results of this article: domain preservation (Proposition~\ref{propo:dpX} in Section~\ref{sec:main-DP}), strong error estimates with order $1/2$ (Theorem~\ref{theo:strong} in Section~\ref{sec:main-strong}, weak error estimates (Theorem~\ref{theo:weak} in Section~\ref{sec:main-weak}) and strong error estimates with order $1$ for the scheme~\eqref{eq:schemeX1} (Theorem~\ref{theo:stronghigherorder} in Section~\ref{sec:main-strong1}). All the convergence results are illustrated with dedicated numerical experiments. Section~\ref{sec:Strato} presents how to adapt the domain preserving schemes and the convergence results for systems driven by Stratonovich noise. Finally, the proofs of the main results are provided in Section~\ref{sec:proof1} (Theorem~\ref{theo:strong}), Section~\ref{sec:proof2} (Theorem~\ref{theo:weak}) and Section~\ref{sec:proof3} (Theorem~\ref{theo:stronghigherorder}).

\section{Setting}\label{sec:setting}

\subsection{Notation}

Given an integer $d\in\mathbb{N}$, the domain of the problem is defined as $\mathcal{D}=[-1,1]^d$. The interior and the boundary of the domain are denoted by $\mathring{\mathcal{D}}=(-1,1)^d$ and $\partial\mathcal{D}$ respectively. Note that domains of type $\prod_{k=1}^{d}[L_k,R_k]$, with arbitrary real numbers $L_k<R_k$, may be considered with minor modifications. Without loss of generality we impose $L_k=-1$ and $R_k=+1$ for all $k\in·\{1,\ldots,d\}$ to simplify the presentation.

For any integer $D\in\N$ and any continuous function $\theta\colon\mathcal{D}\to \R^D$, set $\|\theta\|_\infty=\underset{x\in\mathcal{D}}\sup~\|\theta(x)\|$.

Let $(\Omega,\mathcal F, \mathbb P)$ be a probability space. Given an integer $M\in\N$, let $\bigl(B(t)\bigr)_{t\ge 0}$ be a standard $\R^M$-valued Brownian motion. For all $t\ge 0$, using the notation $B(t)=\bigl(B^1(t),\ldots,B^M(t)\bigr)$, the processes $\{B^m;~1\le m\le M\}$ are independent standard real-valued Brownian motions. If $z,z'\in\R^M$, the inner product of $z$ and $z'$ is denoted by $z\cdot z'$.

For all $m\in\{0,1,\ldots,M\}$, the vector field $g^m\colon\mathcal{D}\to\R^d$ is assumed at least to be of class $\mathcal{C}^2$. Stronger regularity conditions are imposed in the sequel for the statements of the convergence results. For all $x=(x_1,\ldots,x_d)\in\mathcal{D}$, let $g^m(x)=\bigl(g_1^m(x),\ldots,g_d^m(x)\bigr)$, with functions $g_k^m\colon\mathcal{D}\to\R$ for all $k\in\{1,\ldots,d\}$. Moreover, for all $k\in\{1,\ldots,d\}$ and all $x\in\mathcal{D}$, set $g_k(x)=\bigl(g_k^1(x),\ldots,g_k^M(x)\bigr)\in\R^M$.

The vector fields are assumed to satisfy the following condition.
\begin{assumption}\label{ass:boundary}
For all $m\in\{0,1,\ldots,M\}$ and all $k\in\{1,\ldots,d\}$, one has
\begin{equation*}
x_k\in\{-1,1\}\:\Rightarrow\: g_k^m(x)=0.
\end{equation*}
\end{assumption}

Let $x_0\in\mathcal{D}$ be a deterministic initial value. We consider the It\^o stochastic differential equation
\begin{equation}\label{eq:sdeNagumog}
\left\lbrace
\begin{aligned}
\dd X(t)&=g^0(X(t))\dd t+\sum_{m=1}^Mg^m(X(t))\dd B^m(t)\quad t\ge 0,\\
X(0)&=x_0,
\end{aligned}
\right.
\end{equation}
where the unknown $\bigl(X(t)\bigr)_{t\ge 0}$ is a continuous process with values in $\R^d$. The vector fields $g^m$ for $m\in\{0,\ldots,M\}$ satisfy Assumption~\ref{ass:boundary} and are of class $\mathcal{C}^2$ on $\mathcal{D}$, therefore~\eqref{eq:sdeNagumog} admits a unique solution $\bigl(X(t)\bigr)_{t\ge 0}$, which takes values in the domain $\mathcal{D}$: almost surely $X(t)\in\mathcal{D}$ for all $t\ge 0$.

Define the components $X_1,\ldots,X_d$ of the solution $X$ as follows: let $X(t)=\bigl(X_1(t),\ldots,X_{d}(t)\bigr)$ for all $t\ge 0$. Using the notation introduced above for the inner product in $\R^M$ and for the mappings $g_k\colon\mathcal{D}\to\R^M$, the $d$-dimensional stochastic differential equation~\eqref{eq:sdeNagumog} can be rewritten as the system of stochastic differential equations, for $k\in\{1,\ldots,d\}$,
\begin{equation}\label{eq:sdeNagumogSystem}
\left\lbrace
\begin{aligned}
\dd X_k(t)&=g_k^0(X(t))\dd t+\sum_{m=1}^{M}g_k^m(X(t))\dd B^m(t)\\
&=g_k^0(X(t))\dd t+g_k(X(t))\cdot \dd B(t)\quad t\ge 0,\\
X_k(0)&=x_{0,k},
\end{aligned}
\right.
\end{equation}
where the initial value is written as $x_0=\bigl(x_{0,1},\ldots,x_{0,d}\bigr)$. Note that almost surely $X_k(t)\in[-1,1]$ for all $t\ge 0$ and all $k\in\{1,\ldots,d\}$.

In the error analysis developed in this work, the following temporal regularity property of solutions $\bigl(X(t)\bigr)_{t\geq0}$ to the stochastic differential equation~\eqref{eq:sdeNagumog} is employed.
\begin{proposition}\label{propo:reg_exact}
Under Assumption~\ref{ass:boundary}, for all $T\in(0,\infty)$, there exists $C(T)\in(0,\infty)$ such that for all $t,s\in[0,T]$ one has 
\begin{equation}\label{eq:reg_exact}
\E[\|X(t)-X(s)\|^2]\le C(T)|t-s|.
\end{equation}
\end{proposition}

\begin{proof}[Proof of Proposition~\ref{propo:reg_exact}]
For all $t,s\in[0,T]$, one has
\[
X(t)-X(s)=\int_{s}^{t}g^0(X(r))\dd r+\sum_{m=1}^{M}\int_{s}^{t}g^m(X(r))\dd B^m(r).
\]
The vector fields $g^0,g^1,\ldots,g^M$ are continuous on $\mathcal{D}=[-1,1]^d$, thus they are bounded. As a result, applying the It\^o isometry property, one obtains
\[
\E[\|X(t)-X(s)\|^2]\le 2|t-s|^2\|g^0\|_\infty^2+2\sum_{m=1}^{M}|t-s|\|g^m\|_\infty^2.
\]
The proof is thus completed.
\end{proof}

\subsection{Equivalent formulation of the stochastic differential equation} 
In this section, we explain one of the main idea for the construction of domain preserving numerical scheme for the stochastic differential equation~\ref{eq:sdeNagumog}: we derive the equivalent formulation~\eqref{eq:sdeNagumo} related to auxiliary differential equations. 

For that purpose, introduce the mapping $\sigma\colon[-1,1]\to\R$ defined by
\begin{equation}\label{eq:sigma}
\sigma(y)=(y-1)(y+1),\qquad \forall~y\in[-1,1].
\end{equation}

The construction of domain preserving schemes for~\eqref{eq:sdeNagumog} is based on making a connection between the components $X_1,\ldots,X_d$ given by~\eqref{eq:sdeNagumogSystem} and two elementary scalar differential equations. On the one hand, the elementary scalar It\^o stochastic differential equation with diffusion coefficient is given by 
\begin{equation}\label{eq:Y}
\left\lbrace
\begin{aligned}
\dd Y(t)&=\sigma(Y(t))\dd \beta(t),\quad t\ge 0,\\
Y(0)&=y_0,
\end{aligned}
\right.
\end{equation}
driven by a standard-real valued Brownian motion $\bigl(\beta(t)\bigr)_{t\ge 0}$. On the other hand, the elementary scalar ordinary differential equation with vector field $\sigma$ is given by 
\begin{equation}\label{eq:y}
\left\lbrace
\begin{aligned}
\dot{y}(t)&=\sigma(y(t))\dd t,\quad t\in\R,\\
y(0)&=y_0,
\end{aligned}
\right.
\end{equation}
with arbitrary initial value $y_0\in[-1,1]$.

The connection is made by the following result.
\begin{lemma}\label{lem:g2f}
Let $G\colon[-1,1]\to\R$ be a function of class $\mathcal{C}^2$, which satisfies $G(-1)=G(1)=0$.

There exists a continuous mapping $F\colon[-1,1]\to\R$ such that 
\[
G(y)=F(y)\sigma(y)=F(y)(y-1)(y+1),\quad \forall~y\in[-1,1].
\]
Moreover, for any integer $p\in\N$, if $G$ is of class $\mathcal{C}^{p+2}$, then $F$ is of class $\mathcal{C}^p$.
\end{lemma}

\begin{proof}[Proof of Lemma~\ref{lem:g2f}]

First, since $G(-1)=0$, for all $y\in[-1,1]$ one has
\[
G(y)=G(y)-G(-1)=\int_{0}^{1}G'\bigl((\eta-1)+\eta y\bigr)(y+1)\dd \eta=(y+1)F_1(y),
\]
where the mapping $F_1$ is defined as
\[
F_1(y)=\int_{0}^{1}G'\bigl((\eta-1)+\eta y\bigr)\dd \eta,\qquad \forall~y\in[-1,1].
\]
Since $G$ is of class $\mathcal{C}^2$, the mapping $F_1$ defined above is of class $\mathcal{C}^1$ on $[-1,1]$.

Second, observe that $F_1(1)=\frac12 G(1)=0$. Therefore for all $y\in[-1,1]$ one has
\[
F_1(y)=F_1(y)-F_1(1)=\int_0^1 F_1'\bigl((1-\zeta)+\zeta y)(y-1)\dd \zeta=(y-1)F_2(y),
\]
where the mapping $F_2$ is defined as
\[
F_2(y)=\int_{0}^1 F_1'(1-\zeta+\zeta y)\dd \zeta=\int_0^1\int_0^1 G''\bigl(\zeta y+(1-\zeta)\eta+\eta-1\bigr)\eta\dd \zeta\dd \eta,\qquad \forall~y\in[-1,1].
\]
Since $G$ is of class $\mathcal{C}^2$, the mapping $F_2$ defined above is continuous on $[-1,1]$.

As a result, one obtains the required result with $F=F_2$: for all $y\in[-1,1]$, one has
\[
G(y)=(y+1)F_1(y)=(y+1)(y-1)F(y).
\]
Finally, given $p\in\N$, it is straightforward to check that if $G$ is of class $\mathcal{C}^{p+2}$, then $G''$ and $F=F_2$ are of class $\mathcal{C}^p$.

The proof of Lemma~\ref{lem:g2f} is thus completed. 
\end{proof}

For all $m\in\{0,\ldots,M\}$ and $k\in\{1,\ldots,d\}$, applying Lemma~\ref{lem:g2f} to the mapping $y\colon[-1,1]\mapsto g_k^m(x_1,\ldots,x_{k-1},y,x_{k+1},\ldots,x_d)$ for arbitrary $(x_1,\ldots,x_{k-1},x_{k+1},\ldots,x_d)\in[-1,1]^{d-1}$, there exists a function $f_k^m\colon\mathcal{D}\to\R$ such that one has
\begin{equation}\label{eq:sigmaF}
g_k^m(x)=\sigma(x_k)f_k^m(x),\qquad \forall~x\in\mathcal{D},
\end{equation}
where we recall that $\sigma$ is defined in~\eqref{eq:sigma}. Moreover, owing to the proof of Lemma~\ref{lem:g2f}, the function $f_k^m$ is given by
\begin{equation}\label{eq:fmk}
f_k^m(x)=\int_0^1\int_0^1 \partial_{x_k}^2g_k^m\bigl(x_1,\ldots,x_{k-1},\zeta x_k+(1-\zeta)\eta+\eta-1,x_{k+1},\ldots,x_d\bigr)\eta\dd \zeta\dd \eta.
\end{equation}
As a result, since for all $m\in\{0,\ldots,M\}$ the vector fields $g^m$ are of class $\mathcal{C}^2$ on the domain $\mathcal{D}$, the functions $f_k^m$ are continuous on $\mathcal{D}$ for all $m\in\{0,\ldots,M\}$ and $k\in\{1,\ldots,d\}$. Moreover, for any integer $p\in\N$, if for all $m\in\{0,\ldots,M\}$ the vector fields $g^m$ are of class $\mathcal{C}^{p+2}$ on the domain $\mathcal{D}$, the functions $f_k^m$ are of class $\mathcal{C}^p$ on $\mathcal{D}$ for all $m\in\{0,\ldots,M\}$ and $k\in\{1,\ldots,d\}$. Finally, one has
\[
\underset{|\alpha|\le p}\sup~\underset{x\in\mathcal{D}}\sup~|\partial^\alpha f_k^m|\le \underset{|\alpha|\le p+2}\sup~\underset{x\in\mathcal{D}}\sup~|\partial^\alpha g_k^m|,
\]
where for any multi-index $\alpha=(\alpha_1,\ldots,\alpha_d)\in\N_0^d$, one has $|\alpha|=\sum_{k=1}^{d}\alpha_k$ and $\partial^\alpha v=\partial_{x_1}^{\alpha_1}\ldots\partial_{x_d}^{\alpha_d}v$ for any function $v$ of class $\mathcal{C}^{|\alpha|}$.

For all $k\in\{1,\ldots,d\}$ and $x\in\mathcal{D}$, set
\begin{align}
f_k(x)&=\bigl(f_k^1(x),\ldots,f_k^M(x)\bigr)\in\R^M,\label{eq:deffk}\\
h_k(x)&=|f_k^0(x)|+\|f_k(x)\|^2\in\R^+.\label{eq:defhk}
\end{align}
Note that the mappings $f_k$ and $h_k$ are continous and bounded on $\mathcal{D}$.

The formulation~\eqref{eq:sdeNagumogSystem} of the stochastic differential equation~\eqref{eq:sdeNagumog} can then be written as the equivalent system, for $k\in\{1,\ldots,d\}$,
\begin{equation}\label{eq:sdeNagumo}
\left\lbrace
\begin{aligned}
\dd X_k(t)&=\sigma(X_k(t))f^0_k(X(t))\dd t+\sigma(X_k(t))\sum_{m=1}^Mf^m_k(X(t))\dd B^m(t),\\
&=\sigma(X_k(t))f^0_k(X(t))\dd t+\sigma(X_k(t))f_k(X(t))\dd B(t),\quad t>0,\\
X_k(0)&=x_{0,k}.
\end{aligned}
\right.
\end{equation}
The formulation~\eqref{eq:sdeNagumo} is the fundamental ingredient for the construction of domain preserving schemes for the SDE~\eqref{eq:sdeNagumog}. Furthermore, it also plays a key role in the error analysis.

\section{Domain preserving integrators}\label{sec:schemes}

The objective of this section is to describe a class of domain preserving integrators for the stochastic differential equation~\eqref{eq:sdeNagumog}. As emphasized above, it is convenient to consider the equivalent formulation~\eqref{eq:sdeNagumo}.

For an arbitrary time horizon $T\in(0,\infty)$ and a positive integer $N\in\N$, define the time-step size $\dt=T/N$.
Furthermore, for all $n\in\{0,1,\ldots,N\}$ set $t_n=n\dt$, and define the Brownian increments $\delta B_{n}^m=B^m(t_{n+1})-B^m(t_n)$ for $m=1,\ldots,M$. The notation $\delta B_n=(\delta B_{n}^1,\ldots,\delta B_{n}^M)=B(t_{n+1})-B(t_n)$ is employed below for increments of the $M$-dimensional Brownian motion $B$.

Given the time-step size $\dt=T/N$, the numerical solution defined by a domain preserving integrator is denoted by $\bigl(X_n\bigr)_{n=0,\ldots,N}$, where for all $n\in\{0,\ldots,N\}$, $X_n$ is a random variable with values in $\mathcal{D}\subset \R^d$. The components of $X_n$ are real-valued random variables denoted by $X_{n,1},\ldots,X_{n,d}$, taking values in $[-1,1]$.

At the initial time $t_0=0$, one sets $X_0=x_0$, i.e. $X_{0,k}=x_{0,k}$ for all $k\in\{1,\ldots,d\}$. Given $n\in\{0,\ldots,N-1\}$ and the numerical solution $X_n$ at time $t_n$, the numerical solution $X_{n+1}$ at time $t_{n+1}=t_n+\dt$ is constructed as follows. The formulation~\eqref{eq:sdeNagumo} suggests considering in the time interval $[t_n,t_{n+1}]$ the auxiliary stochastic differential equations
\begin{equation}\label{eq:auxsdescheme}
\left\lbrace
\begin{aligned}
\dd X_{n,k}(t)&=\sigma(X_{n,k}(t))f^0_k(X_n)\dd t+\sigma(X_{n,k}(t))\sum_{m=1}^Mf^m_k(X_n)\dd B^m(t),\quad t\in[t_n,t_{n+1}],\\
X_{n,k}(t_n)&=X_{n,k},
\end{aligned}
\right.
\end{equation}
indexed by $k\in\{1,\ldots,d\}$. In~\eqref{eq:auxsdescheme}, the functions $f_k^0,f_k^1,\ldots,f_k^M$ are frozen at the value $X_{n,k}(t_n)=X_{n,k}$ at time $t_n$, whereas the mapping $\sigma$ is not frozen. Recalling the relation $g_k^m(x)=f_k^m(x)\sigma(x_k)$, note that freezing the coefficient $g_k^m$ in the formulation~\eqref{eq:sdeNagumogSystem} of the system~\eqref{eq:sdeNagumo} would simply provide the standard explicit Euler--Maruyama, which is consistent but is not domain preserving. Keeping a variable $\sigma$ in the stochastic differential equation~\eqref{eq:auxsdescheme} is a key ingredient for the construction of domain preserving schemes.

Next, since the functions $f_k^0,f_k^1,\ldots,f_k^M$ are frozen on the time interval $[t_n,t_{n+1}]$, it suffices to be able to combine domain preserving integrators for the auxiliary stochastic and ordinary differential equations~\eqref{eq:Y} and~\eqref{eq:y} presented in Section~\ref{sec:setting}. The ordinary differential equation~\eqref{eq:y} can be solved exactly, its flow is denoted by $\varphi$ (its expression is provided below). However the exact solution to the stochastic differential equation~\eqref{eq:Y} is not known, and some numerical approximation procedure is required.

Consider integrators for~\eqref{eq:Y} of the type
\begin{equation}\label{eq:schemeY}
\left\lbrace
\begin{aligned}
Y_{n+1}&=\Phi(\dt,\delta \beta_n,Y_n),\quad n=0,\ldots,N-1,\\
Y_0&=y_0,
\end{aligned}
\right.
\end{equation}
where $\delta\beta_n=\beta(t_{n+1})-\beta(t_n)$ are increments of the Brownian motion $\bigl(\beta(t)\bigr)_{t\ge 0}$. Approximations $X_{n+1,k}$ of $X_{n,k}(t_{n+1})$, for instance, can be obtained by a Lie--Trotter splitting procedure: for all $k\in\{1,\ldots,d\}$, set
\[
X_{n+1,k}=\varphi(f_k^0(X_n)\dt,\cdot)\circ \Phi(f_k^1(X_n)^2\dt,f_k^1(X_n)\delta B_n^1,\cdot)\circ\ldots\circ\Phi(f_k^M(X_n)^2\dt,f_k^M(X_n)\delta B_n^m,\cdot)(X_{n,k}).
\]
Note that multiplying the time-step size $\dt$ by $f_k^m(X_n)^2$ in the first variable of $\Phi$, and multiplying the Brownian increment $\delta B_n^m$ by $f_k^m(X_n)$ in the second variable of $\Phi$ is natural owing to scaling properties of Brownian motion.

Even if using a Lie--Trotter splitting procedure is appealing, in this work we consider a different class of numerical schemes: given the initial value $X_0=x_0$, for all $k\in\{1,\ldots,d\}$, for all $n\in\{0,\ldots,N-1\}$, set
\begin{equation}\label{eq:schemeX}
\begin{aligned}
X_{n+1,k}&=\Phi\left(\sum_{m=1}^Mf^m_k(X_n)^2\dt,\sum_{m=1}^Mf^m_k(X_n)\delta B^m_{n}+f^0_k(X_n)\dt,X_{n,k}\right)\\
&=\Phi\left(\|f_k(X_n)\|^2\dt,f_k(X_n)\cdot\delta B_{n}+f^0_k(X_n)\dt,X_{n,k}\right),
\end{aligned}
\end{equation}
using the notation $\|f_k(x)\|^2=\sum_{m=1}^Mf^m_k(x)^2$ for all $x\in\mathcal{D}$. Note that the numerical scheme~\eqref{eq:schemeX} is explicit.

It remains to provide a class of integrators $\Phi$ which ensures that the numerical scheme~\eqref{eq:schemeX} is domain preserving and consistent. Before providing a general framework with appropriate order and regularity conditions, two examples are presented below. These examples are implemented in the numerical experiments reported below.

A basic ingredient employed for the two examples is the flow $\varphi$ associated with the ordinary differential equation~\eqref{eq:y}: for all $s\in\R$ and all $y\in[-1,1]$, one has the expression
\begin{equation}\label{eq:flowvarphi}
\varphi(s,y)=\frac{(1-e^{2s})+(1+e^{2s})y}{(1+e^{2s})+(1-e^{2s})y}.
\end{equation}
Then, for any $y_0\in[-1,1]$, the mapping $t\in\R\mapsto y(t)=\varphi(t,y_0)$ is the unique solution to the ordinary differential equation~\eqref{eq:y} (on positive and negative times). It is worth noting that $\varphi(s,y)\in[-1,1]$ for all $s\in\R$ and all $y\in[-1,1]$.

\begin{example}\label{exA}
The stochastic differential equation~\eqref{eq:Y} is interpreted in the It\^o sense, therefore its solution is not directly related to the flow $\varphi$ of~\eqref{eq:y}. However, its equivalent Stratonovich formulation is given by the stochastic differential equation
\begin{equation}\label{eq:Y-Strato}
\left\lbrace
\begin{aligned}
\dd Y(t)&=-\frac12\bigl(\sigma'\sigma\bigr)(Y(t))\dd t+\sigma(Y(t))\circ\dd B(t),\quad t\ge 0,\\
Y(0)&=y_0,
\end{aligned}
\right.
\end{equation}
and it can be approximated by applying a Lie--Trotter splitting strategy. Indeed, for all $y\in[-1,1]$ one has
\[
-\frac12\bigl(\sigma'\sigma\bigr)(y)=y-y^3,
\]
and the flow $\phi$ of the auxiliary ordinary differential equation
\begin{equation}\label{eq:y2}
\dot{y}(t)=-\frac12\bigl(\sigma'\sigma\bigr)(y(t))
\end{equation}
is given by
\begin{equation}\label{eq:flowphi}
\phi(s,y)=\frac{y}{\sqrt{y^2+(1-y^2)e^{-2s}}},\qquad \forall~s\in\R,~y\in[-1,1].
\end{equation}
It is worth noting that $\phi(s,y)\in[-1,1]$ for all $s\in\R$ and all $y\in[-1,1]$.

Applying a Lie--Trotter splitting scheme for the approximation of the Stratonovich stochastic differential equation~\eqref{eq:Y-Strato}, one obtains the following numerical scheme: with initial value $Y_0=y_0$, for all $n\in\{0,\ldots,N-1\}$, set
\[
\left\lbrace
\begin{aligned}
&\widehat{Y}_n=\phi(\dt,Y_n),\\
&Y_{n+1}=\varphi(\delta B_n,\widehat{Y}_n).
\end{aligned}
\right.
\]
This yields a scheme of the type~\eqref{eq:schemeY} for the auxiliary stochastic differential equation~\eqref{eq:Y}, with the integrator $\Phi$ defined as 
\begin{equation}\label{integratorA}
\Phi(s,\gamma,y)=\varphi\left(\gamma,\phi(s,y)\right),\quad \forall~s\geq0,\gamma\in\mathbb R,y\in[-1,1].  
\end{equation}
Alternatively, one could consider the integrator
\begin{equation}\label{integratorAbis}
\Phi(s,\gamma,y)=\phi\left(s,\varphi(\gamma,y)\right),\quad \forall~s\geq0,\gamma\in\mathbb R,y\in[-1,1],
\end{equation}
based on a Lie--Trotter splitting strategy, with reverse order of treatment for the drift and diffusion contributions. The integrators~\eqref{integratorA} and~\eqref{integratorB} satisfy similar properties, and in the sequel only the first version~\eqref{integratorA} is considered.
\end{example}

\begin{example}\label{exB}
As in Example~\ref{exA}, consider the Stratonovich equivalent formulation~\eqref{eq:Y-Strato} of the It\^o stochastic differential equation~\eqref{eq:Y}. On each time interval $[t_n,t_{n+1}]$, if $Y_n$ denotes the approximation of the solution at time $t_n$, then the coefficient $\sigma'$ may be frozen at $Y_n$, and one thus needs to solve the auxiliary Stratonovich stochastic differential equation, for $t\in[t_n,t_{n+1}]$,
\[
\left\lbrace
\begin{aligned}
\dd Y_n(t)&=-\frac12\sigma'(Y_n)\sigma(Y_n(t))+\sigma(Y_n(t))\circ \dd B(t),\\
&=\sigma(Y_n(t))\circ\bigl(-\frac12\sigma'(Y_n)\dd t+\dd B(t)\bigr),\\
Y_n(t_n)&=Y_n.
\end{aligned}
\right.
\]
The solution $Y_n(t_{n+1})$ and the numerical approximation $Y_{n+1}$ at time $t_{n+1}$ are then given as
\[
Y_{n+1}=Y_n(t_{n+1})=\varphi\left(\delta B_n-\frac12\dt\sigma'(Y_n),Y_n\right).
\]
This yields a scheme of the type~\eqref{eq:schemeY} for the auxiliary stochastic differential equation~\eqref{eq:Y}, with the integrator $\Phi$ defined as 
\begin{equation}\label{integratorB}
\Phi(s,\gamma,y)=\varphi\left(\gamma-\frac12\sigma'(y)s,y\right),\quad \forall~s\geq0,\gamma\in\mathbb R,y\in[-1,1]. 
\end{equation}
Compared to the scheme~\eqref{integratorA}, the scheme~\eqref{integratorB} only requires the evaluation of the flow $\varphi$ associated with the vector field $\sigma$.
\end{example}

It is worth noting that the schemes of type~\eqref{eq:schemeY} obtained with the integrators $\Phi$ given by~\eqref{integratorA} from Example~\ref{exA} and~\eqref{integratorB} from Example~\ref{exB} are domain preserving schemes for the auxiliary stochastic differential equation~\eqref{eq:Y} which converge with strong order $1$.

Let us exemplify the domain preserving numerical scheme~\eqref{eq:schemeX} for a simplified version of the stochastic differential equation~\eqref{eq:sdeNagumog}, with no drift, in dimension $d=1$ and with a one-dimensional Brownian motion ($M=1$). Considering
\begin{equation}\label{eq:SDEsimple}
\dd X(t)=g(X(t))\dd B(t),
\end{equation}
choosing the integrator~\eqref{integratorA} from Example~\ref{exA}, one obtains the domain preserving scheme
\begin{equation}\label{eq:schemeintA}
X_{n+1}=\varphi\left(f(X_n)\delta B_n,\phi(f(X_n)^2\dt,X_n)\right),\quad \forall~n\in\{1,\ldots,N\},
\end{equation}
and choosing the integrator~\eqref{integratorB} from Example~\ref{exB}, one obtains the domain preserving scheme
\begin{equation}\label{eq:schemeintB}
X_{n+1}=\varphi\left(f(X_n)\delta B_n-\frac12(\sigma'f^2)(X_n)\dt,X_n\right),\quad \forall~n\in\{1,\ldots,N\}.
\end{equation}

In order to ensure that the scheme~\eqref{eq:schemeX} is domain preserving -- see Proposition~\ref{propo:dpX} in Section~\ref{sec:main} -- the following condition is imposed on the integrator $\Phi$.
\begin{assumption}\label{ass:integratorDP}
The mapping $\Phi\colon\R^+\times\R\times[-1,1]\to\R$ is continuous and one has
\begin{equation}\label{eq:condPhiDP}
\Phi(s,\gamma,y)\in[-1,1],\quad \forall~s\ge 0,~\forall~\gamma\in\R,~\forall~y\in[-1,1].
\end{equation}
\end{assumption}
Note that the integrators~\eqref{integratorA} and~\eqref{integratorB} introduced in Examples~\ref{exA} and~\ref{exB} verify Assumption~\ref{ass:integratorDP}.

In Section~\ref{sec:main}, strong and weak error estimates for the scheme~\eqref{eq:schemeX} applied to~\eqref{eq:sdeNagumog} are established under appropriate regularity conditions on the vector fields $g^0,g^1,\ldots,g^M$ and on the mapping $\Phi$, which will be adapted to the statement of each result. In addition, the consistency is justified when the mapping $\Phi$ satisfies the following conditions.

\begin{assumption}\label{ass:integrator}
The mapping $\Phi\colon\R^+\times\R\times[-1,1]\to\R$ is of class $\mathcal{C}^2$, and satisfies the following conditions: for all $y\in[-1,1]$, one has
\begin{align}
&\Phi(0,0,y)=y,\label{eq:condPhi0}\\
&\partial_\gamma\Phi(0,0,y)=\sigma(y),\label{eq:condPhi1}\\
&\partial_s\Phi(0,0,y)+\frac12\partial_{\gamma}^2\Phi(0,0,y)=0.\label{eq:condPhi2}
\end{align}
\end{assumption}

For the analysis of the numerical scheme~\eqref{eq:schemeX}, it is convenient to introduce the auxiliary mapping
\begin{equation}\label{eq:psi}
\psi=\partial_s\Phi+\frac12\partial_{\gamma}^2\Phi.
\end{equation}
Note that~\eqref{eq:condPhi2} is equivalent to assuming that $\psi(0,0,y)=0$ for all $y\in[-1,1]$.

In Section~\ref{sec:main}, Theorems~\ref{theo:strong} and~\ref{theo:weak} show that the scheme~\eqref{eq:schemeX} converges to the solution~\eqref{eq:sdeNagumog} with strong order $1/2$ and weak order $1$. The numerical experiments illustrate that those orders of convergence are optimal.

One can achieve a higher strong order of convergence with another class of domain preserving integrators, when the stochastic differential equation~\eqref{eq:sdeNagumog} is driven by a one-dimensional Brownian motion, i.e. when $M=1$. Assuming that $M=1$, for simplicity, let $g=g^1$, $f_k=f_k^1$, and $B=B^1$. The proposed higher-order domain preserving scheme is given as follows: given the initial value $X_0=x_0$, for all $k\in\{1,\ldots,d\}$, for all $n\in\{0,\ldots,N-1\}$, set 
\begin{equation}\label{eq:schemeX1}
X_{n+1,k}=\Phi\left(f_k(X_n)^2\dt,f_k(X_n)\delta B_n+\frac12 (g\cdot\nabla)f_k(X_n)(\delta B_n^2-\dt)+f^0_k(X_n)\dt,X_{n,k}\right).
\end{equation}
If the mapping $\Phi$ satisfies Assumption~\ref{ass:integratorDP}, the scheme~\eqref{eq:schemeX1} is domain preserving. Concerning the convergence, Assumption~\ref{ass:integrator} needs to be reinforced to prove that the scheme~\eqref{eq:schemeX1} has strong order of convergence $1$.
\begin{assumption}\label{ass:integrator-higher}
The mapping $\Phi\colon\R^+\times\R\times[-1,1]\to\R$ satisfies the condition~\eqref{eq:condPhiDP} from Assumption~\ref{ass:integratorDP}, and the conditions~\eqref{eq:condPhi0} and~\eqref{eq:condPhi1} from Assumption~\ref{ass:integrator}. Moreover, for all $y\in[-1,1]$, one has
\begin{equation}\label{eq:condPhi-higher}
\partial_s\Phi(0,0,y)=-\frac12\sigma'(y)\sigma(y)\quad\text{and}\quad \partial_{\gamma}^2\Phi(0,0,y)=\sigma'(y)\sigma(y).
\end{equation}
\end{assumption}
In other words, the condition~\eqref{eq:condPhi2} from Assumption~\ref{ass:integrator} is replaced by the condition~\eqref{eq:condPhi-higher} in Assumption~\ref{ass:integrator-higher}. Clearly, the condition~\eqref{eq:condPhi-higher} implies the condition~\eqref{eq:condPhi2}.

The following result shows that Assumptions~\ref{ass:integratorDP},~\ref{ass:integrator} and~\ref{ass:integrator-higher} are satisfied by the integrators~\eqref{integratorA} and~\eqref{integratorB} introduced in Examples~\ref{exA} and~\ref{exB}.
\begin{lemma}\label{lem:examples1}
Assume that $\Phi$ is defined by~\eqref{integratorA} or by~\eqref{integratorB}. Then Assumptions~\ref{ass:integratorDP},~\ref{ass:integrator} and~\ref{ass:integrator-higher} are satisfied.
\end{lemma}
The proof of Lemma~\ref{lem:examples1} is provided in Appendix~\ref{app}.

As above, let us exemplify the domain preserving numerical scheme~\eqref{eq:schemeX1} for the simplified version~\eqref{eq:SDEsimple} of the stochastic differential equation~\eqref{eq:sdeNagumog}, with no drift, in dimension $d=1$ and with a one-dimensional Brownian motion. Choosing the integrator~\eqref{integratorA} from Example~\ref{exA}, one obtains the domain preserving scheme
\begin{equation}\label{eq:scheme1intA}
X_{n+1}=\varphi\left(f(X_n)\delta B_n+\frac12(f'g)(X_n)(\delta B_n^2-\dt),\phi(f(X_n)^2\dt,X_n)\right),\quad \forall~n\in\{1,\ldots,N\},
\end{equation}
and choosing the integrator~\eqref{integratorB} from Example~\ref{exB}, one obtains the domain preserving scheme
\begin{equation}\label{eq:scheme1intB}
X_{n+1}=\varphi\left(f(X_n)\delta B_n+\frac12(f'g)(X_n)(\delta B_n^2-\dt)-\frac12(\sigma'f^2)(X_n)\dt,X_n\right),\quad \forall~n\in\{1,\ldots,N\}.
\end{equation}

It remains to provide regularity assumptions on the mapping $\Phi$. Given an integer $p\in\N$ and a positive real number $\beta\in(0,\infty)$, for any function $\Phi\colon\R^+ \times\R\times[-1,1]\to\R$ of class $\mathcal{C}^p$, define
\begin{equation}\label{eq:regPhi}
\|\Phi\|_{p,\beta}=\sum_{|\alpha|\le p}\underset{s\ge 0,\gamma\in\R}\sup~e^{-\beta(s+|\gamma|)}\left|\partial_s^{\alpha_s}\partial_\gamma^{\alpha_\gamma}\partial_y^{\alpha_y}\Phi(s,\gamma,y) \right|\in[0,\infty],
\end{equation}
where $\alpha=(\alpha_s,\alpha_\gamma,\alpha_y)\in\N_0^3$ and $|\alpha|=\alpha_s+\alpha_\gamma+\alpha_y$.

In Theorems~\ref{theo:strong},~\ref{theo:weak} and~\ref{theo:stronghigherorder}, it is imposed that there exists a sufficiently large integer $p$ (depending on the result) and a positive real number $\beta$ such that $\|\Phi\|_{p,\beta}<\infty$. Imposing this type of condition is natural, indeed the integrators~\eqref{integratorA} and~\eqref{integratorB} introduced in Examples~\ref{exA} and~\ref{exB} satisfy the following property.
\begin{lemma}\label{lem:examples2}
Assume that $\Phi$ is defined by~\eqref{integratorA} or by~\eqref{integratorB}. Then, for all $p\in\N$, there exists $\beta\in(0,\infty)$ such that $\|\Phi\|_{p,\beta}<\infty$.
\end{lemma}
The proof of Lemma~\ref{lem:examples2} is provided in Appendix~\ref{app}.

\section{Main results and numerical illustrations}\label{sec:main}

In this section, we present the main theoretical results on the proposed numerical schemes~\eqref{eq:schemeX}~and~\eqref{eq:schemeX1} (when $M=1$), applied to the stochastic differential equation~\eqref{eq:sdeNagumog}. 
The results are illustrated with numerical experiments using several domain preserving schemes denoted as follows:
\begin{itemize}
\item the scheme~\eqref{eq:schemeX} applied with the integrator~\eqref{integratorA} from Example~\ref{exA} is denoted by~\sI,
\item the scheme~\eqref{eq:schemeX} applied with the integrator~\eqref{integratorB} from Example~\ref{exB} is denoted by~\sII,
\item the scheme~\eqref{eq:schemeX1} applied when $M=1$ with the integrator~\eqref{integratorA} from Example~\ref{exA} is denoted by~\hosI,
\item the scheme~\eqref{eq:schemeX1} applied when $M=1$ with the integrator~\eqref{integratorB} from Example~\ref{exB} is denoted by~\hosII.
\end{itemize}
Moreover, we compare the above domain preserving schemes with the standard Euler--Maruyama and Milstein schemes which are denoted by {\upshape EM} and {\upshape Mil}, respectively, see for instance \cite{MR1214374,MR4369963,MR4241457}. 

Section~\ref{sec:main-DP} is devoted to check and illustrate the domain preservation for the schemes~\eqref{eq:schemeX} and~\eqref{eq:schemeX1}. Strong error estimates with order $1/2$ for the scheme~\eqref{eq:schemeX} are stated in Theorem~\ref{theo:strong} and illustrated by numerical experiments in Section~\ref{sec:main-strong}. Weak error estimates with order $1$ for the scheme~\eqref{eq:schemeX} are stated in Theorem~\ref{theo:weak} and illustrated by numerical experiments in Section~\ref{sec:main-weak}. Strong error estimates with order $1$ when $M=1$ for the scheme~\eqref{eq:schemeX1} are stated in Theorem~\ref{theo:stronghigherorder} and illustrated by numerical experiments in Section~\ref{sec:main-strong1}.

\subsection{Domain preservation}\label{sec:main-DP}

First, we check that the schemes~\eqref{eq:schemeX}~and~\eqref{eq:schemeX1} constructed in Section~\ref{sec:schemes} are domain preserving, if the mapping $\Phi$ satisfies Assumption~\ref{ass:integratorDP}.

\begin{proposition}\label{propo:dpX} 
Assume that the vector fields $g^m$ for $m\in\{0,1,\ldots,M\}$ satisfy Assumption~\ref{ass:boundary}. Assume that the mapping $\Phi$ satisfies Assumption~\ref{ass:integratorDP}. The schemes~\eqref{eq:schemeX}~and~\eqref{eq:schemeX1} are domain preserving: for any initial value $x_0\in\mathcal{D}$, for any time $T\in(0,\infty)$ and for any time-step size $\dt=T/N$ with $N\in\N$, almost surely one has
\[
X_n\in\mathcal{D},\quad \forall~n\in\{0,\ldots,N\}.
\]
\end{proposition}

\begin{proof}[Proof of Proposition~\ref{propo:dpX}]
It suffices to apply a recursion argument. For $n=0$, it is assumed that $X_0=x_0\in\mathcal{D}$. Given $n\in\{0,\ldots,N-1\}$, assuming that $X_n\in\mathcal{D}$, then for all $k\in\{1,\ldots,d\}$ one has $X_{n,k}\in[-1,1]$. Then, owing to the condition~\eqref{eq:condPhiDP} from Assumption~\ref{ass:integratorDP}, using either~\eqref{eq:schemeX} or~\eqref{eq:schemeX1}, one obtains $X_{n+1,k}\in[-1,1]$ for all $k\in\{1,\ldots,d\}$, which yields $X_{n+1}\in\mathcal{D}$. This concludes the recursion argument and the proof of Proposition~\ref{propo:dpX}.
\end{proof}

\begin{remark}
If for the definition of the schemes~\eqref{eq:schemeX} and~\eqref{eq:schemeX1} one considers the integrators~\eqref{integratorA} and~\eqref{integratorB} from Examples~\ref{exA} and~\ref{exB}, which are based on the flows $\varphi$ and $\phi$ of the ordinary differential equations~\eqref{eq:y} and~\eqref{eq:y2}, a stronger property is obtained. Recall that $\mathring{\mathcal{D}}=(-1,1)^d$ is the interior of the domain $\mathcal{D}=[-1,1]^d$. If $X_0\in\mathring{\mathcal{D}}$, then almost surely one has $X_n\in\mathring{\mathcal{D}}$ for all $n\in\{0,\ldots,N\}$.
\end{remark}

In order to illustrate the domain preservation for the scheme~\eqref{eq:schemeX} and its superiority compared to the standard Euler--Maruyama scheme, results of a numerical experiment are reported in Table~\ref{table:DP1}. It is performed in dimension $d=2$, for a one-dimensional Brownian motion ($M=1$), on the time interval $[0,T]=[0,20]$, with a time-step size $\dt=2^{4}$, and with the initial value $x_0=(0.5,-0.5)$. Two examples of vector fields $g^0$ and $g^1$ are considered:
\[
g^m(x_1,x_2)=\bigl(f_1^m(x_1,x_2)\sigma(x_1),f_2^m(x_1,x_2)\sigma(x_2)\bigr),\quad \forall~(x_1,x_2)\in[-1,1]^2,~m\in\{0,1\},
\]
with either
\[
f_1^0(x_1,x_2)=x_1x_2,~f_2^0(x_1,x_2)=\cos(x_1)\sin(x_2),\quad f_1^1(x_1,x_2)=1.5x_1^2x_2,~f_2^1(x_1,x_2)=1.5x_1x_2^2,
\]
or
\[
f_1^0(x_1,x_2)=x_1^2x_2,~f_2^0(x_1,x_2)=\cos(x_1)+x_2^2,\quad f_1^1(x_1,x_2)=\exp(x_2),~f_2^1(x_1,x_2)=\exp(x_1+x_2).
\]
In this numerical experiment, $N_{\rm MC}=10^2$ independent realizations for each scheme are simulated and the proportions of numerical solutions which remain in the domain $\mathcal{D}=[-1,1]^2$ at all times, is reported in Table~\ref{table:DP1}. More precisely, one computes the proportions of realizations such that $X_{n,1}\in[-1,1]$ and such that $X_{n,2}\in[-1,1]$ for all $n\in\{0,\ldots,N\}$. The results reported in Table~\ref{table:DP1} validate the property that the scheme~\eqref{eq:schemeX} is domain preserving, as stated in Proposition~\ref{propo:dpX}, whereas the standard Euler--Maruyama scheme is not.

\begin{table}[h]
\begin{center}
\resizebox{\columnwidth}{!}{%
\begin{tabular}{|c | c |c | c| c| }
   %\multicolumn{}{c}{}\\
   %\hline
   \cline{3-5}
    \multicolumn{1}{c}{} &   & \sI  & \sII  & {\upshape EM}  \\
  \hline
  $f^0(x_1,x_2)$ & $f^1(x_1,x_2)$ & ($x_1$, $x_2$) & ($x_1$, $x_2$) & ($x_1$, $x_2$)  \\
  \specialrule{.1em}{.05em}{.05em}
  $(x_1x_2,\cos(x_1)\sin(x_2))^T$ & $1.5(x_1^2x_2,x_1x_2^2)^T$ & $100/100$, $100/100$ & $100/100$, $100/100$ & $91/100$, $92/100$ \\
  \hline
  $(x_1^2x_2,\cos(x_1)+x_2^2)^T$ & $(\exp(x_2),\exp(x_1+x_2))^T$ & $100/100$, $100/100$ & $100/100$, $100/100$ & $90/100$, $14/100$ \\
  \hline
\end{tabular}
}
\end{center}
\caption{Validation of the domain preservation for the schemes 
~\sI~and~\sII~(scheme~\eqref{eq:schemeX} with the integrator $\Phi$ given by~\eqref{integratorA} or~\eqref{integratorB}), and comparison with the Euler--Maruyama scheme {\upshape EM}. Proportion of realizations which remain in the domain.}
\label{table:DP1}
\end{table}

The numerical experiment is repeated to illustrate the domain preservation for the scheme~\eqref{eq:schemeX1} and its superiority compared to the Milstein scheme. The results are obtained with the same parameters as for Table~\ref{table:DP1} and are reported in Table~\ref{table:DP2}. The results reported in Table~\ref{table:DP2} validate the property that the scheme~\eqref{eq:schemeX1} is domain preserving, as stated in Proposition~\ref{propo:dpX}, whereas the standard Milstein scheme is not.

\begin{table}[h]
\begin{center}
\resizebox{\columnwidth}{!}{%
\begin{tabular}{|c | c |c | c| c| }
   %\multicolumn{}{c}{}\\
   %\hline
   \cline{3-5}
    \multicolumn{1}{c}{} &   & \hosI  & \hosII  & {\upshape Mil}  \\
  \hline
  $f^0(x_1,x_2)$ & $f^1(x_1,x_2)$ & ($x_1$, $x_2$) & ($x_1$, $x_2$) & ($x_1$, $x_2$)  \\
  \specialrule{.1em}{.05em}{.05em}
  $(x_1x_2,\cos(x_1)\sin(x_2))^T$ & $1.5(x_1^2x_2,x_1x_2^2)^T$ & $100/100$, $100/100$ & $100/100$, $100/100$ & $99/100$, $100/100$ \\
  \hline
  $(x_1^2x_2,\cos(x_1)+x_2^2)^T$ & $(\exp(x_2),\exp(x_1+x_2))^T$ & $100/100$, $100/100$ & $100/100$, $100/100$ & $90/100$, $89/100$ \\
  \hline
\end{tabular}
}
\end{center}
\caption{Validation of the domain preservation for the schemes 
\hosI~and~\hosII~(scheme~\eqref{eq:schemeX1} with the integrator $\Phi$ given by~\eqref{integratorA} or~\eqref{integratorB}), and comparison with the Milstein scheme {\upshape Mil}. Proportion of realizations which remain in the domain.}
\label{table:DP2}
\end{table}

Figure~\ref{fig:plotDP} also illustrates the domain preservation property for the schemes~\eqref{eq:schemeX} and~\eqref{eq:schemeX1} stated in Proposition~\ref{propo:dpX}, and their superiority compared with the standard Euler--Maruyama and Milstein schemes. The figures show the evolutions of the two components $X_{n,1}$ and $X_{n,2}$ as time evolves for $t\in\{t_n=n\dt,~n=0,\ldots,N\}$ with final time $T=1$ and time-step size $\dt=2^{-4}$. In Figure~\ref{fig:plotDPa}, the domain preserving scheme~\sI (scheme~\eqref{eq:schemeX} with integrator~\eqref{integratorA} from Example~\ref{exA}) and the standard Euler--Maruyama scheme are considered. In Figure~\ref{fig:plotDPb}, the domain preserving scheme~\hosI (scheme~\eqref{eq:schemeX1} with integrator~\eqref{integratorA} from Example~\ref{exB}) and the Milstein scheme are considered. Variants with the integrator~\eqref{integratorB} from Example~\ref{exB} are not considered as the results would be similar. One observes in Figure~\ref{fig:plotDP} that the proposed schemes are domain preserving, contrary to the standard Euler--Maruyama and Milstein schemes.

\begin{figure}[h]
\centering
\begin{subfigure}{.5\textwidth}
\includegraphics[width=\textwidth]{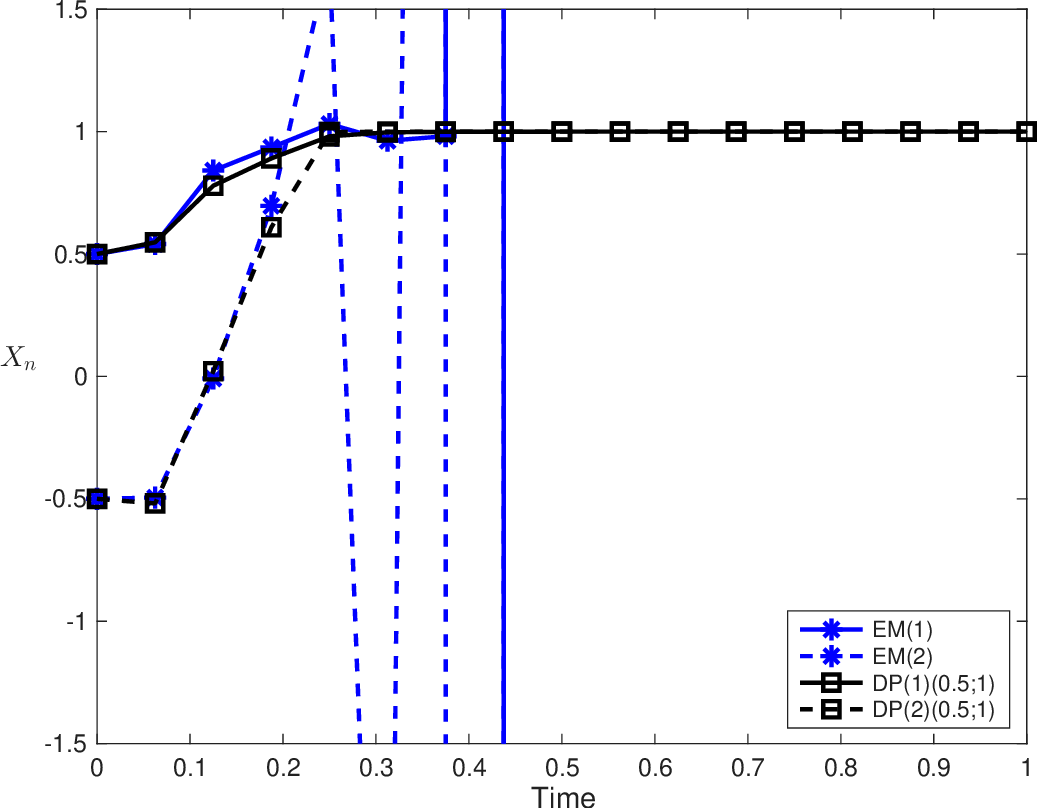} 
\caption{\sI~$\&$~Euler--Maruyama schemes}
\label{fig:plotDPa}
\end{subfigure}%
\begin{subfigure}{.5\textwidth}
\includegraphics[width=\textwidth]{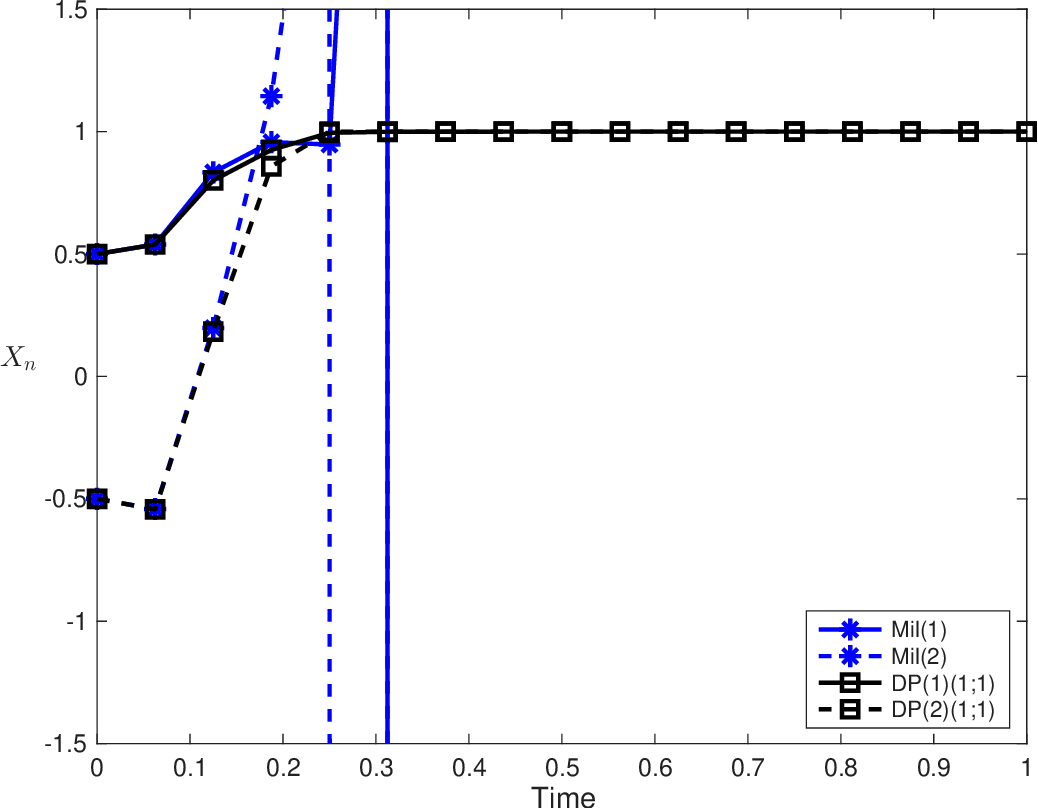}
\caption{\hosI~$\&$~Milstein schemes}
\label{fig:plotDPb}
\end{subfigure}
\caption{Illustration of Proposition~\ref{propo:dpX} (domain preservation for the schemes~\sI~and~\hosI). Evolution of the components of the schemes.}
\label{fig:plotDP}
\end{figure}

\subsection{Strong error estimates}\label{sec:main-strong}

This section is devoted to the first convergence result concerning the domain preserving scheme~\eqref{eq:schemeX}: it converges in the mean-square sense, with order $1/2$. Error estimates are stated in Theorem~\ref{theo:strong}. The result is then illustrated with numerical experiments, where different parameters (dimension $d$ of the system, dimension $M$ of the Brownian motion, presence or absence of drift) are varied.

\begin{theorem}\label{theo:strong}
Assume that the vector fields $g^m$ for $m\in\{0,1,\ldots,M\}$ are of class $\mathcal{C}^2$ on $\mathcal{D}$ and satisfy Assumption~\ref{ass:boundary}. Let $\bigl(X(t)\bigr)_{t\ge 0}$ denote the solution to the stochastic differential equation~\eqref{eq:sdeNagumog} with initial value $X(0)=x_0\in\mathcal D$. Given the final time $T\in(0,\infty)$ and the time-step size $\dt=T/N$ with $N\in\N$, let $\bigl(X_n\bigr)_{n=0,\ldots,N}$ be the solution to the domain preserving scheme~\eqref{eq:schemeX}.

Assume that the integrator $\Phi$ satisfies Assumptions~\ref{ass:integratorDP} and~\ref{ass:integrator}, and that there exists $\beta\in(0,\infty)$ such that $\|\Phi\|_{3,\beta}<\infty$ (see the definition~\eqref{eq:regPhi} of $\|\Phi\|_{p,\beta}$).

For all $T\in(0,\infty)$, there exists $C(T)\in(0,\infty)$ such that for all $\dt=T/N$ with $N\in\N$, one has
\begin{equation}\label{eq:strong}
\sup_{0\leq n\leq N}\left( \E\left[ \|X_n-X(t_n)\|^2 \right]\right)^{\frac12}\le C(T)\dt^{\frac12}.
\end{equation}
\end{theorem}

The proof of Theorem~\ref{theo:strong} is postponed to Section~\ref{sec:proof1}.

\begin{remark}
Theorem~\ref{theo:strong} may be generalized as follows: for all $p\in\N$ and all $T\in(0,\infty)$, there exists $C_p(T)\in(0,\infty)$ such that for all $\dt=T/N$ with $N\in\N$, one has
\[
\sup_{0\leq n\leq N}\left( \E\left[ \|X_n-X(t_n)\|^p \right]\right)^{\frac1p}\le C(T)\dt^{\frac12}.
\]
In this article, we consider only the case $p=2$, i.\,e. mean-square error estimates, to simplify the presentation.
\end{remark}

Figures~\ref{fig:plot1d} and~\ref{fig:plot2d} are illustrations of Theorem~\ref{theo:strong} in various situations and show that the order $1/2$ is optimal. The mean-square error is defined as the left-hand side of~\eqref{eq:strong}. In the numerical experiments, the expectation is estimated by Monte Carlo averaging over $10^3$ independent realizations, which has been verified to be sufficient in practice to have a negligible statistical error for the observation of the rate of convergence. The final time is $T=1$. The time-step size $\dt$ takes values ranging from $2^{-4}$ to $2^{-14}$. Moreover, the reference solution $X^{\text{ref}}$ is computed using the scheme~\sI with time-step size $\dt_{\text{ref}}=2^{-14}$. In these plots, a reference slope indicating the order of convergence $1/2$ with respect to the time-step size $\dt$ is also given.

In the numerical experiment reported in Figure~\ref{fig:plot1d}, the dimension is $d=1$ and the initial value is $x_0=0.5$. In Figure~\ref{fig:plot1da} (left-hand side), there is no drift ($f^0(x)=0$), and one has $M=1$ with the diffusion coefficient $f^1(x)=x^2$. In Figure~\ref{fig:plot1db} (right-hand side), a drift is added such that $f^0(x)=x$, and one has $M=2$, with diffusion coefficients such that $f^1(x)=x^2$ and $f^2(x)=\exp(x)$.

\begin{figure}[h]
\centering
\begin{subfigure}{.5\textwidth}
\includegraphics[width=\textwidth]{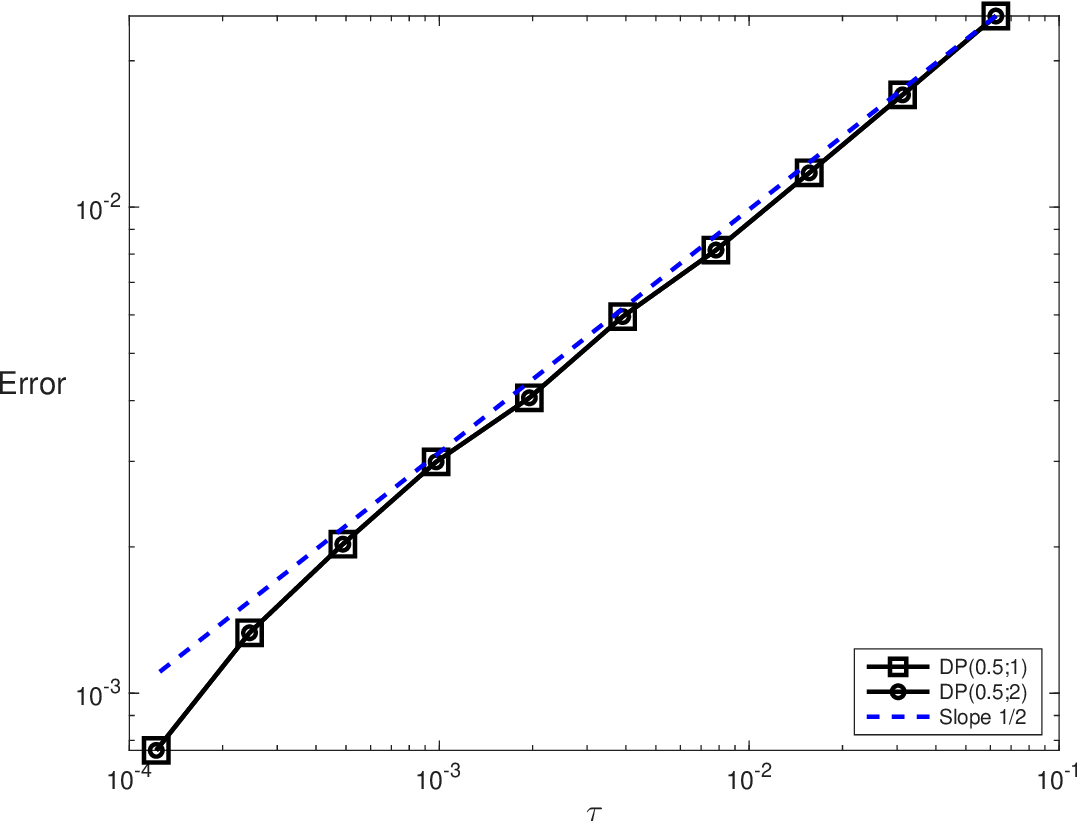} 
\caption{Case with no drift and $M=1$.
%\\ $f^1(x)=x^2$
}
\label{fig:plot1da}
\end{subfigure}%
\begin{subfigure}{.5\textwidth}
\includegraphics[width=\textwidth]{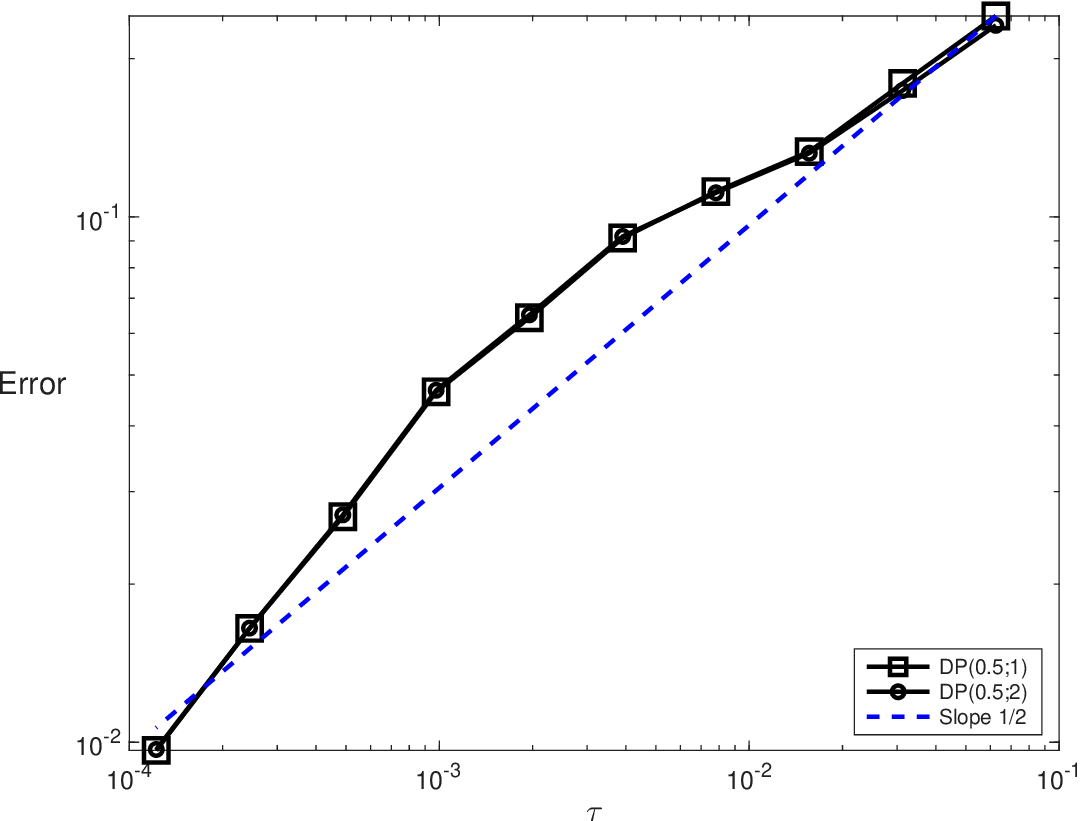}
\caption{Case with drift and $M=2$.
%\\ $f^0(x)=x$, $f^1(x)=x^2$, $f^2(x)=e^x$.
}
\label{fig:plot1db}
\end{subfigure}
\caption{Illustration of Theorem~\ref{theo:strong}. Behavior of the mean-square error for the domain preserving schemes ~\sI~and~\sII~in dimension $d=1$.}
\label{fig:plot1d}
\end{figure}

In the numerical experiment reported in Figure~\ref{fig:plot2d}, the dimension is $d=2$ and the initial value is $x_0=(0.5,-0.5)$. In Figure~\ref{fig:plot2da} (left-hand side), there is no drift ($f^0(x)=0$), and one has $M=1$, with diffusion coefficient such that $f^1(x_1,x_2)=\bigl(x_1^2x_2,x_1x_2^2\bigr)$. In Figure~\ref{fig:plot2db} (right-hand side), a drift is added such that $f^0(x_1,x_2)=\bigl(x_1x_2,x_1x_2^2)$, and one has $M=2$, with diffusion coefficients such that $f^1(x_1,x_2)=\bigl(x_1^2x_2,x_1x_2^2\bigr)$ and $f^2(x_1,x_2)=\bigl(\exp(x_1+x_2),x_1+x_2\bigr)$.

\begin{figure}[h]
\centering
\begin{subfigure}{.5\textwidth}
\includegraphics[width=\textwidth]{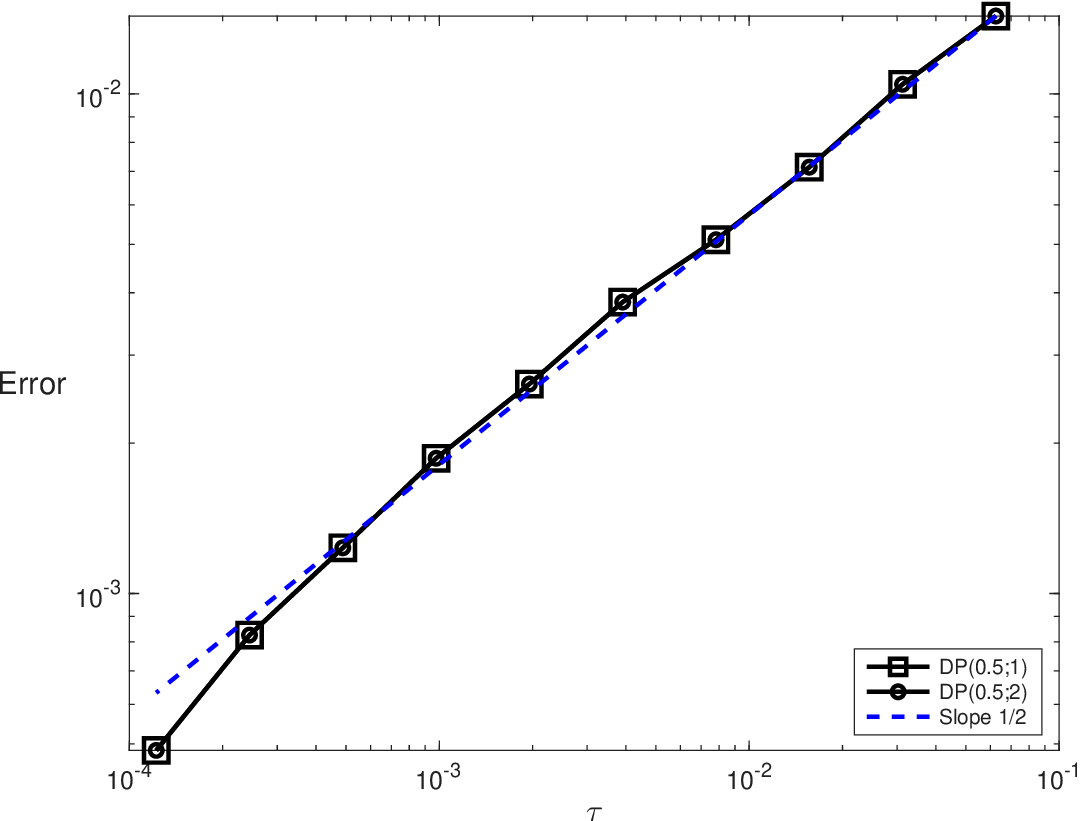} 
\caption{Case with no drift and $M=1$.
%\\ $f^1(x_1,x_2)=\bigl(x_1^2x_2,x_1x_2^2\bigr)$
}
\label{fig:plot2da}
\end{subfigure}%
\begin{subfigure}{.5\textwidth}
\includegraphics[width=\textwidth]{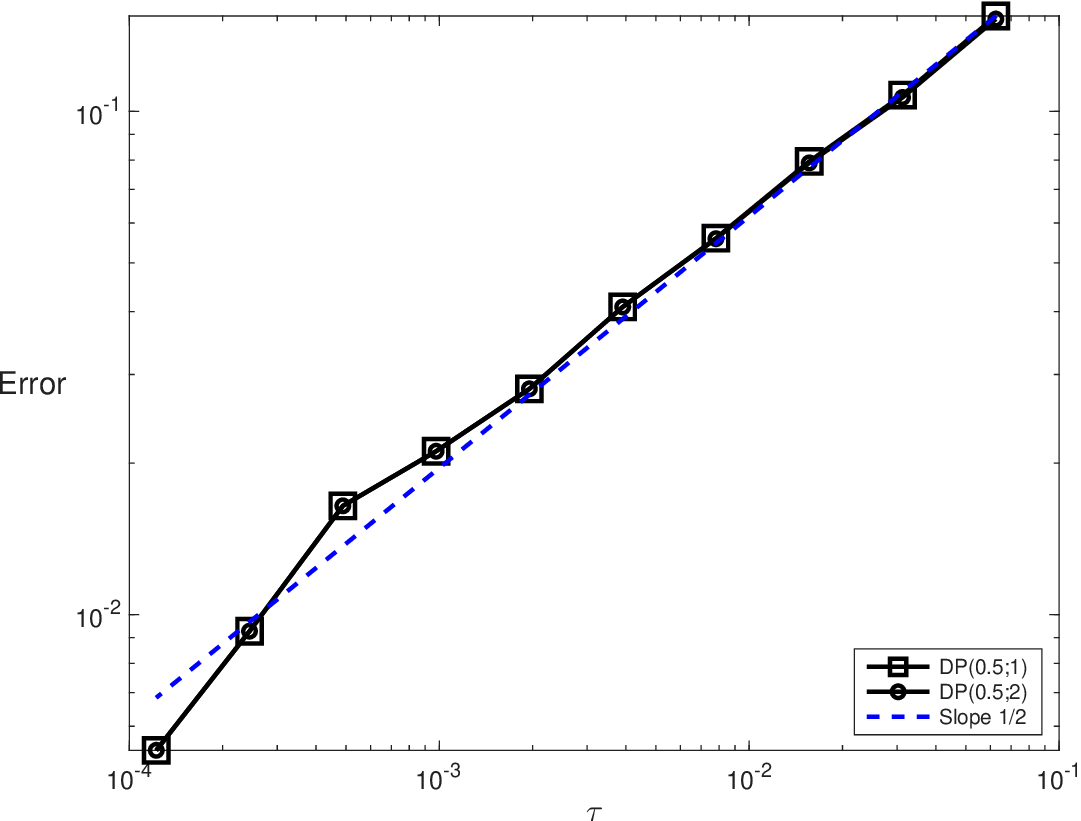}
\caption{Case with drift and $M=2$.
%\\ $f^0(x_1,x_2)=\bigl(x_1x_2,x_1x_2^2)$, $f^1(x_1,x_2)=\bigl(x_1^2x_2,x_1x_2^2\bigr)$, $f^2(x_1,x_2)=\bigl(e^{x_1+x_2},x_1+x_2\bigr)$
}
\label{fig:plot2db}
\end{subfigure}
\caption{Illustration of Theorem~\ref{theo:strong}. Behavior of the mean-square error for the domain preserving schemes ~\sI~and~\sII~in dimension $d=2$.}
\label{fig:plot2d}
\end{figure}

Figures~\ref{fig:plot1d} and~\ref{fig:plot2d} validate the mean-square convergence result with order $1/2$ obtained in Theorem~\ref{theo:strong}.

\subsection{Weak error estimates}\label{sec:main-weak}

This section is devoted to the second convergence result concerning the domain preserving scheme~\eqref{eq:schemeX}: it converges in the weak sense, with order $1$. Error estimates are stated in Theorem~\ref{theo:weak}. The result is then illustrated with numerical experiments.

\begin{theorem}\label{theo:weak}
Assume that the vector fields $g^m$ for $m\in\{0,1,\ldots,M\}$ are of class $\mathcal{C}^3$ on $\mathcal{D}$ and satisfy Assumption~\ref{ass:boundary}. Let $\bigl(X(t)\bigr)_{t\ge 0}$ denote the solution to the stochastic differential equation~\eqref{eq:sdeNagumog} with initial value $X(0)=x_0\in\mathcal D$. Given the final time $T\in(0,\infty)$ and the time-step size $\dt=T/N$ with $N\in\N$, let $\bigl(X_n\bigr)_{n=0,\ldots,N}$ be the solution to the domain preserving scheme~\eqref{eq:schemeX}.

Assume that the integrator $\Phi$ satisfies Assumptions~\ref{ass:integratorDP} and~\ref{ass:integrator}, and that there exists $\beta\in(0,\infty)$ such that $\|\Phi\|_{4,\beta}<\infty$ (see the definition~\eqref{eq:regPhi} of $\|\Phi\|_{p,\beta}$).

For all $T\in(0,\infty)$ and any mapping $\theta\colon \mathcal{D}\to\R$ of class $\mathcal{C}^3$, there exists $C(T,\theta)\in(0,\infty)$ such that for all $\dt=T/N$ with $N\in\N$, one has
\begin{equation}\label{eq:weak}
|\E\left[ \theta(X_N) \right]-\E\left[ \theta(X(T))\right]|\leq C(T,\theta)\dt.
\end{equation}
\end{theorem}

The proof of Theorem~\ref{theo:weak} is postponed to Section~\ref{sec:proof2}.

Figures~\ref{fig:plotw1d} and~\ref{fig:plotw2d} are illustrations of Theorem~\ref{theo:weak} in various situations, and for various choices of the test function $\theta$. They show that the order $1$ of the domain preserving scheme~\eqref{eq:schemeX} is optimal. The weak error is defined as the left-hand side of~\eqref{eq:weak}. In the numerical experiments, the expectation is estimated by Monte Carlo averaging over $10^6$ independent realizations, which has been verified to be sufficient in practice to have a negligible statistical error for the observation of the rate of convergence. The final time is $T=1$. The time-step size $\dt$ takes values ranging from $2^{-4}$ to $2^{-12}$. Moreover, the reference solution $X^{\text{ref}}$ is computed using the schemes~\sI~or~\sII~with time-step size $\dt_{\text{ref}}=2^{-12}$. Reference slopes indicating the orders of convergence $1/2$ and $1$ with respect to the time-step size $\dt$ are also provided in these figures.

In the numerical experiment reported in Figure~\ref{fig:plotw1d}, the dimension is $d=1$ and the initial value is $x_0=0.5$. In Figure~\ref{fig:plotw1da} (left-hand side), the scheme~\sI~is considered, i.\,e. the scheme~\eqref{eq:schemeX} with the integrator~\eqref{integratorA} from Example~\ref{exA}. In Figure~\ref{fig:plotw1db} (right-hand side), the scheme~\sII~is considered, i.\,e. the scheme~\eqref{eq:schemeX} with the integrator~\eqref{integratorB} from Example~\ref{exB}. One has $M=2$ and the vector fields are given such that for all $x\in[-1,1]$ one has
\[
f^0(x)=x,\quad f^1(x)=x^2,\quad f^2(x)=\exp(x).
\]
The test function $\theta$ is given by one of the examples below: for all $x\in[-1,1]$ one has
\[
\theta_1(x)=x^2,\quad \theta_2(x)=x,\quad \theta_3(x)=\exp(x),\quad \theta_4(x)=\cos(\pi x),\quad \theta_5(x)=\sin(\pi x).
\]

\begin{figure}[h]
\centering
\begin{subfigure}{.5\textwidth}
\includegraphics[width=\textwidth]{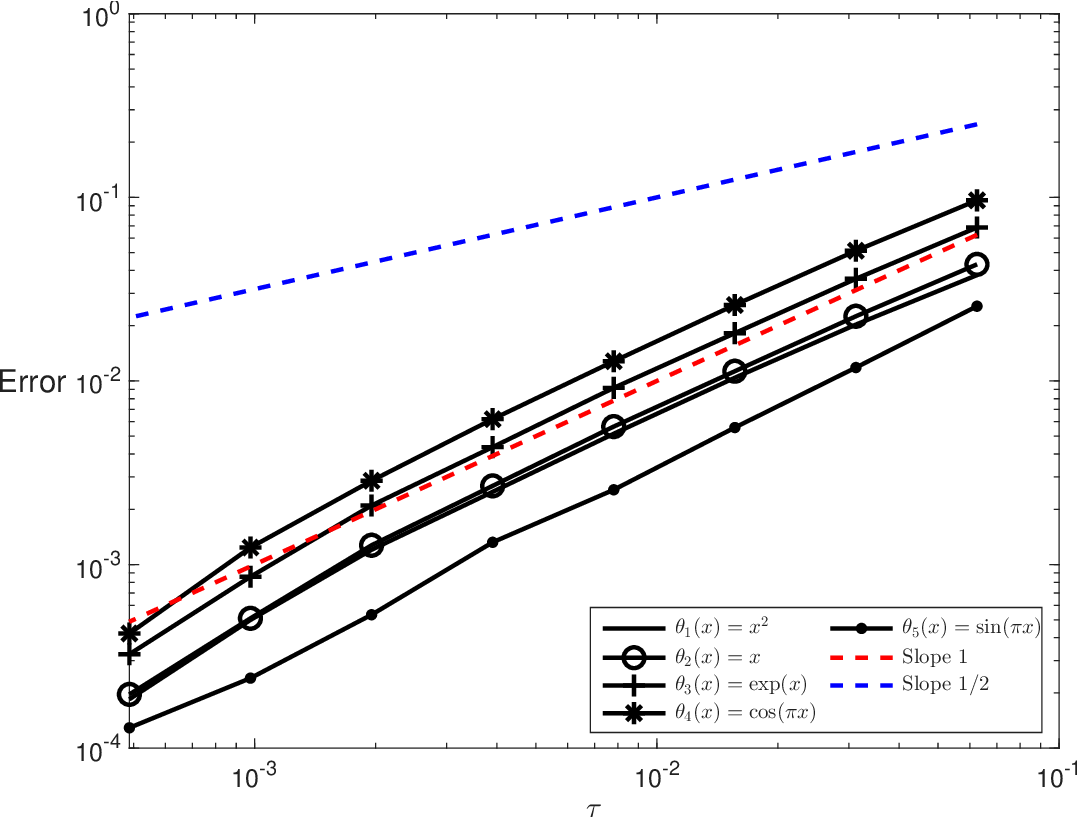} 
\caption{Scheme~\sI.}
\label{fig:plotw1da}
\end{subfigure}%
\begin{subfigure}{.5\textwidth}
\includegraphics[width=\textwidth]{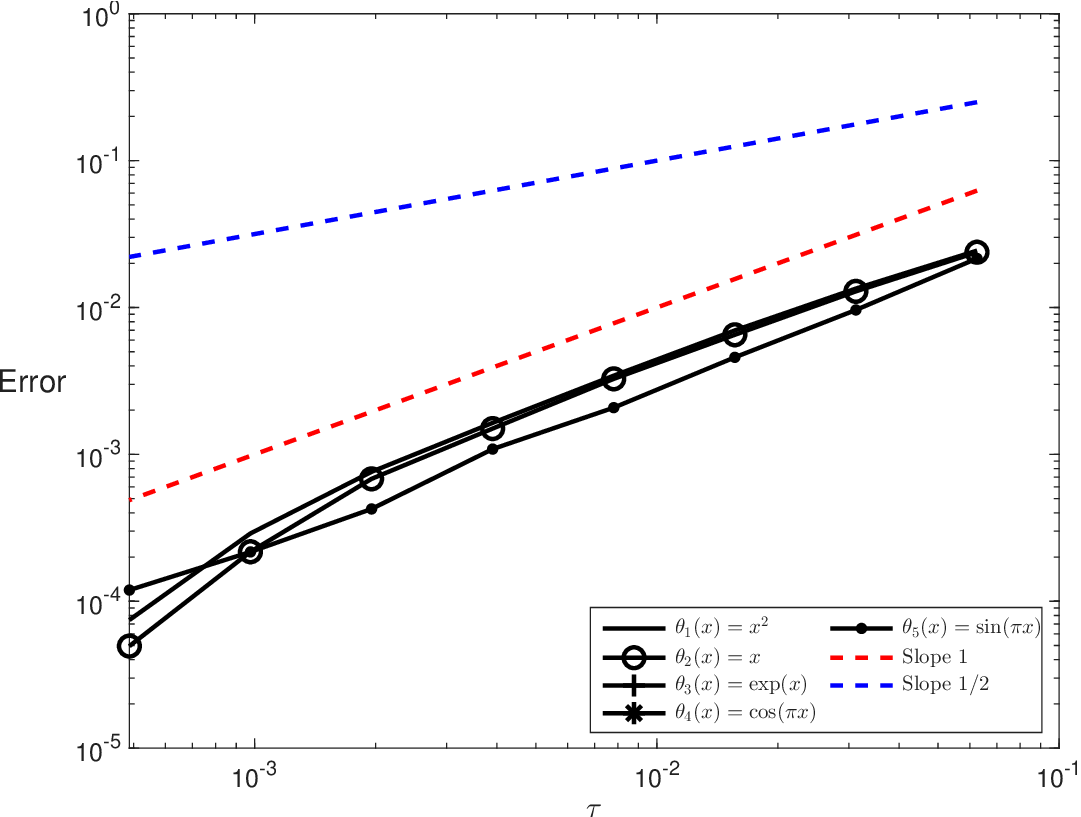}
\caption{Scheme~\sII.}
\label{fig:plotw1db}
\end{subfigure}
\caption{Illustration of Theorem~\ref{theo:weak}. Behavior of the weak error with different functions $\theta$ for the domain preserving schemes~\sI~and~\sII~ in dimension $d=1$.}
\label{fig:plotw1d}
\end{figure}

In the numerical experiment reported in Figure~\ref{fig:plotw2d}, the dimension is $d=2$ and the initial value is $x_0=(0.5,-0.5)$. In Figure~\ref{fig:plotw2da} (left-hand side), the scheme~\sI~is considered, i.\,e. the scheme~\eqref{eq:schemeX} with the integrator~\eqref{integratorA} from Example~\ref{exA}. In Figure~\ref{fig:plotw2db} (right-hand side), the scheme~\sII~is considered, i.\,e. the scheme~\eqref{eq:schemeX} with the integrator~\eqref{integratorB} from Example~\ref{exB}. One has $M=2$ and the vector fields are given such that for all $x=(x_1,x_2)\in[-1,1]^2$ one has
\[
f^0(x)=\bigl(x_1x_2,\cos(x_1)\sin(x_2)\bigr),\quad f^1(x)=\bigl(x_1^2x_2,x_1x_2^2\bigr),\quad f^2(x)=\bigl(\exp(x_1+x_2),x_1+x_2\bigr).
\]
The function $\theta$ is given by one of the examples below: for all $x=(x_1,x_2)\in[-1,1]^2$ one has
\[
\theta_1(x)=x_1+x_2^2,\quad \theta_2(x)=\exp(x_1)+x_2,\quad \theta_3(x)=\cos(x_1)+\sin(x_2),\quad \theta_4(x)=x_1+x_2.
\]

\begin{figure}[h]
\centering
\begin{subfigure}{.5\textwidth}
\includegraphics[width=\textwidth]{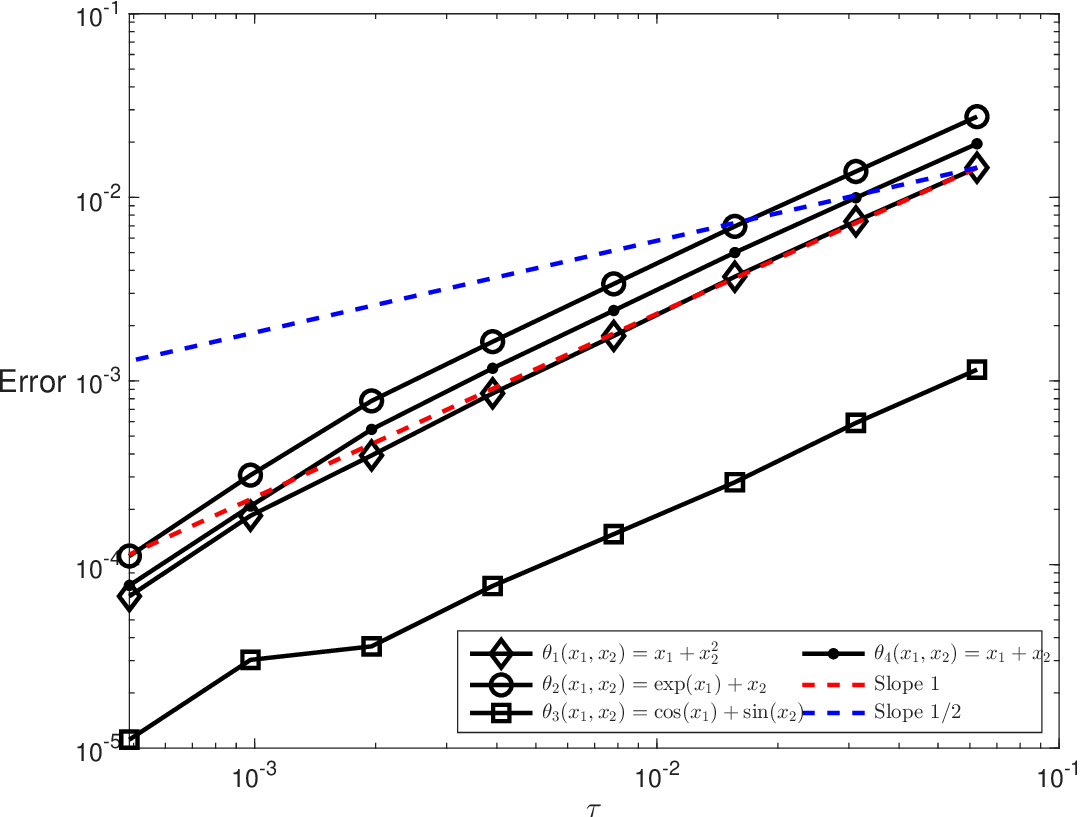} 
\caption{Scheme \sI.}
\label{fig:plotw2da}
\end{subfigure}%
\begin{subfigure}{.5\textwidth}
\includegraphics[width=\textwidth]{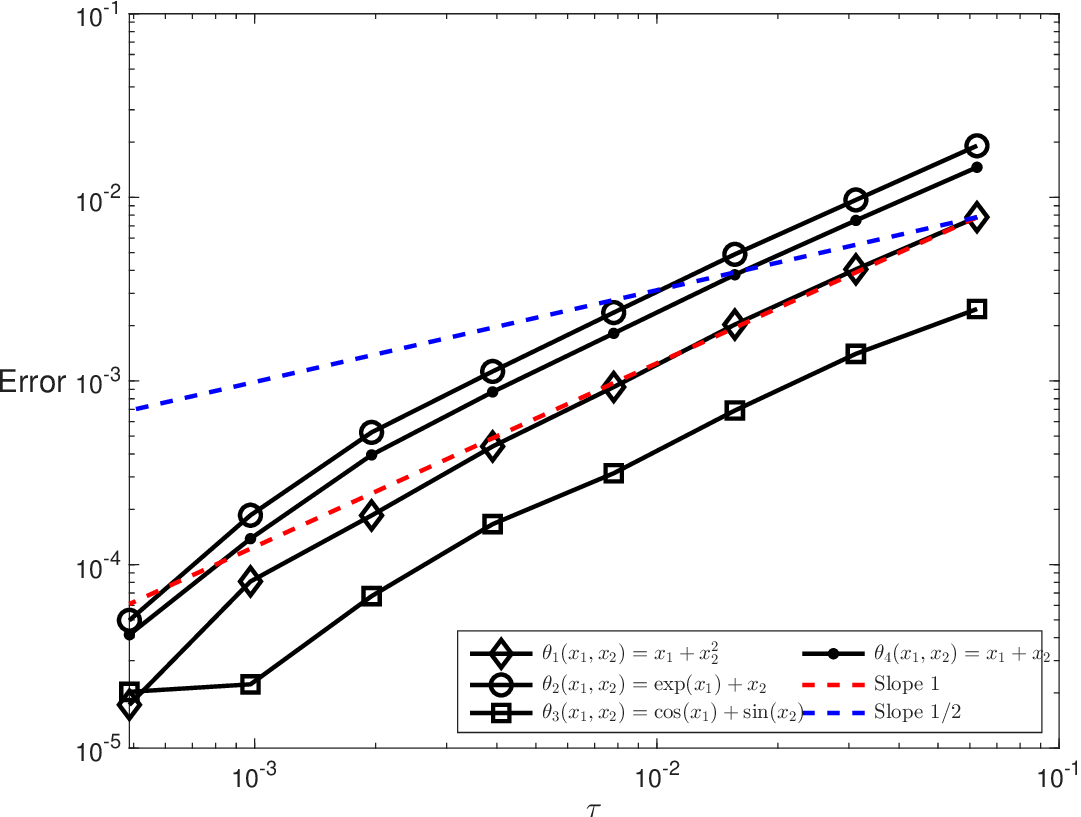}
\caption{Scheme \sII.}
\label{fig:plotw2db}
\end{subfigure}
\caption{Illustration of Theorem~\ref{theo:weak}. Behavior of the weak error with different functions $\theta$ for the domain preserving schemes~\sI~and~\sII~ in dimension $d=2$.}
\label{fig:plotw2d}
\end{figure}

Figures~\ref{fig:plotw1d} and~\ref{fig:plotw2d} validate the weak convergence result with order $1$ obtained in Theorem~\ref{theo:weak}.

\subsection{Higher-order strong error estimates ($\mathbf{M=1}$)}\label{sec:main-strong1}

This section is devoted to showing that when $M=1$ the domain preserving scheme~\eqref{eq:schemeX1} converges with strong order $1$, whereas for the scheme~\eqref{eq:schemeX} the strong order of convergence is $1/2$ as stated in Theorem~\ref{theo:strong} and observed in Figure~\ref{fig:plot1d}. Error estimates are stated in Theorem~\ref{theo:stronghigherorder}. The result is then illustrated with numerical experiments.

\begin{theorem}\label{theo:stronghigherorder}
Assume that $M=1$, and that the vector fields $g^0$ and $g=g^1$ are of class $\mathcal{C}^3$ on $\mathcal{D}$ and satisfy Assumption~\ref{ass:boundary}. Let $\bigl(X(t)\bigr)_{t\ge 0}$ denote the solution to the stochastic differential equation~\eqref{eq:sdeNagumog} with initial value $X(0)=x_0\in\mathcal D$ and $M=1$. Given the final time $T\in(0,\infty)$ and the time-step size $\dt=T/N$ with $N\in\N$, let $\bigl(X_n\bigr)_{n=0,\ldots,N}$ be the solution to the domain preserving scheme~\eqref{eq:schemeX1} with $M=1$.

Assume that the integrator $\Phi$ satisfies Assumptions~\ref{ass:integratorDP} and~\ref{ass:integrator-higher}, and that there exists $\beta\in(0,\infty)$ such that $\|\Phi\|_{4,\beta}<\infty$ (see the definition~\eqref{eq:regPhi} of $\|\Phi\|_{4,\beta}$).

There exists $\overline{\dt}\in(0,1)$ such that the following holds. For all $T\in(0,\infty)$, there exists $C(T)\in(0,\infty)$ such that for all $\dt=T/N$ with $N\in\N$ with $\dt\in(0,\overline{\dt})$, one has
\begin{equation}\label{eq:stronghigherorder}
\sup_{0\leq n\leq N} \left( \E \left[ \|X_n-X(t_n)\|^2 \right] \right)^{1/2}\leq C(T)\dt.
\end{equation}
\end{theorem}

\begin{remark}
Similar to Theorem~\ref{theo:strong},  Theorem~\ref{theo:stronghigherorder} may be generalized as follows: for all $p\in\N$ and all $T\in(0,\infty)$, there exists $C_p(T)\in(0,\infty)$ such that for all $\dt=T/N$ with $N\in\N$, one has
\[
\sup_{0\leq n\leq N}\left( \E\left[ \|X_n-X(t_n)\|^p \right]\right)^{\frac1p}\le C(T)\dt.
\]
In this article, we consider only the case $p=2$, i.\,e. mean-square error estimates, to simplify the presentation.
\end{remark}

Figure~\ref{fig:plotOrder1} is an illustration of Theorem~\ref{theo:stronghigherorder}. The mean-square error is defined as the left-hand side of~\eqref{eq:stronghigherorder}. In the numerical experiments, the expectation is estimated by Monte Carlo averaging over $10^3$ independent realizations, which has been verified to be sufficient in practice to have a negligible statistical error for the observation of the rate of convergence. The final time is $T=1$. The time-step size $\dt$ takes values ranging from $2^{-4}$ to $2^{-14}$. Moreover, the reference solution $X^{\text{ref}}$ is computed with time-step size $\dt_{\text{ref}}=2^{-14}$. Reference slopes indicating the orders of convergence $1/2$ and $1$ with respect to the time-step size $\dt$ are also provided. The schemes~\hosI~and~\hosII~are considered, i.\,e. the schemes~\eqref{eq:schemeX1} with the integrator~\eqref{integratorA} from Example~\ref{exA} and the integrator~\eqref{integratorB} from Example~\ref{exB} respectively. For comparison with Theorem~\ref{theo:strong}, the schemes~\sI~and~\sII~are also considered.

In the numerical experiment reported in Figure~\ref{fig:plotOrder1a} (left-hand side), the dimension is $d=1$ and the initial value is $x_0=0.5$. The vector fields are given such that $f^0(x)=x$ and $f^1(x)=x^2$ for all $x\in[-1,1]$. In the numerical experiment reported in Figure~\ref{fig:plotOrder1b} (right-hand side), the dimension is $d=2$ and the initial value is $x_0=(0.5,-0.5)$. The vector fields are given such that $f^0(x_1,x_2)=\bigl(x_1x_2\cos(x_1)\sin(x_2)\bigr)$ and $f^1(x_1,x_2)=x_1^2x_2,x_1x_2^2\bigr)$ for all $x=(x_1,x_2)\in[-1,1]^2$.

\begin{figure}[h]
\centering
\begin{subfigure}{.5\textwidth}
\includegraphics[width=\textwidth]{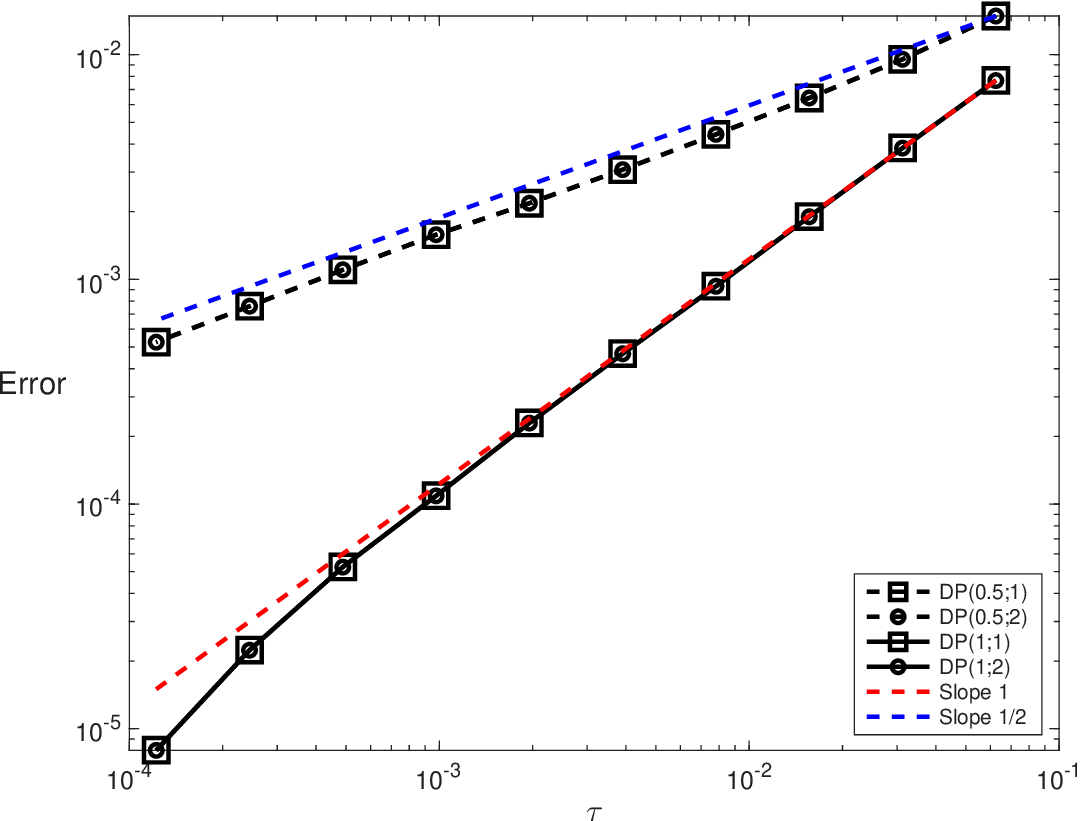} 
\caption{Dimension $d=1$.}
\label{fig:plotOrder1a}
\end{subfigure}%
\begin{subfigure}{.5\textwidth}
\includegraphics[width=\textwidth]{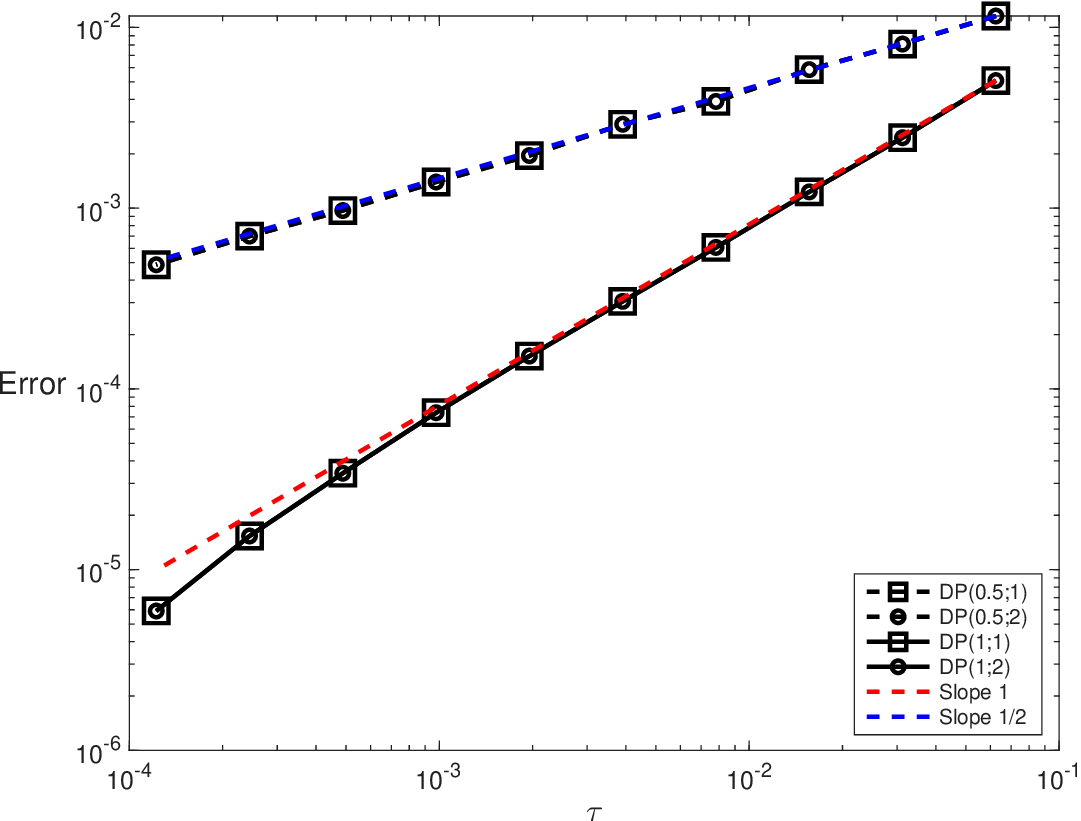}
\caption{Dimension $d=2$.}
\label{fig:plotOrder1b}
\end{subfigure}
\caption{Illustration of Theorem~\ref{theo:stronghigherorder}. Behavior of the mean-square error when $M=1$ for the domain preserving schemes~\hosI~and~\hosII, and comparison with the domain preserving schemes~\sI~and~\sII.}
\label{fig:plotOrder1}
\end{figure}

Figure~\ref{fig:plotOrder1} validates the mean-square convergence result with order $1$ obtained in Theorem~\ref{theo:stronghigherorder}.

\section{Domain preserving schemes for Stratonovich stochastic differential equations}\label{sec:Strato}

The objective of this section is to describe how to adapt the construction of explicit domain preserving schemes for stochastic differential equations where the noise is interpreted in the Stratonovich sense. In this section, we consider
\begin{equation}\label{eq:sdeNagumogStrato}
\left\lbrace
\begin{aligned}
\dd X(t)&=g^0(X(t))\dd t+\sum_{m=1}^Mg^m(X(t))\circ\dd B^m(t)\quad t\ge 0,\\
X(0)&=x_0.
\end{aligned}
\right.
\end{equation}
The equivalent It\^o formulation of the Stratonovich stochastic differential equation~\eqref{eq:sdeNagumogStrato} is given by
\begin{equation}\label{eq:sdeNagumogStratoIto}
\left\lbrace
\begin{aligned}
\dd X(t)&=\overline{g}^0(X(t))\dd t+\sum_{m=1}^Mg^m(X(t))\dd B^m(t)\quad t\ge 0,\\
X(0)&=x_0,
\end{aligned}
\right.
\end{equation}
with the vector field $\overline{g}^0\colon\mathcal{D}\to \R^d$ given by
\[
\overline{g}^0(x)=g^0(x)+\frac12\sum_{m=1}^{M}\bigl(g^m\bigr)'(x)g^m(x),\qquad \forall~x\in\mathcal{D},
\]
where $\bigl(g^m\bigr)'(x)$ denotes the Jacobian matrix of the vector field $g^m$ at $x\in\mathcal{D}$. The components $\overline{g}_1^0,\ldots,\overline{g}_d^0$ of the vector field $\overline{g}^0$ are expressed as follows: for all $k\in\{1,\ldots,d\}$ and all $x\in\mathcal{D}$, one has
\[
\overline{g}_k^0(x)=g_k^0(x)+\frac12\sum_{m=1}^{M}\sum_{\ell=1}^{d}\partial_{x_\ell}g_k^m(x)g_\ell^m(x)=g_k^0(x)+\frac12\sum_{m=1}^{M}\bigl(g^m\cdot\nabla\bigr)g^m(x).
\]
In order to ensure that the vector field $\overline{g}^0$ is of class $\mathcal{C}^2$ on $\mathcal{D}$, it is assumed that $g^0$ is of class $\mathcal{C}^2$ on $\mathcal{D}$ and that $g^1,\ldots,g^M$ are of class $\mathcal{C}^3$ on $\mathcal{D}$.

Observe that even if the vector fields $g^0,g^1,\ldots,g^M$ satisfy Assumption~\ref{ass:boundary}, the property is not verified for the vector field $\overline{g}^0$. In this section, the following stronger condition 
is thus imposed.
\begin{assumption}\label{ass:boundaryStrato}
For all $m\in\{0,1,\ldots,M\}$, one has
\begin{equation*}
g^m(x)=0,\qquad \forall~x\in\mathcal{D}.
\end{equation*}
\end{assumption}
Under Assumption~\ref{ass:boundaryStrato}, the vector fields $\overline{g}^0,g^1,\ldots,g^M$ satisfy Assumption~\ref{ass:boundary}, the It\^o stochastic differential equation~\eqref{eq:sdeNagumogStratoIto}, and thus its equivalent Stratonovich formulation~\eqref{eq:sdeNagumogStrato}, admits a unique solution $\bigl(X(t)\bigr)_{t\ge 0}$, which takes values in the domain $\mathcal{D}$ almost surely.

Sections~\ref{sec:variantschemeStrato} and~\ref{sec:variantscheme1Strato} below present variants of the schemes~\eqref{eq:schemeX} and of the scheme~\eqref{eq:schemeX1} (if $M=1$) when applied to the Stratonovich stochastic differential equation~\eqref{eq:sdeNagumogStrato}. The proposed domain preserving schemes are constructed as the schemes~\eqref{eq:schemeX} and~\eqref{eq:schemeX1} applied to the equivalent It\^o formulation~\eqref{eq:sdeNagumogStratoIto}. Convergence results are stated without proofs since they directly follow from those given in Section~\ref{sec:main}. Some numerical experiments are also reported to illustrate the behavior of the domain preserving schemes. Before proceeding, a last ingredient is necessary.

For all $m\in\{1,\ldots,M\}$ and all $k\in\{1,\ldots,d\}$, recall the expression~\eqref{eq:sigmaF} of $g_k^m$, where the mapping $\sigma$ is given by~\eqref{eq:sigma} and the function $f_k^m$ is given by~\eqref{eq:fmk} and is obtained from the application of Lemma~\ref{lem:g2f}. For the vector field $\overline{g}^0$, a similar property is obtained, due to Assumption~\ref{ass:boundaryStrato}. For all $k,\ell\in\{1,\ldots,d\}$, applying Lemma~\ref{lem:g2f} to the mapping $y\colon[-1,1]\mapsto g_\ell^m(x_1,\ldots,x_{k-1},y,x_{k+1},\ldots,x_d)$ for arbitrary $(x_1,\ldots,x_{k-1},x_{k+1},\ldots,x_d)\in[-1,1]^{d-1}$, there exists a function $f_{k,\ell}^m\colon\mathcal{D}\to\R$ such that one has
\begin{equation}\label{eq:sigmaFkl}
g_\ell^m(x)=\sigma(x_k)f_{k,\ell}^m(x),\qquad \forall~x\in\mathcal{D}, 
\end{equation}
where, owing to the proof of Lemma~\ref{lem:g2f}, the function $f_{k,\ell}^m$ is given by
\begin{equation}\label{eq:fmkl}
f_{k,\ell}^m(x)=\int_0^1\int_0^1 \partial_{x_k}^2g_\ell^m\bigl(x_1,\ldots,x_{k-1},\zeta x_k+(1-\zeta)\eta+\eta-1,x_{k+1},\ldots,x_d\bigr)\eta\dd \zeta\dd \eta.
\end{equation}
As the vector fields $g^1,\ldots,g^M$ are of class $\mathcal{C}^3$, thus of class $\mathcal{C}^2$, on $\mathcal{D}$, the mappings $f_{k,\ell}^m$ are continuous and bounded on $\mathcal{D}$.

For all $k\in\{1,\ldots,d\}$ and all $x\in\mathcal{D}$, define
\begin{equation}\label{eq:defbarfk}
\overline{f}_k^0(x)=f_k^0(x)+\frac12\sum_{m=1}^{M}\sum_{\ell=1}^{d}\partial_{x_\ell}g_{k}^m(x)f_{k,\ell}^m(x),
\end{equation}
and note that by construction one obtains the required expression for the components of the vector field $\overline{g}^0$: for all $k\in\{1,\ldots,d\}$ one has
\[
\overline{g}_k^0(x)=\sigma(x_k)\overline{f}_k^0(x),\qquad \forall~x\in\mathcal{D}.
\]

\subsection{Variant of the numerical scheme~\eqref{eq:schemeX} applied to~\eqref{eq:sdeNagumogStrato} }\label{sec:variantschemeStrato}

Applying the domain preserving scheme~\eqref{eq:schemeX} to the equivalent It\^o formulation~\eqref{eq:sdeNagumogStratoIto} of the Stratonovich stochastic differential equation~\eqref{eq:sdeNagumogStrato}, one obtains the following class of domain preserving numerical schemes for~\eqref{eq:sdeNagumogStrato}: given the initial value $X_0=x_0$, for all $k\in\{1,\ldots,d\}$, for all $n\in\{0,\ldots,N-1\}$, set
\begin{equation}\label{eq:schemeXStrato}
\begin{aligned}
X_{n+1,k}&=\Phi\left(\sum_{m=1}^Mf^m_k(X_n)^2\dt,\sum_{m=1}^Mf^m_k(X_n)\delta B^m_{n}+\overline{f}^0_k(X_n)\dt,X_{n,k}\right)\\
&=\Phi\left(\|f_k(X_n)\|^2\dt,f_k(X_n)\cdot\delta B_{n}+\overline{f}^0_k(X_n)\dt,X_{n,k}\right),
\end{aligned}
\end{equation}
where the auxiliary mappings $\overline{f}_k^0$ are defined by~\eqref{eq:defbarfk}.
The integrator $\Phi$ is assumed to satisfy Assumption~\ref{ass:integratorDP} to ensure that the scheme~\eqref{eq:schemeXStrato} is domain preserving: if $X_0\in\mathcal{D}$, then almost surely one has $X_n=\bigl(X_{n,1},\ldots,X_{n,d}\bigr)\in\mathcal{D}$ for all $n\in\{1,\ldots,N\}$. This is a straightforward corollary of Proposition~\ref{propo:dpX}.

As in Section~\ref{sec:schemes}, let us exemplify the domain preserving numerical scheme~\eqref{eq:schemeXStrato} for a simplified version of the stochastic differential equation~\eqref{eq:sdeNagumogStrato}, with no drift, in dimension $d=1$ and with a one-dimensional Brownian motion. Considering
\begin{equation}\label{eq:SDEStratosimple}
\dd X(t)=g(X(t))\circ\dd B(t),
\end{equation}
choosing the integrator~\eqref{integratorA} from Example~\ref{exA}, one obtains the domain preserving scheme
\begin{equation}\label{eq:schemeintAStrato}
X_{n+1}=\varphi\left(f(X_n)\delta B_n+\frac12(g'f)(X_n)\tau,\phi(f(X_n)^2\dt,X_n)\right),\quad \forall~n\in\{1,\ldots,N\},
\end{equation}
whereas choosing the integrator~\eqref{integratorB} from Example~\ref{exB}, and taking the identities $g=f\sigma$ and $g'=f'\sigma+f\sigma'$ into account, one obtains the domain preserving scheme
\begin{equation}\label{eq:schemeintBStrato}
X_{n+1}=\varphi\left(f(X_n)\delta B_n+\frac12\bigl(f'g\bigr)(X_n)\dt,X_n\right),\quad \forall~n\in\{1,\ldots,N\}.
\end{equation}

Concerning consistency of the scheme~\eqref{eq:schemeXStrato} and strong and weak convergence rates, one has the two results below, which follow from applying Theorem~\ref{theo:strong} and~\ref{theo:weak} to the scheme~\eqref{eq:schemeXStrato} for the equivalent It\^o formulation~\eqref{eq:sdeNagumogStratoIto} of the stochastic differential equation~\eqref{eq:sdeNagumogStrato}.

\begin{corollary}\label{cor:strongStrato}
Assume that the vector field $g^0$ is of class $\mathcal{C}^2$ on $\mathcal{D}$, that the vector fields $g^m$ for $m\in\{1,\ldots,M\}$ are of class $\mathcal{C}^3$ on $\mathcal{D}$, and that Assumption~\ref{ass:boundaryStrato} is satisfied. Let $\bigl(X(t)\bigr)_{t\ge 0}$ denote the solution to the stochastic differential equation~\eqref{eq:sdeNagumogStrato} with initial value $X(0)=x_0\in\mathcal D$. Given the final time $T\in(0,\infty)$ and the time-step size $\dt=T/N$ with $N\in\N$, let $\bigl(X_n\bigr)_{n=0,\ldots,N}$ be the solution to the domain preserving scheme~\eqref{eq:schemeXStrato}.

Assume that the integrator $\Phi$ satisfies Assumptions~\ref{ass:integratorDP} and~\ref{ass:integrator}, and that there exists $\beta\in(0,\infty)$ such that $\|\Phi\|_{3,\beta}<\infty$ (see the definition~\eqref{eq:regPhi} of $\|\Phi\|_{p,\beta}$).

For all $T\in(0,\infty)$, there exists $C(T)\in(0,\infty)$ such that for all $\dt=T/N$ with $N\in\N$, one has
\begin{equation}\label{eq:strongStrato}
\sup_{0\leq n\leq N}\left( \E\left[ \|X_n-X(t_n)\|^2 \right]\right)^{\frac12}\le C(T)\dt^{\frac12}.
\end{equation}
\end{corollary}

\begin{corollary}\label{cor:weakStrato}
Assume that the vector field $g^0$ is of class $\mathcal{C}^3$ on $\mathcal{D}$, that the vector fields $g^m$ for $m\in\{1,\ldots,M\}$ are of class $\mathcal{C}^4$ on $\mathcal{D}$, and that Assumption~\ref{ass:boundaryStrato} is satisfied. Let $\bigl(X(t)\bigr)_{t\ge 0}$ denote the solution to the stochastic differential equation~\eqref{eq:sdeNagumogStrato} with initial value $X(0)=x_0\in\mathcal D$. Given the final time $T\in(0,\infty)$ and the time-step size $\dt=T/N$ with $N\in\N$, let $\bigl(X_n\bigr)_{n=0,\ldots,N}$ be the solution to the domain preserving scheme~\eqref{eq:schemeXStrato}.

Assume that the integrator $\Phi$ satisfies Assumptions~\ref{ass:integratorDP} and~\ref{ass:integrator}, and that there exists $\beta\in(0,\infty)$ such that $\|\Phi\|_{4,\beta}<\infty$ (see the definition~\eqref{eq:regPhi} of $\|\Phi\|_{p,\beta}$).

For all $T\in(0,\infty)$ and any mapping $\theta\colon \mathcal{D}\to\R$ of class $\mathcal{C}^3$, there exists $C(T,\theta)\in(0,\infty)$ such that for all $\dt=T/N$ with $N\in\N$, one has
\begin{equation}\label{eq:weakStrato}
|\E\left[ \theta(X_N) \right]-\E\left[ \theta(X(T))\right]|\leq C(T,\theta)\dt.
\end{equation}
\end{corollary}
We conclude this section with a remark explaining the need to consider the equivalent It\^o formulation~\eqref{eq:sdeNagumogStratoIto} of the Stratonovich stochastic differential equation~\eqref{eq:sdeNagumogStrato} when designing consistent domain preserving numerical schemes.
\begin{remark}\label{rem:badschemeStrato}
Writing $g=f\sigma$, and, on each time interval $[t_n,t_{n+1}]$, freezing the evolution of $f$ and solving exactly the Stratonovich stochastic differential equation
\[
\dd Y(t)=\sigma(Y(t))\circ \dd B(t)
\]
provides the scheme
\begin{equation}\label{eq:badschemeStrato}
X_{n+1}=\varphi\left(f(X_n)\delta B_n,X_n\right),\quad \forall~n\in\{0,\ldots,N-1\}.
\end{equation}
The expression~\eqref{eq:schemeintBStrato} of the scheme~\eqref{eq:schemeX} obtained for the simplified Stratonovich stochastic differential equation~\eqref{eq:SDEStratosimple} suggests that the scheme~\eqref{eq:badschemeStrato} is not consistent in general (if $f$ is not constant), and that a correction term $\frac12\sigma f'f$ is needed. This is illustrated in numerical experiments below. This justifies the need to consider the equivalent It\^o formulation~\eqref{eq:sdeNagumogStratoIto} of the Stratonovich stochastic differential equation~\eqref{eq:sdeNagumogStrato} for the construction of consistent domain preserving numerical schemes~\eqref{eq:schemeXStrato}.
\end{remark}

\subsection{Variant of the numerical scheme~\eqref{eq:schemeX1} applied to~\eqref{eq:sdeNagumogStrato} ($\mathbf{M=1}$)}\label{sec:variantscheme1Strato}

When $M=1$, applying the domain preserving scheme~\eqref{eq:schemeX1} to the It\^o formulation~\eqref{eq:sdeNagumogStratoIto} of the Stratonovich stochastic differential equation~\eqref{eq:sdeNagumogStrato} provides the following class of domain preserving schemes for~\eqref{eq:sdeNagumogStrato}:  given the initial value $X_0=x_0$, for all $k\in\{1,\ldots,d\}$, for all $n\in\{0,\ldots,N-1\}$, set
\begin{equation}\label{eq:schemeX1Strato}
X_{n+1,k}=\Phi\left(f_k(X_n)^2\dt,f_k(X_n)\delta B_n+\frac12 (g\cdot\nabla)f_k(X_n)(\delta B_n^2-\dt)+\overline{f}^0_k(X_n)\dt,X_{n,k}\right),
\end{equation}
where the auxiliary mappings $\overline{f}_k^0$ are defined by~\eqref{eq:defbarfk}. As for the scheme~\eqref{eq:schemeXStrato}, the integrator $\Phi$ is assumed to satisfy Assumption~\ref{ass:integratorDP} to ensure that the scheme~\eqref{eq:schemeX1Strato} is domain preserving: if $X_0\in\mathcal{D}$, then almost surely one has $X_n=\bigl(X_{n,1},\ldots,X_{n,d}\bigr)\in\mathcal{D}$ for all $n\in\{1,\ldots,N\}$. This is a straightforward corollary of Proposition~\ref{propo:dpX}.

In addition, let us exemplify the domain preserving numerical scheme~\eqref{eq:schemeXStrato} for the simplified version of the stochastic differential equation~\eqref{eq:sdeNagumogStrato}, with no drift, in dimension $d=1$ and with a one-dimensional Brownian motion. Choosing the integrator~\eqref{integratorA} from Example~\ref{exA}, one obtains the domain preserving scheme
\begin{equation}\label{eq:scheme1intAStrato}
X_{n+1}=\varphi\left(f(X_n)\delta B_n+\frac12(f'g)(X_n)(\delta B_n^2-\dt)+\frac12(g'f)(X_n)\tau,\phi(f(X_n)^2\dt,X_n)\right),\quad \forall~n\in\{1,\ldots,N\},
\end{equation}
whereas choosing the integrator~\eqref{integratorB} from Example~\ref{exB}, one obtains the domain preserving scheme
\begin{equation}\label{eq:scheme1intBStrato}
X_{n+1}=\varphi\left(f(X_n)\delta B_n+\frac12\bigl(f'g\bigr)\delta B_n^2,X_n\right),\quad \forall~n\in\{1,\ldots,N\}.
\end{equation}
Observe that the scheme~\eqref{eq:scheme1intBStrato} has a simple expression as
\[
X_{n+1}=\overline{\Phi}(\delta B_n,X_n),
\]
where the mapping $\overline{\Phi}:\R\times\mathcal{D}\to\R$ is defined as
\begin{equation}\label{eq:interpretorder2}
\overline{\Phi}(\gamma,x)=\varphi\left(f(x)\gamma+\frac12(f'g)\gamma^2,x\right),\quad \forall~\gamma\in\R,~x\in\mathcal{D}.
\end{equation}
It can be checked that $\overline{\Phi}$ is a domain preserving second-order integrator for the ordinary differential equation
\[
\dot{x}(t)=g(x(t))
\]
with vector field $g=f\sigma$.

By standard techniques, it can be shown that a scheme of type~\eqref{eq:interpretorder2} converges to the solution of the Stratonovich stochastic differential equation~\eqref{eq:SDEStratosimple} with strong order $1$ for such second-order integrators $\overline{\Phi}$ applied to the associated ordinary differential equations. Note also that the first-order scheme~\eqref{eq:scheme1intBStrato} is obtained from the scheme~\eqref{eq:schemeintBStrato} when $\dt$ is replaced by $\delta B_n^2$.

The expression of the general version~\eqref{eq:schemeX1Strato} of the scheme is designed as a natural generalization of the scheme~\eqref{eq:scheme1intBStrato}. Then the scheme~\eqref{eq:schemeX1} for the It\^o stochastic differential equation~\eqref{eq:sdeNagumog} can be retrieved from the application of the scheme~\eqref{eq:schemeX1Strato} to its equivalent Stratonovich version. The details of the computation are omitted.

Finally, let us state a convergence result for the scheme~\eqref{eq:schemeX1Strato} with strong order $1$. The result below follows from applying Theorem~\ref{theo:stronghigherorder} to the scheme~\eqref{eq:schemeX1Strato} for the equivalent It\^o formulation~\eqref{eq:sdeNagumogStratoIto} of~\eqref{eq:sdeNagumogStrato}.

\begin{corollary}\label{cor:stronghigherorderStrato}
Assume that $M=1$, that the vector field $g^0$ is of class $\mathcal{C}^3$, that the vector field $g=g^1$ is of class $\mathcal{C}^4$ on $\mathcal{D}$ and that Assumption~\ref{ass:boundaryStrato} is satisfied. Let $\bigl(X(t)\bigr)_{t\ge 0}$ denote the solution to the stochastic differential equation~\eqref{eq:sdeNagumogStrato} with initial value $X(0)=x_0\in\mathcal D$ and $M=1$. Given the final time $T\in(0,\infty)$ and the time-step size $\dt=T/N$ with $N\in\N$, let $\bigl(X_n\bigr)_{n=0,\ldots,N}$ be the solution to the domain preserving scheme~\eqref{eq:schemeX1Strato} with $M=1$.

Assume that the integrator $\Phi$ satisfies Assumptions~\ref{ass:integratorDP} and~\ref{ass:integrator-higher}, and that there exists $\beta\in(0,\infty)$ such that $\|\Phi\|_{4,\beta}<\infty$ (see the definition~\eqref{eq:regPhi} of $\|\Phi\|_{4,\beta}$).

There exists $\overline{\dt}\in(0,1)$ such that the following holds. For all $T\in(0,\infty)$, there exists $C(T)\in(0,\infty)$ such that for all $\dt=T/N$ with $N\in\N$ with $\dt\in(0,\overline{\dt})$, one has
\begin{equation}\label{eq:stronghigherorderStrato}
\sup_{0\leq n\leq N} \left( \E \left[ \|X_n-X(t_n)\|^2 \right] \right)^{1/2}\leq C(T)\dt.
\end{equation}
\end{corollary}

\subsection{Numerical experiment}

Finally, let us present numerical illustrations. The experiments reported in Section~\ref{sec:main} are not repeated, since one can transform any Stratonovich stochastic differential equation into its equivalent It\^o formulation. The only objective of this numerical experiment is to illustrate in the simple case of equation~\eqref{eq:SDEStratosimple} (no drift, $d=1$, $M=1$) the convergence of the domain preserving  schemes~\eqref{eq:schemeXStrato} and~\eqref{eq:schemeX1Strato}, and the fact that the scheme~\eqref{eq:badschemeStrato} proposed in Remark~\ref{rem:badschemeStrato} is not consistent.

Figure~\ref{fig:plotStratoa} is an illustration of Corollaries~\ref{cor:strongStrato}~and~\ref{cor:stronghigherorderStrato}, and of the content of Remark~\ref{rem:badschemeStrato}, with $d=1$, $M=1$ and $f^1(x)=x^2$. The expectation is estimated by Monte Carlo averaging over $10^3$ independent realizations, which has been verified to be sufficient in practice to have a negligible statistical error for the observation of the rate of convergence. The final time is $T=1$. The time-step size $\dt$ takes values ranging from $2^{-4}$ to $2^{-14}$. Moreover, the reference solution $X^{\text{ref}}$ is computed using the scheme~\eqref{eq:scheme1intAStrato} with time-step size $\dt_{\text{ref}}=2^{-14}$. In this plot, reference slopes indicating the order of convergence $1/2$, resp. $1$ with respect to the time-step size $\dt$ are also given. As expected one observes that the scheme~\eqref{eq:badschemeStrato} is not consistent with the Stratonovich stochastic differential equation~\eqref{eq:sdeNagumogStrato}, and one observes the strong rates of convergence $1/2$ and $1$.

Figure~\ref{fig:plotStratob} is an illustration of Corollary~\ref{cor:weakStrato} for the same problem as above. 
The function $\theta$ is given by one of the examples below: 
for all $x\in[-1,1]$ one has
\[
\theta_1(x)=x^2,\quad \theta_2(x)=\cos(x).
\]
The time-step size $\dt$ takes values ranging from $2^{-4}$ to $2^{-12}$. Moreover, the reference solution $X^{\text{ref}}$ is computed using the scheme~\eqref{eq:scheme1intAStrato} with time-step size $\dt_{\text{ref}}=2^{-12}$. In this plot, reference slopes indicating the order of convergence $1/2$, resp. $1$ with respect to the time-step size $\dt$ are also given. In this figure, one observes the weak rates of convergence $1$ 
for the schemes~\eqref{eq:schemeintAStrato}~and~\eqref{eq:schemeintBStrato}.

\begin{figure}[h]
\centering
\begin{subfigure}{.5\textwidth}
\includegraphics[width=\textwidth]{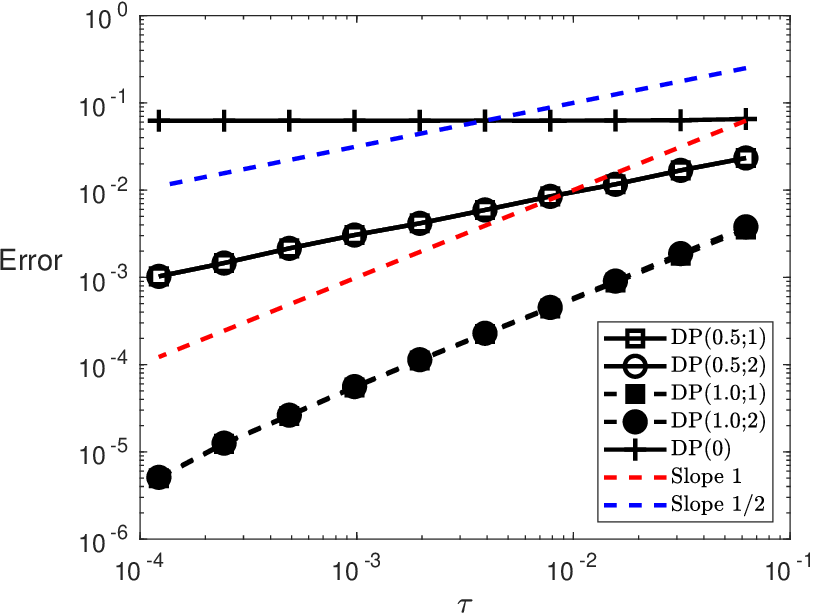} 
\caption{Mean-square error.}
\label{fig:plotStratoa}
\end{subfigure}%
\begin{subfigure}{.5\textwidth}
\includegraphics[width=\textwidth]{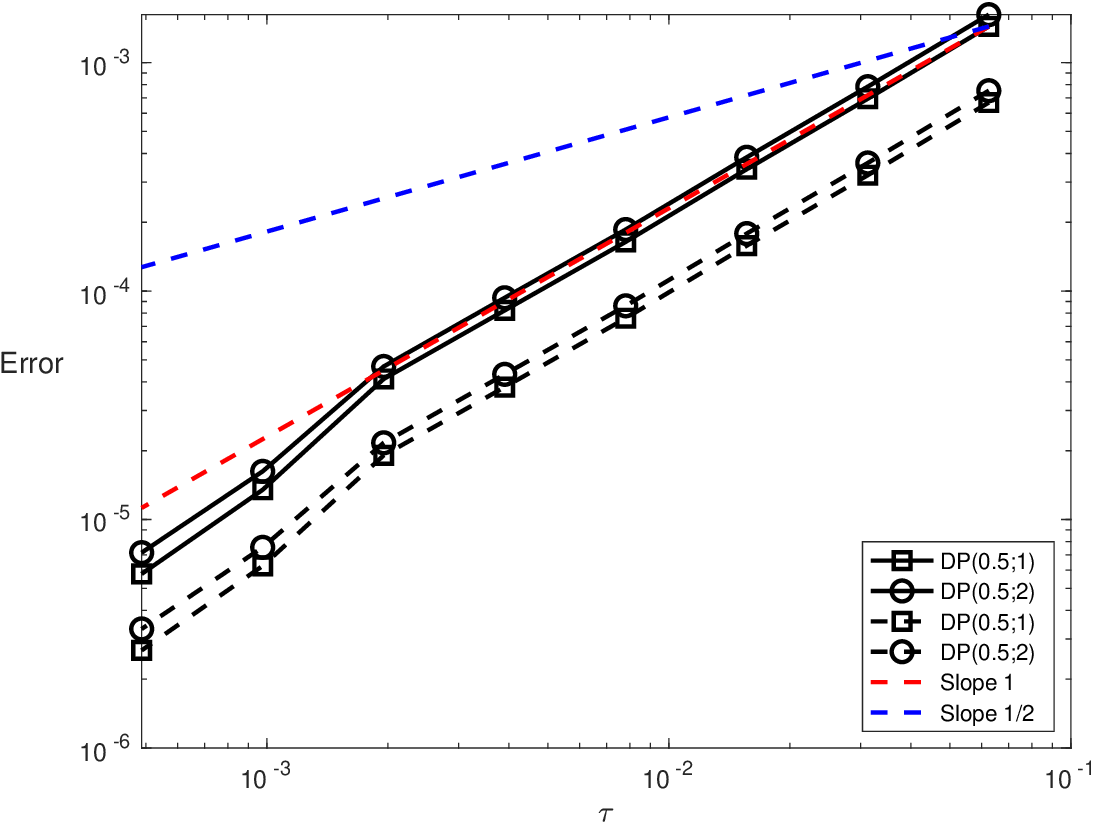} 
\caption{Weak error.% for the test functions $\theta_1(x)=x^2$ (solid lines) and $\theta_2(x)=\cos(x)$ (dashed lines)
}
\label{fig:plotStratob}
\end{subfigure}%
\caption{Illustration of Corollaries~\ref{cor:strongStrato},~\ref{cor:stronghigherorderStrato}~and~\ref{cor:weakStrato} for domain preserving schemes applied to Stratonovich stochastic differential equations. Behavior of the mean-square error (Figure~\ref{fig:plotStratoa}) and weak error (Figure~\ref{fig:plotStratob}) for the domain preserving schemes~\eqref{eq:schemeintAStrato}~and~\eqref{eq:schemeintBStrato} (denoted~\sI~and~\sII), the domain preserving schemes~\eqref{eq:scheme1intAStrato}~and~\eqref{eq:scheme1intBStrato} (denoted~\hosI~and~\hosII), and the not consistent scheme~\eqref{eq:badschemeStrato} (denoted by \upshape DP(0)).}
\label{fig:plotStrato}
\end{figure}

\section{Proof of Theorem~\ref{theo:strong}}\label{sec:proof1}

The objective of this section is to prove Theorem~\ref{theo:strong}, i.\,e. that the domain preserving scheme~\eqref{eq:schemeX} converges with strong rate of convergence equal to $1/2$ when the mapping $\Phi$ satisfies the conditions given in the statement of Theorem~\ref{theo:strong}. Specifically, in the following it is assumed that Assumptions~\ref{ass:integratorDP} and~\ref{ass:integrator} are satisfied, and that there exists $\beta\in(0,\infty)$ such that $\|\Phi\|_{3,\beta}<\infty$. As a result, recalling the definition~\eqref{eq:psi} of the mapping $\psi$, there exists $C\in(0,\infty)$ such that for all $s\ge 0$, all $\gamma\in\R$ and $y\in[-1,1]$ one has
\begin{align}
|\partial_\gamma\Phi(s,\gamma,y)|+|\psi(y)|\le Ce^{C(s+|\gamma|)},\label{eq:boundPhipsi}\\
|\partial_s\partial_\gamma\Phi(s,\gamma,y)|+|\partial_\gamma^2\Phi(s,\gamma,y)|+|\partial_y\partial_\gamma\Phi(s,\gamma,y)|\le Ce^{C(s+|\gamma|)},\label{eq:LipPhi}\\
|\partial_s\psi(s,\gamma,y)|+|\partial_\gamma\psi(s,\gamma,y)|+|\partial_y\psi(s,\gamma,y)|\le Ce^{C(s+|\gamma|)}.\label{eq:Lippsi}
\end{align}

This section is organized as follows. In Section~\ref{sec:aux1}, we introduce and study the auxiliary processes $\widetilde{Z}_k$ and $\widetilde{X}_k$, for all $k\in\{1,\ldots,d\}$. Note that these two auxiliary processes are also employed in Section~\ref{sec:proof2} for the proof of Theorem~\ref{theo:weak} on the weak convergence of the numerical scheme. In Section~\ref{sec:error1}, we explain how the strong error is decomposed and we provide the proof of Theorem~\ref{theo:strong}.

\subsection{Auxiliary processes}\label{sec:aux1}

First, for all $k\in\{1,\ldots,d\}$, the auxiliary process $\widetilde{Z}_k$ is given by~\eqref{eq:Ztilde} below: for all $n\in\{0,\ldots,N-1\}$ and $t\in[t_n,t_{n+1})$, set
\begin{equation}\label{eq:Ztilde}
\widetilde{Z}_k(t)=\Bigl(\|f_k(X_n)\|^2(t-t_n),f_k^0(X_n)(t-t_n)+f_k(X_n)\cdot(B(t)-B(t_n)),X_{n,k}\Bigr).
\end{equation}
Note that one has $\widetilde{Z}_k(t)\in\R^+\times\R\times[-1,1]$ for all $t\in[0,T]$. Moreover, for all $n\in\{0,\ldots,N-1\}$, one has
$\widetilde{Z}_k(t_n)=\bigl(0,0,X_{n,k}\bigr)$. Observe that the process $\widetilde{Z}_k$ is discontinuous at the grid times $t_n$.

Next, for all $k\in\{1,\ldots,d\}$, the auxiliary process $\widetilde{X}_k$  is given by~\eqref{eq:Xtilde} below: for all $n\in\{0,\ldots,N-1\}$ and $t\in[t_n,t_{n+1})$, we set
\begin{equation}\label{eq:Xtilde}
\widetilde{X}_k(t)=\Phi\bigl(\widetilde{Z}_k(t)\bigr)=\Phi\Bigl(\|f_k(X_n)\|^2(t-t_n),f_k^0(X_n)(t-t_n)+f_k(X_n)\cdot(B(t)-B(t_n)),X_{n,k}\Bigr).
\end{equation}
Note that one has $\widetilde{X}_k(t)\in[-1,1]$ for all $t\in[0,T]$, owing to the condition~\eqref{eq:condPhiDP} from Assumption~\ref{ass:integrator} on the integrator $\Phi$. Moreover, for all $n\in\{0,\ldots,N-1\}$, one has $\widetilde{X}_k(t_n)=\Phi(\widetilde{Z}_k(t_n))=\Phi(0,0,X_{n,k})=X_{n,k}$. Finally, by the definition~\eqref{eq:schemeX} of the domain preserving scheme, for all $k\in\{1,\ldots,d\}$ the process $\widetilde{X}_k$ is continuous on $[0,T]$: for all $n\in\{0,\ldots,N-1\}$ one has $\underset{t\to t_{n+1}^{-}}\lim~\widetilde{X}_k(t)=X_{n+1,k}=\widetilde{X}_k(t_{n+1})$.

The following results, Lemma~\ref{lem:auxZtilde} and Lemma~\ref{lem:auxXtilde} below, provide properties of the auxiliary processes $\widetilde{Z}_k$ and $\widetilde{X}_k$ respectively.

\begin{lemma}\label{lem:auxZtilde}
For all $p\in[1,\infty)$ and $T\in(0,\infty)$, there exists $C_p(T)\in(0,\infty)$ such that for all $\dt\in(0,1)$, one has the increment bounds
\begin{equation}\label{eq:auxZtilde1}
\underset{1\le k\le d}\sup~\underset{0\le n\le N-1}\sup~\underset{t\in[t_n,t_{n+1})}\sup~\E[\|\widetilde{Z}_k(t)-\widetilde{Z}_k(t_n)\|^{2p}]\le C_p(T)\dt^p,
\end{equation}
and the moment bounds
\begin{equation}\label{eq:auxZtilde2}
\underset{1\le k\le d}\sup~\underset{t\in[0,T]}\sup~\E[\|\widetilde{Z}_k(t)\|^{2p}]\leq C_p(T).
\end{equation}
\end{lemma}

\begin{proof}[Proof of Lemma~\ref{lem:auxZtilde}]

First, let us prove the inequality~\eqref{eq:auxZtilde1}. By definition~\eqref{eq:Ztilde} of the auxiliary process $\widetilde{Z}_k$, for all $n\in\{0,\ldots,N-1\}$, for all $k\in\{1,\ldots,d\}$, one has for all $t\in[t_n,t_{n+1})$
\[
\widetilde{Z}_k(t)-\widetilde{Z}_k(t_n)=\bigl(\|f_k(X_n)\|^2(t-t_n),f_k^0(X_n)(t-t_n)+f_k(X_n)\cdot(B(t)-B(t_n)),0\bigr).
\]
This gives the upper bound
\[
\|\widetilde{Z}_k(t)-\widetilde{Z}_k(t_n)\|^2\le \|f_k(X_n)\|^4(t-t_n)^2+2|f_k^0(X_n)|^2(t-t_n)^2+2|f_k(X_n)\cdot(B(t)-B(t_n))|^2.
\]
Recall that $X_n\in\mathcal{D}$ almost surely by Proposition~\ref{propo:dpX}. Furthermore, the mappings $f_k^0,f_k^1,\ldots,f_k^{M}$ are bounded on $\mathcal{D}$, owing to the definition~\eqref{eq:fmk} and since the vector fields $g^0,g^1,\ldots,g^M$ are assumed to be of class $\mathcal{C}^2$. Therefore, there exists $C_p(T)\in(0,\infty)$ such that one has the upper bound
\[
\E[\|\widetilde{Z}_k(t)-\widetilde{Z}_k(t_n)\|^{2p}]\le C_p(T)\Bigl(\bigl(t-t_n\bigr)^{2p}+\E[\|B(t)-B(t_n)\|^{2p}]\Bigr).
\]
Finally, for the increments of the Brownian motion, one has $\E[\|B(t)-B(t_n)\|^{2p}]=c_p(t-t_n)^p$ for some $c_p\in(0,\infty)$. One thus obtains the upper bound~\eqref{eq:auxZtilde1}.

It remains to prove the inequality~\eqref{eq:auxZtilde2}. Note that for all $t\in[t_n,t_{n+1})$, one has
\[
\widetilde{Z}_k(t)=\widetilde{Z}_k(t)-\widetilde{Z}_k(t_n)+\widetilde{Z}_k(t_n),
\]
with $\|\widetilde{Z}_k(t_n)\|=|X_{n,k}|\le 1$ since $X_n\in\mathcal{D}$ almost surely. It suffices then to apply the inequality~\eqref{eq:auxZtilde1} to obtain the upper bound~\eqref{eq:auxZtilde2}.

The proof of Lemma~\ref{lem:auxZtilde} is thus completed.
\end{proof}

In the proof of Lemma~\ref{lem:auxXtilde} below and in the convergence proofs, the following  auxiliary result is employed.
\begin{lemma}\label{lem:auxexpo}
Let $a\colon\mathcal{D}\to\R^+$ and $b\colon\mathcal{D}\to\R^M$ be continuous functions. There exists $C\in(0,\infty)$ such that for all $\dt\in(0,1)$ and all $n\in\{0,\ldots,N-1\}$ one has 
\begin{equation}\label{eq:lemauxexpo}
\underset{t\in[t_n,t_{n+1}]}\sup~\underset{x\in\mathcal{D}}\sup~\E\left[e^{a(x)(t-t_n)}e^{|b(x)\cdot(B(t)-B(t_n))|}\right]\le C.
\end{equation}
\end{lemma}

\begin{proof}[Proof of Lemma~\ref{lem:auxexpo}]
The mappings $a$ and $b$ are continuous on the compact domain $\mathcal{D}$, hence they are bounded on $\mathcal{D}$. Therefore, there exists $C\in(0,\infty)$ such that for all $x\in\mathcal{D}$, under the conditions $0\le t-t_n\le \dt\le 1$, one has
\[
\E[e^{a(x)(t-t_n)}e^{|b(x)\cdot(B(t)-B(t_n))|}]\le C\E[e^{C\|B(t)-B(t_n)\|}].
\]
If $Z_M$ denotes a standard centered $m$-dimensional Gaussian random variable, then $B(t)-B(t_n)$ is equal in distribution to $\sqrt{t-t_n}Z_M$. Thus under the conditions $0\le t-t_n\le \dt\le 1$, one has
\[
\E[e^{C\|B(t)-B(t_n)\|}]=\E[e^{C\sqrt{t-t_n}\|Z_M\|}]\le \E[e^{C\|Z_M\|}].
\]
Finally, one has $\E[e^{c\|Z_M\|}]<\infty$ for all $c\in(0,\infty)$. This concludes the proof of Lemma~\ref{lem:auxexpo}.
\end{proof}

Lemma~\ref{lem:auxXtilde} below provides stochastic differential equations~\eqref{eq:auxXtilde1} satisfied by the auxiliary processes $\widetilde{X}_k$ on the intervals $[t_n,t_{n+1})$. Moreover, it states increment bounds for this auxiliary process, see~\eqref{eq:auxXtilde2}.
\begin{lemma}\label{lem:auxXtilde}
For all $n\in\{0,\ldots,N-1\}$, for all $k\in\{1,\ldots,d\}$, the stochastic process $t\in[t_n,t_{n+1})\mapsto \widetilde{X}_k(t)$ is solution to the stochastic differential equation
\begin{equation}\label{eq:auxXtilde1}
\dd \widetilde{X}_k(t)=\partial_\gamma\Phi(\widetilde{Z}_k(t))f_k^0(X_n)\dd t+\partial_\gamma\Phi(\widetilde{Z}_k(t))f_k(X_n)\cdot\dd B(t)+\psi(\widetilde{Z}_k(t))\|f_k(X_n)\|^2\dd t.
\end{equation}

Furthermore, for all $T\in(0,\infty)$, there exists $C(T)\in(0,\infty)$ such that for all $\dt\in(0,1)$, one has the regularity property
\begin{equation}\label{eq:auxXtilde2}
\underset{1\le k\le d}\sup~\underset{0\le n\le N-1}\sup~\underset{t\in[t_n,t_{n+1})}\sup~\E[|\widetilde{X}_k(t)-X_{n,k}|^2]\leq C(T)\dt.
\end{equation}
\end{lemma}

\begin{proof}[Proof of Lemma~\ref{lem:auxXtilde}]

First, let us prove the identity~\eqref{eq:auxXtilde1}. For all $n\in\{0,\ldots,N-1\}$, for all $k\in\{1,\ldots,d\}$, applying the It\^o formula, for all $t\in[t_n,t_{n+1})$, one obtains the identity
\begin{align*}
\dd\widetilde{X}_k(t)&=\partial_s\Phi(\widetilde{Z}_k(t)) \|f_k(X_n)\|^2\dd t\\
&+\partial_\gamma\Phi(\widetilde{Z}_k(t)) f_k^0(X_n)\dd t\\
&+\partial_\gamma\Phi(\widetilde{Z}_k(t)) f_k(X_n)\cdot\dd B(t)+\frac12\partial_{\gamma}^2\Phi(\widetilde{Z}_k(t)) \|f_k(X_n)\|^2\dd t.
\end{align*}
Recalling the definition~\eqref{eq:psi} of the auxiliary mapping $\psi$, one then obtains~\eqref{eq:auxXtilde1}.

Next, let us prove the inequality~\eqref{eq:auxXtilde2}. For all $n\in\{0,\ldots,N-1\}$, for all $k\in\{1,\ldots,d\}$, owing to the SDE~\eqref{eq:auxXtilde1}, for all $t\in[t_n,t_{n+1})$ one gets
\begin{align*}
\widetilde{X}_k(t)-X_{n,k}&=\widetilde{X}_k(t)-\widetilde{X}_k(t_n)\\
&=\int_{t_n}^{t}\partial_\gamma\Phi(\widetilde{Z}_k(s))f_k^0(X_n)\dd s+\int_{t_n}^{t}\partial_\gamma\Phi(\widetilde{Z}_k(s))f_k(X_n)\cdot\dd B(s)+\int_{t_n}^{t}\psi(\widetilde{Z}_k(s))\|f_k(X_n)\|^2\dd s.
\end{align*}
Applying the Cauchy--Schwarz inequality and the It\^o isometry formula, one then obtains 
\begin{align*}
\E[|\widetilde{X}_k(t)-X_{n,k}|^2]&\le 3(t-t_n)\int_{t_n}^{t}\E[\partial_\gamma\Phi(\widetilde{Z}_k(s))^2f_k^0(X_n)^2]\dd s\\
&+3\int_{t_n}^{t}\E[\partial_\gamma\Phi(\widetilde{Z}_k(s))^2\|f_k(X_n)\|^2]\dd s\\
&+3(t-t_n)\int_{t_n}^{t}\E[\psi(\widetilde{Z}_k(s))^2\|f_k(X_n)\|^4]\dd s.
\end{align*}
Next, recall that $X_n\in\mathcal{D}$ almost surely owing to Proposition~\ref{propo:dpX}. Furthermore, the mappings $f_k^0,f_k^1,\ldots,f_k^M$ defined by~\eqref{eq:fmk} are bounded on $\mathcal{D}$. Therefore, there exists $C(T)\in(0,\infty)$ such that
\[
\E[|\widetilde{X}_k(t)-X_{n,k}|^2]\le C(T)\int_{t_n}^{t}\Bigl(\E[|\partial_\gamma\Phi(\widetilde{Z}_k(s))|^2]+\E[|\psi(\widetilde{Z}_k(s))|^2]\Bigr)\dd s.
\]

Applying the upper bound~\eqref{eq:boundPhipsi} on $\partial_\gamma\Phi$ and $\psi$ (following from regularity conditions imposed on $\Phi$), one then obtains the upper bound
\begin{align*}
\E[|\widetilde{X}_k(t)-X_{n,k}|^2]&\le C(T)\int_{t_n}^{t}\E[e^{C(h_k(X_n)(s-t_n)}e^{C|f_k(X_n)\cdot(B(s)-B(t_n))|}]\dd s,
\end{align*}
where $h_k$ is defined by~\eqref{eq:defhk}. Applying a conditional expectation argument and the inequality~\eqref{eq:lemauxexpo} from Lemma~\ref{lem:auxexpo}, one then obtains the upper bounds
\[
\E[|\widetilde{X}_k(t)-X_{n,k}|^2]\le C(T)(t-t_n)\le C(T)\dt,\quad \forall~t\in[t_n,t_{n+1}),
\]
which concludes the proof of the inequality~\eqref{eq:auxXtilde2}.

The proof of Lemma~\ref{lem:auxXtilde} is thus completed.
\end{proof}

\subsection{Error analysis}\label{sec:error1}

The first step of the proof of Theorem~\ref{theo:strong} is to provide a decomposition of the strong error of the numerical scheme~\eqref{eq:schemeX}: for all $n\in\{0,\ldots,N-1\}$, we set
\begin{align*}
e_n&=X_n-X(t_n),\\
e_{n,k}&=X_{n,k}-X_k(t_n),\qquad \forall~k\in\{1,\ldots,d\},
\end{align*}
and note that one has the identity
\[
\E[\|e_n\|^2]=\sum_{k=1}^{d}\E[|e_{n,k}|^2].
\]
The error terms $e_{n,k}$ are then decomposed as follows.

On the one hand, on each interval $[t_n,t_{n+1})$, the auxiliary process $\bigl(\widetilde{X}_k(t)\bigr)_{t\in[0,T]}$ is solution to the stochastic differential equation~\eqref{eq:auxXtilde1}. Recalling that $X_{n,k}=\widetilde{X}_k(t_n)$, for all $n\in\{1,\ldots,N\}$, one thus obtains the identity 
\begin{align*}
X_{n,k}-x_{0,k}&=\sum_{j=0}^{n-1}\bigl(X_{j+1,k}-X_{j,k}\bigr)\\
&=\sum_{j=0}^{n-1}\bigl(\widetilde{X}_k(t_{j+1})-\widetilde{X}_k(t_j)\bigr)\\
&=\sum_{j=0}^{n-1}\int_{t_j}^{t_{j+1}}\partial_\gamma\Phi(\widetilde{Z}_k(t))f_k^0(X_j)\dd t+\sum_{j=0}^{n-1}\int_{t_j}^{t_{j+1}}\partial_\gamma\Phi(\widetilde{Z}_k(t))f_k(X_j)\cdot\dd B(t)\\
&+\sum_{j=0}^{n-1}\int_{t_j}^{t_{j+1}}\psi(\widetilde{Z}_k(t))\|f_k(X_j)\|^2\dd t.
\end{align*}

On the other hand, for all $k\in\{1,\ldots,d\}$, the component $\bigl(X_k(t)\bigr)_{t\in[0,T]}$ is solution to the stochastic differential equation
\begin{align*}
\dd X_k(t)&=g_k^0(X(t))\dd t+g_k(X(t))\cdot \dd B(t)\\
&=\sigma(X_k(t))f_k^0(X(t))\dd t+\sigma(X_k(t))f_k(X(t))\cdot \dd B(t).
\end{align*}
Therefore, for all $n\in\{1,\ldots,N\}$, one obtains the identity 
\begin{align*}
X_k(t_n)-x_{0,k}&=\int_{0}^{t_n}g_k^0(X(t))\dd t+\int_{0}^{t_n}g_k(X(t))\dd B(t)\\
&=\sum_{j=0}^{n-1}\int_{t_j}^{t_{j+1}}g_k^0(X(t))\dd t+\sum_{j=0}^{n-1}\int_{t_j}^{t_{j+1}}g_k(X(t))\dd B(t).
\end{align*}

As a result, for all $n\in\{1,\ldots,N\}$ and $k\in\{1,\ldots,d\}$, one has the decomposition of the error
\begin{equation}\label{eq:decomp_e_n}
e_{n,k}=X_{n,k}-X_k(t_n)=e_{n,k}^{(1)}+e_{n,k}^{(2)}+e_{n,k}^{(3)},
\end{equation}
where the error terms $e_{n,k}^{(1)}$, $e_{n,k}^{(2)}$ and $e_{n,k}^{(3)}$ are defined as 
\begin{align}
e_{n,k}^{(1)}&=\sum_{j=0}^{n-1}\int_{t_j}^{t_{j+1}}\bigl[\partial_\gamma\Phi(\widetilde{Z}_k(t))f_k^0(\widetilde{X}(t_j))-g_k^0(X(t))\bigr]\dd t,\label{eq:e_n1}\\
e_{n,k}^{(2)}&=\sum_{j=0}^{n-1}\int_{t_j}^{t_{j+1}}\bigl[\partial_\gamma\Phi(\widetilde{Z}_k(t))f_k(\widetilde{X}(t_j))-g_k(X(t))\bigr]\cdot \dd B(t),\label{eq:e_n2}\\
e_{n,k}^{(3)}&=\sum_{j=0}^{n-1}\int_{t_j}^{t_{j+1}}\psi(\widetilde{Z}_k(t))\|f_k(\widetilde{X}(t_j))\|^2 \dd t.\label{eq:e_n3}
\end{align}
The error terms $e_{n,k}^{(1)}$ and $e_{n,k}^{(2)}$ are then decomposed as follows.

For the error term $e_{n,k}^{(1)}$, defined by~\eqref{eq:e_n1}, one has the decomposition
\begin{equation}\label{eq:decomp_e_n1}
e_{n,k}^{(1)}=e_{n,k}^{(1,1)}+e_{n,k}^{(1,2)}+e_{n,k}^{(1,3)},
\end{equation}
where the error terms $e_{n,k}^{(1,1)}$, $e_{n,k}^{(1,2)}$ and $e_{n,k}^{(1,3)}$ are defined as
\begin{align}
e_{n,k}^{(1,1)}&=\sum_{j=0}^{n-1}\int_{t_j}^{t_{j+1}}\bigl[g_k^0(X(t_j))-g_k^0(X(t))\bigr]\dd t,
\label{eq:e_n11}\\
e_{n,k}^{(1,2)}&=\sum_{j=0}^{n-1}\int_{t_j}^{t_{j+1}}\bigl[g_k^0(\widetilde{X}(t_j))-g_k^0(X(t_j))\bigr]\dd t,\label{eq:e_n12}\\
e_{n,k}^{(1,3)}&=\sum_{j=0}^{n-1}\int_{t_j}^{t_{j+1}}\bigl[\partial_\gamma\Phi(\widetilde{Z}_k(t))f_k^0(\widetilde{X}(t_j))-g_k^0(\widetilde{X}(t_j))\bigr]\dd t.\label{eq:e_n13}
\end{align}
Similarly, for the error term $e_{n,k}^{(2)}$, defined by~\eqref{eq:e_n2}, one has the decomposition
\begin{equation}\label{eq:decomp_e_n2}
e_{n,k}^{(2)}=e_{n,k}^{(2,1)}+e_{n,k}^{(2,2)}+e_{n,k}^{(2,3)},
\end{equation}
where the error terms $e_{n,k}^{(2,1)}$, $e_{n,k}^{(2,2)}$ and $e_{n,k}^{(2,3)}$ are defined as
\begin{align}
e_{n,k}^{(2,1)}&=\sum_{j=0}^{n-1}\int_{t_j}^{t_{j+1}}\bigl[g_k(X(t_j))-g_k(X(t))\bigr]\cdot \dd B(t),
\label{eq:e_n21}\\
e_{n,k}^{(2,2)}&=\sum_{j=0}^{n-1}\int_{t_j}^{t_{j+1}}\bigl[g_k(\widetilde{X}(t_j))-g_k(X(t_j))\bigr]\cdot \dd B(t),\label{eq:e_n22}\\
e_{n,k}^{(2,3)}&=\sum_{j=0}^{n-1}\int_{t_j}^{t_{j+1}}\bigl[\partial_\gamma\Phi(\widetilde{Z}_k(t))f_k(\widetilde{X}(t_j))-g_k(\widetilde{X}(t_j))\bigr]\cdot \dd B(t).\label{eq:e_n23}
\end{align}
It remains to provide upper bounds for the error terms appearing in the decompositions~\eqref{eq:decomp_e_n},~\eqref{eq:decomp_e_n1} and~\eqref{eq:decomp_e_n2}.

\begin{proof}[Proof of Theorem~\ref{theo:strong}]
The proof is divided into three parts, where the error terms are grouped together when similar techniques are employed.
\begin{itemize}
\item the treatment of the error terms $e_{n,k}^{(1,1)}$ and $e_{n,k}^{(2,1)}$
\item the treatment of the error terms $e_{n,k}^{(1,2)}$ and $e_{n,k}^{(2,2)}$
\item the treatment of the error terms $e_{n,k}^{(1,3)}$, $e_{n,k}^{(2,3)}$ and $e_{n,k}^{(3)}$.
\end{itemize}
The strong error estimates~\eqref{eq:strong} are finally established in the conclusion of the proof.

\paragraph{\bf Treatment of the error terms $e_{n,k}^{(1,1)}$ and $e_{n,k}^{(2,1)}$}
The mappings $g_k^0$ and $g_k$ are Lipschitz continuous on $\mathcal{D}$. Applying the Cauchy--Schwarz inequality and the It\^o isometry formula, there exists $C(T)\in(0,\infty)$ such that for all $n\in\{1,\ldots,N\}$ one has
\[
\E[|e_{n,k}^{(1,1)}|^2]+\E[|e_{n,k}^{(2,1)}|^2]\le C(T)\sum_{j=0}^{n-1}\int_{t_j}^{t_{j+1}}\E[\|X(t_j)-X(t)\|^2]\dd t.
\]
Applying the regularity property~\eqref{eq:reg_exact} from Proposition~\ref{propo:reg_exact}, one then obtains the upper bound
\begin{equation}\label{eq:enk11e_nk21}
\underset{0\le n\le N}\sup~\E[|e_{n,k}^{(1,1)}|^2]+\underset{0\le n\le N}\sup~\E[|e_{n,k}^{(2,1)}|^2]\le C(T)\dt.
\end{equation}

\paragraph{\bf Treatment of the error terms $e_{n,k}^{(1,2)}$ and $e_{n,k}^{(2,2)}$.}
The mappings $g_k^0$ and $g_k$ are Lipschitz continuous on $\mathcal{D}$. Applying the Cauchy--Schwarz inequality and the It\^o isometry formula, there exists $C(T)\in(0,\infty)$ such that for all $n\in\{1,\ldots,N\}$ one has
\begin{equation}\label{eq:enk12e_nk22}
\E[|e_{n,k}^{(1,2)}|^2]+\E[|e_{n,k}^{(2,2)}|^2]\le C(T)\dt \sum_{j=0}^{n-1}\E[\|\widetilde{X}(t_j)-X(t_j)\|^2]=C(T)\dt \sum_{j=0}^{n-1}\E[\|e_j\|^2].
\end{equation}

\paragraph{\bf Treatment of the error terms $e_{n,k}^{(1,3)}$, $e_{n,k}^{(2,3)}$ and $e_{n,k}^{(3)}$.}
Recall that $g_k^m(x)=\sigma(x_k)f_k^m(x)$ for all $x\in\mathcal{D}$, all $m\in\{0,\ldots,M\}$ and all $k\in\{1,\ldots,d\}$. Moreover, one has $\partial_\gamma\Phi(0,0,y)=\sigma(y)$ for $y\in[-1,1]$ owing to the condition~\eqref{eq:condPhi1} in Assumption~\ref{ass:integrator} on the integrator $\Phi$. As a result, the error terms $e_{n,k}^{(1,3)}$ and $e_{n,k}^{(2,3)}$ are written as
\begin{align*}
e_{n,k}^{(1,3)}&=\sum_{j=0}^{n-1}\int_{t_j}^{t_{j+1}}\bigl[\partial_\gamma\Phi(\widetilde{Z}_k(t))-\partial_\gamma\Phi(\widetilde{Z}_k(t_j))\bigr]f_k^0(\widetilde{X}(t_j))\dd t\\
e_{n,k}^{(2,3)}&=\sum_{j=0}^{n-1}\int_{t_j}^{t_{j+1}}\bigl[\partial_\gamma\Phi(\widetilde{Z}_k(t))-\partial_\gamma\Phi(\widetilde{Z}_k(t_j))\bigr]f_k(\widetilde{X}(t_j))\cdot \dd B(t).
\end{align*}
Finally, owing to the condition~\eqref{eq:condPhi2} from Assumption~\ref{ass:integrator} on the integrator $\Phi$, the mapping $\psi$ defined by~\eqref{eq:psi} satisfies the condition $\psi(0,0,y)=0$ for all $y\in[-1,1]$. Recall that $\widetilde{Z}_k(t_j)=(0,0,X_{j,k})$. As a result, the error term $e_{n,k}^{(3)}$ is written as
\[
e_{n,k}^{(3)}=\sum_{j=0}^{n-1}\int_{t_j}^{t_{j+1}}\bigl[\psi(\widetilde{Z}_k(t))-\psi(\widetilde{Z}_k(t_j))\bigr]\|f_k(\widetilde{X}(t_j))\|^2 \dd t.
\]
Recall that the mappings $f_k^0$ and $f_k$ are bounded on $\mathcal{D}$. Thus, applying the Cauchy--Schwarz inequality and the It\^o isometry formula, there exists $C(T)\in(0,\infty)$ such that for all $n\in\{1,\ldots,N\}$ one has
\begin{align*}
\E[|e_{n,k}^{(1,3)}|^2]+\E[|e_{n,k}^{(2,3)}|^2]+\E[|e_{n,k}^{(3)}|^2]&\le C(T)\sum_{j=0}^{n-1}\int_{t_j}^{t_{j+1}}\E[\big|\partial_\gamma\Phi(\widetilde{Z}_k(t))-\partial_\gamma\Phi(\widetilde{Z}_k(t_j))\big|^2]\dd t\\
&+C(T)\sum_{j=0}^{n-1}\int_{t_j}^{t_{j+1}}\E[\big|\psi(\widetilde{Z}_k(t))-\psi(\widetilde{Z}_k(t_j))\big|^2]\dd t.
\end{align*}
Recall that $h_k$ is defined by~\eqref{eq:defhk}. Applying the mean-value theorem, the upper bound~\eqref{eq:LipPhi} on the derivatives of $\partial_\gamma\Phi$ (which follows from regularity conditions imposed on $\Phi$), and the Cauchy--Schwarz inequality, for all $j\in\{0,\ldots,N-1\}$ and all $t\in[t_j,t_{j+1})$, one has 
\begin{align*}
\E[\big|\partial_\gamma\Phi(\widetilde{Z}_k(t))-\partial_\gamma\Phi(\widetilde{Z}_k(t_j))\big|^2]&\le C\E\bigl[e^{Ch_k(\widetilde{X}(t_j))(t-t_j)+|f_k(\widetilde{X}(t_j))\cdot(B(t)-B(t_j))|}\|\widetilde{Z}(t)-\widetilde{Z}_k(t_j)\|^2  \bigr]\\
&\le C\bigl(\E[e^{2Ch_k(\widetilde{X}(t_j))(t-t_j)+|f_k(\widetilde{X}(t_j))\cdot(B(t)-B(t_j))|}]\bigr)^{\frac12}\bigl(\E[\|\widetilde{Z}(t)-\widetilde{Z}_k(t_j)\|^4]\bigr)^{\frac12}.
\end{align*}
The mappings $h_k$ and $f_k$ are bounded on $\mathcal{D}$. As a result, using Lemma~\ref{lem:auxexpo} (with a conditional expectation argument) with $a=h_k$ and $b=f_k$, there exists $C(T)\in(0,\infty)$ such that
\[
\underset{0\le j\le N-1}\sup~\underset{t\in[t_j,t_{j+1})}\sup~\bigl(\E[e^{2Ch_k(\widetilde{X}(t_j))(t-t_j)+|f_k(\widetilde{X}(t_j))\cdot(B(t)-B(t_j))|}]\bigr)^{\frac12}\le C(T).
\]
Applying the inequality~\eqref{eq:auxZtilde1} from Lemma~\ref{lem:auxZtilde}, one then obtains the upper bound
\[
\underset{0\le j\le N-1}\sup~\underset{t\in[t_j,t_{j+1})}\sup~\E[\big|\partial_\gamma\Phi(\widetilde{Z}_k(t))-\partial_\gamma\Phi(\widetilde{Z}_k(t_j))\big|^2]\le C(T)\dt.
\]
Similarly, applying the upper bound~\eqref{eq:Lippsi} on the derivatives of $\partial_\gamma\psi$ (which follows from regularity conditions imposed on $\Phi$) one also gets the upper bound
\[
\underset{0\le j\le N-1}\sup~\underset{t\in[t_j,t_{j+1})}\sup~\E[\big|\psi(\widetilde{Z}_k(t))-\psi(\widetilde{Z}_k(t_j))\big|^2]\le C(T)\dt.
\]
As a result, one obtains the upper bound
\begin{equation}\label{eq:enk13_enk23_enk3}
\underset{0\le n\le N}\sup~\E[|e_{n,k}^{(1,3)}|^2]+\underset{0\le n\le N}\sup~\E[|e_{n,k}^{(2,3)}|^2]+\underset{0\le n\le N}\sup~\E[|e_{n,k}^{(3)}|^2]\le C(T)\dt.
\end{equation}

\paragraph{\bf Conclusion.}
Owing to the decomposition~\eqref{eq:decomp_e_n} of the error terms $e_{n,k}$, to the decompositions~\eqref{eq:decomp_e_n1} and~\eqref{eq:decomp_e_n2} of the error terms $e_{n,k}^{(1)}$ and $e_{n,k}^{(2)}$, and owing to the upper bounds~\eqref{eq:enk11e_nk21},~\eqref{eq:enk12e_nk22} and~\eqref{eq:enk13_enk23_enk3}, one finally obtains the following result: there exists $C(T)\in(0,\infty)$ such that for all $n\in\{1,\ldots,N\}$ one has
\[
\E[\|e_n\|^2]=\sum_{k=1}^{d}\E[|e_{n,k}|^2]\le C(T)\dt \sum_{j=0}^{n-1}\E[\|e_j\|^2]+C(T)\dt.
\]
Applying the discrete Gr\"onwall inequality then provides the strong error estimates~\eqref{eq:strong} and concludes the proof of Theorem~\ref{theo:strong}. 
\end{proof}

\section{Proof of Theorem~\ref{theo:weak}}\label{sec:proof2}

The objective of this section is to prove Theorem~\ref{theo:weak}, i.\,e. that the numerical scheme~\eqref{eq:schemeX} converges with weak rate equal to $1$ when the vector fields $g^0,g^1,\ldots,g^M$ and the mapping $\Phi$ satisfy the conditions given in the statement of Theorem~\ref{theo:weak}. Specifically, in the following it is assumed that Assumptions~\ref{ass:integratorDP} and~\ref{ass:integrator} are satisfied, and that there exists $\beta\in(0,\infty)$ such that $\|\Phi\|_{4,\beta}<\infty$, see~\eqref{eq:regPhi}. This assumption provides a control on the growth of the mappings $\partial_\gamma\Phi$ and $\psi=\partial_s\Phi+\frac12\partial_{\gamma}^2\Phi$ (defined by~\eqref{eq:psi}) and of their first and second order derivatives: there exists $C\in(0,\infty)$ such that for all $s\ge 0$, all $\gamma\in\R$ and $y\in[-1,1]$ one has
\begin{equation}\label{eq:regPhipsi-weak}
\sum_{|\alpha|\le 2}\left|\partial_s^{\alpha_s}\partial_\gamma^{\alpha_\gamma}\partial_y^{\alpha_y}\bigl(\partial_\gamma\Phi\bigr)(s,\gamma,y) \right|+\sum_{|\alpha|\le 2}\left|\partial_s^{\alpha_s}\partial_\gamma^{\alpha_\gamma}\partial_y^{\alpha_y}\psi(s,\gamma,y) \right|\le Ce^{C(s+|\gamma|)}.
\end{equation}

This section is organized as follows. Section~\ref{sec:decompweakerror} is devoted to the decomposition of the weak error using the solution to the associated Kolmogorov equation and to the statement of auxiliary error estimates (Lemmas~\ref{lem:epsilonI},~\ref{lem:epsilonII} and~\ref{lem:epsilonIII}). Additional auxiliary results are provided in Section~\ref{sec:auxweak}. The proofs of Lemmas~\ref{lem:epsilonI},~\ref{lem:epsilonII} and~\ref{lem:epsilonIII} are then given in Section~\ref{sec:proofsauxweak}. Finally, the proof of Theorem~\ref{theo:weak} is presented in Section~\ref{sec:conclusionproofweak}.

\subsection{Decomposition of the weak error}\label{sec:decompweakerror} 

The fundamental tool needed to establish weak error estimates is the Kolmogorov equation, for which we refer to~\cite{MR1840644}. Assume that $\theta\colon\mathcal{D}\to\R$ is of class $\mathcal{C}^3$, and introduce the mapping $u\colon[0,T]\times\R\to\R$ defined by
\begin{equation}\label{eq:u}
u(t,x)=\E[\theta(X(t))|X(0)=x],\qquad \forall~t\ge 0,\quad \forall~x\in\mathcal{D}.
\end{equation}
The function $u$ is the solution to the Kolmogorov equation 
\begin{equation}\label{eq:Kolmogorov}
\left\lbrace
\begin{aligned}
&\partial_tu(t,x)=\sum_{k=1}^{d}g_k^0(x)\partial_{x_k}u(t,x)+\frac12\sum_{k,\ell=1}^{d}g_k(x)\cdot g_\ell(x)\partial_{x_k x_\ell}^2u(t,x),\qquad \forall~(t,x)\in\R^+\times\mathcal{D},\\
&u(0,x)=\theta(x),\qquad \forall~x\in\mathcal{D}.
\end{aligned}
\right.
\end{equation}
Moreover, for all $t\ge 0$, the mapping $u(t,\cdot)$ is of class $\mathcal{C}^3$, and for all $T\in(0,\infty)$ there exists $C(T)\in(0,\infty)$ such that
\begin{equation}\label{eq:boundKolmogorov}
\sum_{|\alpha|\le 3}\underset{t\in[0,T]}\sup~\underset{x\in\mathcal{D}}\sup~|\partial_x^{\alpha}u(t,x)|\le C(T)\sum_{|\alpha|\le 3}\underset{x\in\mathcal{D}}\sup~|\partial_x^{\alpha}\theta(x)|,
\end{equation}
where for any multi-index $\alpha=(\alpha_k)_{1\le k\le d}$ one has $|\alpha|=\sum_{k=1}^{d}\alpha_k$ and $\partial_x^\alpha=\partial_{x_1}^{\alpha_1}\ldots\partial_{x_d}^{\alpha_d}$. To simplify notation, in the sequel the right-hand side of~\eqref{eq:boundKolmogorov} is denoted as $C(T,\theta)$.

The first step in the weak error analysis is to interpret the weak error, i.\,e. the left-hand side of~\eqref{eq:weak}, using the function $u$:  recalling that the initial values of the exaxt and numerical solutions satisfy $X(0)=X_0=x_0$, applying the definition~\eqref{eq:u} of the function $u$ yields
\[
\E[\theta(X_N)]-\E[\theta(X(T))]=\E[u(0,X_N)]-\E[u(T,x_0)]=\E[u(0,X_N)]-\E[u(T,X_0)].
\]
Next, by a telescoping sum argument, the weak error is decomposed as
\begin{equation}\label{eq:weakerror}
\E[\theta(X_N)]-\E[\theta(X(T))]=\sum_{n=0}^{N-1}\epsilon_{n},
\end{equation}
where, for all $n\in\{0,\ldots,N-1\}$, the error term $\epsilon_n$ is defined as
\[
\epsilon_n=\E[u(T-t_{n+1},X_{n+1})]-\E[u(T-t_n,X_n)].
\]
Recall that, for all $k\in\{1,\ldots,d\}$, the auxiliary process $\widetilde{X}_k$ is defined by~\eqref{eq:Xtilde}. For all $t\in[0,T]$, let
\[
\widetilde{X}(t)=\bigl(\widetilde{X}_1(t),\ldots,\widetilde{X}_d(t)\bigr).
\]
By construction, one has $\widetilde{X}(t_n)=X_n$ for all $n\in\{0,\ldots,N\}$. Therefore for all $n\in\{0,\ldots,N-1\}$ the error term $\epsilon_n$ is written as
\begin{equation}\label{eq:epsilon_n}
\epsilon_n=\E[u(T-t_{n+1},\widetilde{X}(t_{n+1}))]-\E[u(T-t_n,\widetilde{X}(t_n))].
\end{equation}
Furthermore, for all $k\in\{1,\ldots,d\}$, the auxiliary process $\widetilde{X}$ is the solution to the stochastic differential equation~\eqref{eq:auxXtilde1} on each interval $[t_n,t_{n+1})$. As a result, applying the It\^o formula, one obtains 
\begin{align*}
\epsilon_n&=-\int_{t_n}^{t_{n+1}}\E[\partial_tu(T-t,\widetilde{X}(t))]\dd t\\
&+\sum_{k=1}^{d}\int_{t_n}^{t_{n+1}}\E[\partial_\gamma\Phi(\widetilde{Z}_k(t))f_k^0(X_n)\partial_{x_k}u(T-t,\widetilde{X}(t))]\dd t\\
&+\frac12\sum_{k,\ell=1}^{d}\int_{t_n}^{t_{n+1}}\E[\partial_\gamma\Phi(\widetilde{Z}_k(t))\partial_\gamma\Phi(\widetilde{Z}_\ell(t))f_k(X_n)\cdot f_\ell(X_n)\partial_{x_k x_\ell}^2u(T-t,\widetilde{X}(t))]\dd t\\
&+\sum_{k=1}^d\int_{t_n}^{t_{n+1}}\E[\psi(\widetilde{Z}_k(t))\|f_k(X_n)\|^2\partial_{x_k}u(T-t,\widetilde{X}(t))]\dd t.
\end{align*}
Finally, the function $u$ is the solution to the Kolmogorov equation~\eqref{eq:Kolmogorov}. As a result, for all $n\in\{0,\ldots,N-1\}$, the error term $\epsilon_n$ is decomposed as
\begin{equation}\label{eq:decomp_epsilon_n}
\epsilon_n=\sum_{k=1}^d\epsilon_{n,k}^{(1)}+\sum_{k,\ell=1}^{d}\epsilon_{n,k,\ell}^{(2)}+\sum_{k=1}^d\epsilon_{n,k}^{(3)},
\end{equation}
where, for all $n\in\{0,\ldots,N-1\}$ and $k,\ell\in\{1,\ldots,d\}$, the error terms $\epsilon_{n,k}^{(1)}$, $\epsilon_{n,k}^{(2)}$ and $\epsilon_{n,k}^{(3)}$ are defined as
\begin{align}
\epsilon_{n,k}^{(1)}&=\int_{t_n}^{t_{n+1}}\E\left[\left(\partial_\gamma\Phi(\widetilde{Z}_k(t))f_k^0(X_n)-g_k^0(\widetilde{X}(t))\right)\partial_{x_k}u(T-t,\widetilde{X}(t))\right]\dd t \label{eq:epsilon_n1}\\
\epsilon_{n,k,\ell}^{(2)}&=\frac12\int_{t_n}^{t_{n+1}}\E\left[\left(\partial_\gamma\Phi(\widetilde{Z}_k(t))\partial_\gamma\Phi(\widetilde{Z}_\ell(t))\bigl(f_k\cdot f_\ell\bigr)(X_n)-\bigl(g_k\cdot g_\ell\bigr)(\widetilde{X}(t))\right)\partial_{x_k x_\ell}^2u(T-t,\widetilde{X}(t))\right]\dd t \label{eq:epsilon_n2}\\
\epsilon_{n,k}^{(3)}&=\int_{t_n}^{t_{n+1}}\E\left[\psi(\widetilde{Z}_k(t))\|f_k(X_n)\|^2\partial_{x_k}u(T-t,\widetilde{X}(t))\right]\dd t. \label{eq:epsilon_n3}
\end{align}
The error terms $\epsilon_{n,k}^{(1)}$ and $\epsilon_{n,k,\ell}^{(2)}$ are decomposed into further error terms, as explained below.

First, consider the error term $\epsilon_{n,k}^{(1)}$, for all $n\in\{0,\ldots,N-1\}$ and $k\in\{1,\ldots,d\}$, defined by~\eqref{eq:epsilon_n1}. Recall that by construction of the auxiliary processes $\widetilde{X}$ and $\widetilde{Z}_k$ and by assumption~\ref{eq:condPhi1}, one has the identities $X_n=\widetilde{X}(t_n)$, $\widetilde{Z}_k(t_n)=\bigl(0,0,\widetilde{X}_k(t_n)\bigr)$, $\partial_\gamma\Phi(\widetilde{Z}_k(t_n))=\sigma(\widetilde{X}_k(t_n))$. As a result, owing to~\eqref{eq:sigmaF}, one has
\[
\partial_\gamma\Phi(\widetilde{Z}_k(t_n))f_k^0(\widetilde{X}(t_n))=g_k^0(X_n).
\]
The error term $\epsilon_{n,k}^{(1)}$ defined by~\eqref{eq:epsilon_n1} is thus decomposed as
\begin{equation}\label{eq:decomp_epsilon_n1}
\epsilon_{n,k}^{(1)}=\epsilon_{n,k}^{(1,1)}+\epsilon_{n,k}^{(1,2)},
\end{equation}
where the error terms $\epsilon_{n,k}^{(1,1)}$ and $\epsilon_{n,k}^{(1,2)}$ are defined as
\begin{align}
\epsilon_{n,k}^{(1,1)}&=\int_{t_n}^{t_{n+1}}\E\left[\left(g_k^0(\widetilde{X}(t_n))-g_k^0(\widetilde{X}(t))\right)\partial_{x_k}u(T-t,\widetilde{X}(t))\right]\dd t,\label{eq:epsilon_n11}\\
\epsilon_{n,k}^{(1,2)}&=\int_{t_n}^{t_{n+1}}\E\left[\left(\partial_\gamma\Phi(\widetilde{Z}_k(t))-\partial_\gamma\Phi(\widetilde{Z}_k(t_n))\right)f_k^0(\widetilde{X}(t_n))\partial_{x_k}u(T-t,\widetilde{X}(t))\right]\dd t.\label{eq:epsilon_n12}
\end{align}

Second, consider the error term $\epsilon_{n,k,\ell}^{(2)}$, for all $n\in\{0,\ldots,N-1\}$ and $k,\ell\in\{1,\ldots,d\}$, defined by~\eqref{eq:epsilon_n2}. To simplify the notation, for all $n\in\{0,\ldots,N-1\}$, $k\in\{1,\ldots,d\}$, and $t\in[t_n,t_{n+1})$, set
\begin{equation}\label{eq:Delta_nk}
\Delta_{n,k}(t)=\partial_\gamma\Phi(\widetilde{Z}_k(t))-\partial_\gamma\Phi(\widetilde{Z}_k(t_n)).
\end{equation}
Recall that by construction of the auxiliary processes $\widetilde{X}$ and $\widetilde{Z}_k$, one has the identities $X_n=\widetilde{X}(t_n)$, $\widetilde{Z}_k(t_n)=\bigl(0,0,\widetilde{X}_k(t_n)\bigr)$, $\partial_\gamma\Phi(\widetilde{Z}_k(t_n))=\sigma(\widetilde{X}_k(t_n))$. As a result, owing to~\eqref{eq:sigmaF}, one has
\begin{align*}
\partial_\gamma\Phi(\widetilde{Z}_k(t))\partial_\gamma\Phi(\widetilde{Z}_\ell(t))&\bigl(f_k\cdot f_\ell\bigr)(X_n)-\bigl(g_k\cdot g_\ell\bigr)(\widetilde{X}(t))\\
&=\bigl(g_k\cdot g_\ell\bigr)(\widetilde{X}(t_n))-\bigl(g_k\cdot g_\ell\bigr)(\widetilde{X}(t))\\
&+\partial_\gamma\Phi(\widetilde{Z}_k(t))\partial_\gamma\Phi(\widetilde{Z}_\ell(t))\bigl(f_k\cdot f_\ell\bigr)(\widetilde{X}(t_n))-\bigl(g_k\cdot g_\ell\bigr)(\widetilde{X}(t_n))\\
&=\bigl(g_k\cdot g_\ell\bigr)(\widetilde{X}(t_n))-\bigl(g_k\cdot g_\ell\bigr)(\widetilde{X}(t))\\
&+\Bigl(\partial_\gamma\Phi(\widetilde{Z}_k(t))\partial_\gamma\Phi(\widetilde{Z}_\ell(t))-\partial_\gamma\Phi(\widetilde{Z}_k(t_n))\partial_\gamma\Phi(\widetilde{Z}_\ell(t_n))\Bigr)\bigl(f_k\cdot f_\ell\bigr)(X_n).
\end{align*}
In addition, owing to the auxiliary notation~\eqref{eq:Delta_nk} introduced above, one has
\begin{align*}
\partial_\gamma\Phi(\widetilde{Z}_k(t))\partial_\gamma\Phi(\widetilde{Z}_\ell(t))&-\partial_\gamma\Phi(\widetilde{Z}_k(t_n))\partial_\gamma\Phi(\widetilde{Z}_\ell(t_n))\\
&=\Delta_{n,k}(t)\partial_\gamma\Phi(\widetilde{Z}_\ell(t_n))+\partial_\gamma\Phi(\widetilde{Z}_k(t_n))\Delta_{n,\ell}(t)+\Delta_{n,k}(t)\Delta_{n,\ell}(t)\\
&=\Delta_{n,k}(t)\sigma(\widetilde{X}_\ell(t_n))+\sigma(\widetilde{X}_k(t_n))\Delta_{n,\ell}(t)+\Delta_{n,k}(t)\Delta_{n,\ell}(t).
\end{align*}
As a result, the error term $\epsilon_{n,k,\ell}^{(2)}$ defined by~\eqref{eq:epsilon_n2} is decomposed as
\begin{equation}\label{eq:decomp_epsilon_n2}
\epsilon_{n,k,\ell}^{(2)}=\epsilon_{n,k,\ell}^{(2,1)}+\epsilon_{n,k,\ell}^{(2,2)}+\epsilon_{n,k,\ell}^{(2,3)},
\end{equation}
where the error terms $\epsilon_{n,k,\ell}^{(2,1)}$, $\epsilon_{n,k,\ell}^{(2,2)}$ and $\epsilon_{n,k,\ell}^{(2,3)}$ are defined as
\begin{align}
\epsilon_{n,k,\ell}^{(2,1)}&=\frac12\int_{t_n}^{t_{n+1}}\E\left[\left(\bigl(g_k\cdot g_\ell\bigr)(\widetilde{X}(t_n))-\bigl(g_k\cdot g_\ell\bigr)(\widetilde{X}(t))\right)\partial_{x_k x_\ell}^2u(T-t,\widetilde{X}(t))\right]\dd t\label{eq:epsilon_n21}\\
\epsilon_{n,k,\ell}^{(2,2)}&=\frac12\int_{t_n}^{t_{n+1}}\E\left[\left(\partial_\gamma\Phi(\widetilde{Z}_k(t))-\partial_\gamma\Phi(\widetilde{Z}_k(t_n))\right)\bigl(f_k\cdot g_\ell\bigr)(X_n)\partial_{x_k x_\ell}^2u(T-t,\widetilde{X}(t))\right]\dd t \label{eq:epsilon_n22}\\
\epsilon_{n,k,\ell}^{(2,3)}&=\frac12\int_{t_n}^{t_{n+1}}\E\left[\Delta_{n,k}(t)\Delta_{n,\ell}(t)\bigl(f_k\cdot f_\ell\bigr)(X_n)\partial_{x_k x_\ell}^2u(T-t,\widetilde{X}(t))\right]\dd t.\label{eq:epsilon_n23}
\end{align}
Finally, owing to the condition~\eqref{eq:condPhi2} and the definition~\eqref{eq:psi} of the function $\psi$, for all $n\in\{0,\ldots,N-1\}$ and $k\in\{1,\ldots,d\}$, the error term $\epsilon_{n,k}^{(3)}$ is expressed as
\begin{equation}\label{eq:epsilon_n3bis}
\epsilon_{n,k}^{(3)}=\int_{t_n}^{t_{n+1}}\E\left[\left(\psi(\widetilde{Z}_k(t))-\psi(\widetilde{Z}_k(t_n))\right)\|f_k(X_n)\|^2\partial_{x_k}u(T-t,\widetilde{X}(t))\right]\dd t.
\end{equation}
To simplify the presentation and in order to avoid repeating the same arguments several times, the error terms introduced above are not treated independently. The following results are proved below.

\begin{lemma}\label{lem:epsilonI} 
There exists $C(T,\theta)\in(0,\infty)$ such that
\begin{equation}\label{eq:lem_epsilonI}
\sum_{k=1}^{d}\underset{0\le n\le N-1}\sup~|\epsilon_{n,k}^{(1,2)}|+\sum_{k,\ell=1}^{d}\underset{0\le n\le N-1}\sup~|\epsilon_{n,k,\ell}^{(2,2)}|+\sum_{k=1}^{d}\underset{0\le n\le N-1}\sup~|\epsilon_{n,k}^{(3)}|\le C(T,\theta)\dt^2.
\end{equation}
\end{lemma}

\begin{lemma}\label{lem:epsilonII}
There exists $C(T,\theta)\in(0,\infty)$ such that
\begin{equation}\label{eq:lem_epsilonII}
\sum_{k=1}^{d}\underset{0\le n\le N-1}\sup~|\epsilon_{n,k}^{(1,1)}|+\sum_{k,\ell=1}^{d}\underset{0\le n\le N-1}\sup~|\epsilon_{n,k,\ell}^{(2,1)}|\le C(T,\theta)\dt^2.
\end{equation}
\end{lemma}

\begin{lemma}\label{lem:epsilonIII}
There exists $C(T,\theta)\in(0,\infty)$ such that
\begin{equation}\label{eq:lem_epsilonIII}
\sum_{k,\ell=1}^{d}\underset{0\le n\le N-1}\sup~|\epsilon_{n,k,\ell}^{(2,3)}|\le C(T,\theta)\dt^2.
\end{equation}
\end{lemma}

The proofs of Lemmas~\ref{lem:epsilonI},~\ref{lem:epsilonII} and~\ref{lem:epsilonIII} are postponed to Section~\ref{sec:proofsauxweak}. Theorem~\ref{theo:weak} is then established as a straightforward consequence of those results in Section~\ref{sec:conclusionproofweak}.

\subsection{Auxiliary results}\label{sec:auxweak}

The objective of this section is to provide some auxiliary results in order to simplify the proofs given in Section~\ref{sec:proofsauxweak} below.

Lemma~\ref{lem:Delta_nk} provides upper bounds on the auxiliary error terms $\Delta_{n,k}(t)$ defined by~\eqref{eq:Delta_nk}.
\begin{lemma}\label{lem:Delta_nk}
There exists $C\in(0,\infty)$ such that one has 
\begin{equation}\label{eq:lemDelta_nk}
\sum_{k=1}^d\underset{0\le n\le N-1}\sup~\underset{t\in[t_n,t_{n+1})}\sup~\E[|\Delta_{n,k}(t)|^2]\le C\dt.
\end{equation}
\end{lemma}

\begin{proof}[Proof of Lemma~\ref{lem:Delta_nk}]
Let $k\in\{1,\ldots,d\}$ and $n\in\{0,\ldots,N-1\}$. Recall the definition~\eqref{eq:defhk} of the mapping $h_k$. Applying the regularity property of $\Phi$ and the Cauchy--Schwarz inequality, for all $t\in[t_n,t_{n+1})$ one then obtains
\begin{align*}
\E[|\Delta_{n,k}(t)|^2]&\le C\E\bigl[e^{Ch_k(X_n)(t-t_n)+C|f_k(X_n)\cdot(B(t)-B(t_n))|}\|\widetilde{Z}_k(t)-\widetilde{Z}_k(t_n)\|^2  \bigr]\\
&\le C\bigl(\E\bigl[e^{2Ch_k(X_n)(t-t_n)+2C|f_k(X_n)\cdot(B(t)-B(t_n))|}\bigr]\bigr)^{\frac12}\bigl(\E\bigl[\|\widetilde{Z}_k(t)-\widetilde{Z}_k(t_n)\|^4 \bigr]\bigr)^{\frac12}.
\end{align*}
Applying the inequality~\eqref{eq:lemauxexpo} from Lemma~\ref{lem:auxexpo} (with a conditional expectation argument) and the regularity property~\eqref{lem:auxZtilde} from Lemma~\ref{lem:auxZtilde} on the auxiliary process $\widetilde{Z}_k$, for all $t\in[t_n,t_{n+1})$ one obtains the upper bound~\eqref{eq:lemDelta_nk} and the proof of Lemma~\ref{lem:Delta_nk} is completed.
\end{proof}

Next, for all $n\in\{0,\ldots,N-1\}$ and $t\in[t_n,t_{n+1})$, set
\begin{align}
\delta_{n,k}^{1}(t)&=\partial_{x_k}u(T-t,\widetilde{X}(t))-\partial_{x_k}u(T-t,\widetilde{X}(t_n))
\label{eq:delta1}\\
\delta_{n,k,\ell}^{2}(t)&=\partial_{x_k x_\ell}^2u(T-t,\widetilde{X}(t))-\partial_{x_k x_\ell}^2u(T-t,\widetilde{X}(t_n)).\label{eq:delta2}
\end{align}

Lemma~\ref{lem:delta12} provides upper bounds on the auxiliary error terms $\delta_{n,k}^{1}(t)$ and $\delta_{n,k,\ell}^{2}(t)$ defined by~\eqref{eq:delta1} and~\eqref{eq:delta2} above.
\begin{lemma}\label{lem:delta12}
There exists $C(T,\theta)\in(0,\infty)$ such that one has
\begin{equation}\label{eq:bounddelta12}
\sum_{k=1}^{d}\underset{0\le n\le N-1}\sup~\underset{t\in[t_n,t_{n+1}]}\sup~\E[|\delta_{n,k}^{1}(t)|^2]+\sum_{k,\ell=1}^{d}\underset{0\le n\le N-1}\sup~\underset{t\in[t_n,t_{n+1}]}\sup~\E[|\delta_{n,k,\ell}^{2}(t)|^2]\le C(T,\theta)\dt.
\end{equation}
\end{lemma}

\begin{proof}
Let $k,\ell\in\{1,\ldots,d\}$ and $n\in\{0,\ldots,N-1\}$. The spatial second and third order partial derivatives of the solution $u$ to the Kolmogorov equation~\eqref{eq:Kolmogorov} are bounded owing to the regularity property~\eqref{eq:boundKolmogorov}. As a consequence, for all $t\in[t_n,t_{n+1}]$ one has
\begin{align*}
\E[|\delta_{n,k}^{1}(t)|^2]\le C(T,\theta)\E[\|\widetilde{X}(t)-\widetilde{X}(t_n)\|^2]\\
\E[|\delta_{n,k,\ell}^{2}(t)|^2]\le C(T,\theta)\E[\|\widetilde{X}(t)-\widetilde{X}(t_n)\|^2].
\end{align*}
Applying the regularity property~\eqref{eq:auxXtilde2} from Lemma~\ref{lem:auxXtilde} for the processes $\widetilde{X}_1,\ldots,\widetilde{X}_d$, one obtains the following upper bound: for all $t\in[t_n,t_{n+1}]$ one has
\[
\E[|\delta_{n,k}^{1}(t)|^2]+\E[|\delta_{n,k,\ell}^{2}(t)|^2]\le C(T,\theta)\dt.
\]
This yields the inequality~\eqref{eq:bounddelta12} and the proof of Lemma~\ref{lem:delta12} is completed.
\end{proof}

\subsection{Proof of Lemmas~\ref{lem:epsilonI},~\ref{lem:epsilonII} and~\ref{lem:epsilonIII}}\label{sec:proofsauxweak}

\begin{proof}[Proof of Lemma~\ref{lem:epsilonI}]

The objective is to prove the error bounds~\eqref{eq:lem_epsilonI} on the error terms $\epsilon_{n,k}^{(1,2)}$, $\epsilon_{n,k,\ell}^{(2,2)}$ and $\epsilon_{n,k,\ell}^{(3)}$ defined by~\eqref{eq:epsilon_n12},~\eqref{eq:epsilon_n22} and~\eqref{eq:epsilon_n3}.

Let $k,\ell\in\{1,\ldots,d\}$ and $n\in\{0,\ldots,N-1\}$.

First, the error term $\epsilon_{n,k}^{(1,2)}$ is decomposed as
\begin{equation}\label{eq:decomp_epsilon_n12}
\epsilon_{n,k}^{(1,2)}=\epsilon_{n,k}^{(1,2,1)}+\epsilon_{n,k}^{(1,2,2)},
\end{equation}
where the error terms $\epsilon_{n,k}^{(1,2,1)}$ and $\epsilon_{n,k}^{(1,2,2)}$ are defined as
\begin{align}
\epsilon_{n,k}^{(1,2,1)}&=\int_{t_n}^{t_{n+1}}\E\left[\left(\partial_\gamma\Phi(\widetilde{Z}_k(t))-\partial_\gamma\Phi(\widetilde{Z}_k(t_n))\right)f_k^0(\widetilde{X}(t_n))\delta_{n,k}^{1}(t)\right]\dd t,\label{eq:epsilon_n121}\\
\epsilon_{n,k}^{(1,2,2)}&=\int_{t_n}^{t_{n+1}}\E\left[\left(\partial_\gamma\Phi(\widetilde{Z}_k(t))-\partial_\gamma\Phi(\widetilde{Z}_k(t_n))\right)f_k^0(\widetilde{X}(t_n))\partial_{x_k}u(T-t,\widetilde{X}(t_n))\right]\dd t,\label{eq:epsilon_n122}
\end{align}
with the auxiliary error term $\delta_{n,k}^{1}(t)$ defined by~\eqref{eq:delta1}.

Second, the error term $\epsilon_{n,k,\ell}^{(2,2)}$ is decomposed as
\begin{equation}\label{eq:decomp_epsilon_n22}
\epsilon_{n,k}^{(2,2)}=\epsilon_{n,k}^{(2,2,1)}+\epsilon_{n,k}^{(2,2,2)},
\end{equation}
where the error terms $\epsilon_{n,k}^{(2,2,1)}$ and $\epsilon_{n,k}^{(2,2,2)}$ are defined as
\begin{align}
\epsilon_{n,k,\ell}^{(2,2,1)}&=\frac12\int_{t_n}^{t_{n+1}}\E\left[\left(\partial_\gamma\Phi(\widetilde{Z}_k(t))-\partial_\gamma\Phi(\widetilde{Z}_k(t_n))\right)\bigl(f_k\cdot g_\ell\bigr)(X_n)\delta_{n,k,\ell}^2(t)\right]\dd t \label{eq:epsilon_n221}\\
\epsilon_{n,k,\ell}^{(2,2,2)}&=\frac12\int_{t_n}^{t_{n+1}}\E\left[\left(\partial_\gamma\Phi(\widetilde{Z}_k(t))-\partial_\gamma\Phi(\widetilde{Z}_k(t_n))\right)\bigl(f_k\cdot g_\ell\bigr)(X_n)\partial_{x_k x_\ell}^2u(T-t,\widetilde{X}(t_n))\right]\dd t,\label{eq:epsilon_n222}
\end{align}
with the auxiliary error term $\delta_{n,k,\ell}^{2}(t)$ defined by~\eqref{eq:delta2}.

Finally, the error term $\epsilon_{n,k}^{(3)}$ is decomposed as
\begin{equation}\label{eq:decomp_epsilon_n3}
\epsilon_{n,k}^{(3)}=\epsilon_{n,k}^{(3,1)}+\epsilon_{n,k}^{(3,2)},
\end{equation}
where the error terms $\epsilon_{n,k}^{(3,1)}$ and $\epsilon_{n,k}^{(3,2)}$ are defined as
\begin{align}
\epsilon_{n,k}^{(3,1)}&=\int_{t_n}^{t_{n+1}}\E\left[\left(\psi(\widetilde{Z}(t))-\psi(\widetilde{Z}(t_n))\right)\|f_k(X_n)\|^2\delta_{n,k}^1(t)\right]\dd t,\label{eq:epsilon_n31}\\
\epsilon_{n,k}^{(3,2)}&=\int_{t_n}^{t_{n+1}}\E\left[\left(\psi(\widetilde{Z}(t))-\psi(\widetilde{Z}(t_n))\right)\|f_k(X_n)\|^2\partial_{x_k}u(T-t,\widetilde{X}(t_n))\right]\dd t.\label{eq:epsilon_n32}
\end{align}
with the auxiliary error term $\delta_{n,k}^1(t)$ defined by~\eqref{eq:delta1}.

\paragraph{\it Treatment of the error terms $\epsilon_{n,k}^{(1,2,1)}$, $\epsilon_{n,k,\ell}^{(2,2,1)}$ and $\epsilon_{n,k}^{(3,1)}$.}

Recall that the mappings $f_k$ and $g_\ell$ are bounded on $\mathcal{D}$. Applying the Cauchy--Schwarz inequality and the upper bound~\eqref{eq:bounddelta12} for $\delta_{n,k}^{1}(t)$ and $\delta_{n,k,\ell}^{2}(t)$, one has
\begin{align*}
|\epsilon_{n,k}^{(1,2,1)}|+|\epsilon_{n,k,\ell}^{(2,2,1)}|&\le C\int_{t_n}^{t_{n+1}}\bigl(\E[|\partial_\gamma\Phi(\widetilde{Z}_k(t))-\partial_\gamma\Phi(\widetilde{Z}_k(t_n))|^2]\bigr)^{\frac12}\bigl(\E[|\delta_{n,k}^{1}(t)|^2]\bigr)^{\frac12}\dd t\\
&+C\int_{t_n}^{t_{n+1}}\bigl(\E[|\partial_\gamma\Phi(\widetilde{Z}_k(t))-\partial_\gamma\Phi(\widetilde{Z}_k(t_n))|^2]\bigr)^{\frac12}\bigl(\E[|\delta_{n,k,\ell}^{2}(t)|^2]\bigr)^{\frac12}\dd t\\
&+C\int_{t_n}^{t_{n+1}}\bigl(\E[|\psi(\widetilde{Z}_k(t))-\psi(\widetilde{Z}_k(t_n))|^2]\bigr)^{\frac12}\bigl(\E[|\delta_{n,k}^{1}(t)|^2]\bigr)^{\frac12}\dd t\\
&\le C\dt^{\frac12}\int_{t_n}^{t_{n+1}}\bigl(\E[|\partial_\gamma\Phi(\widetilde{Z}_k(t))-\partial_\gamma\Phi(\widetilde{Z}_k(t_n))|^2]\bigr)^{\frac12}\dd t\\
&+C\dt^{\frac12}\int_{t_n}^{t_{n+1}}\bigl(\E[|\psi(\widetilde{Z}_k(t))-\psi(\widetilde{Z}_k(t_n))|^2]\bigr)^{\frac12}\dd t.
\end{align*}
Moreover, the mappings $\partial_\gamma\Phi$ and $\psi$ (defined by~\eqref{eq:psi}) are of class $\mathcal{C}^1$ and the growth of their derivatives is controlled since one has assumed that $\|\Phi\|_{3,\beta}\le\|\Phi\|_{4,\beta}<\infty$ for some $\beta\in(0,\infty)$, see~\eqref{eq:regPhi} and~\eqref{eq:regPhipsi-weak}. Applying the Cauchy--Schwarz inequality, one then has
\begin{align*}
|\epsilon_{n,k}^{(1,2,1)}|+|\epsilon_{n,k,\ell}^{(2,2,1)}|&\le C\dt^{\frac12}\int_{t_n}^{t_{n+1}}\bigl(\E\bigl[e^{Ch_k(X_n)(t-t_n)+C|f_k(X_n)\cdot(B(t)-B(t_n))|}\|\widetilde{Z}_k(t)-\widetilde{Z}_k(t_n)\|^2\bigr]\bigr)^{\frac12}\dd t\\
&\le C\dt^{\frac12}\int_{t_n}^{t_{n+1}}\bigl(\E\bigl[e^{2Ch_k(X_n)(t-t_n)+2C|f_k(X_n)\cdot(B(t)-B(t_n))|}]\bigr)^{\frac14}\bigl(\E[\|\widetilde{Z}_k(t)-\widetilde{Z}_k(t_n)\|^4 \bigr]\bigr)^{\frac14}\dd t,
\end{align*}
where the mapping $h_k$ is defined by~\eqref{eq:defhk}. Applying the inequality~\eqref{eq:lemauxexpo} from Lemma~\ref{lem:auxexpo} (with a conditional expectation argument) and the regularity property~\eqref{eq:auxZtilde1} from Lemma~\ref{lem:auxZtilde} for the process $\widetilde{Z}_k$, one obtains the following upper bound for the error terms $\epsilon_{n,k}^{(1,2,1)}$, $\epsilon_{n,k,\ell}^{(2,2,1)}$ and $\epsilon_{n,k}^{(3,1)}$ defined by~\eqref{eq:epsilon_n121},~\eqref{eq:epsilon_n221} and~\eqref{eq:epsilon_n31}: there exists $C(T,\theta)\in(0,\infty)$ such that for all $n\in\{0,\ldots,N-1\}$ one has
\begin{equation}\label{eq:boundepsilon_n121-221-31}
\sum_{k=1}^{d}|\epsilon_{n,k}^{(1,2,1)}|+\sum_{k,\ell=1}^{d}|\epsilon_{n,k,\ell}^{(2,2,1)}|+\sum_{k=1}^{d}|\epsilon_{n,k}^{(3,1)}|\le C(T,\theta)\dt^2.
\end{equation}

\paragraph{\it Treatment of the error terms $\epsilon_{n,k}^{(1,2,2)}$ and $\epsilon_{n,k,\ell}^{(2,2,2)}$.}

The mapping $\partial_\gamma\Phi$ is of class $\mathcal{C}^2$ and the growth of its first and second order derivatives is controlled since one has assumed that $\|\Phi\|_{3,\beta}\le\|\Phi\|_{4,\beta}<\infty$ for some $\beta\in(0,\infty)$, see~\eqref{eq:regPhi} and~\eqref{eq:regPhipsi-weak}. Owing to the definition~\eqref{eq:Ztilde} of the auxiliary process $\widetilde{Z}_k$, applying the It\^o formula, for all $t\in[t_n,t_{n+1})$ one has
\begin{align*}
\partial_\gamma\Phi(\widetilde{Z}_k(t))-\partial_\gamma\Phi(\widetilde{Z}_k(t_n))&=\int_{t_n}^{t}\partial_s\partial_\gamma\Phi(\widetilde{Z}_k(s))\|f_k(X_n)\|^2 \dd s\\
&+\int_{t_n}^{t}\partial_{\gamma}^2\Phi(\widetilde{Z}_k(s))f_k^0(X_n)\dd s\\
&+\int_{t_n}^{t}\partial_{\gamma}^2\Phi(\widetilde{Z}_k(s))f_k(X_n)\cdot\dd B(s)\\
&+\frac12\int_{t_n}^{t}\partial_{\gamma}^3\Phi(\widetilde{Z}_k(s))\|f_k(X_n)\|^2\dd s.
\end{align*}
By a conditional expectation argument, one has
\begin{align*}
&\E\left[\int_{t_n}^{t}\partial_{\gamma}^2\Phi(\widetilde{Z}_k(s))f_k(X_n)\cdot\dd B(s)f_k^0(\widetilde{X}(t_n))\partial_{x_k}u(T-t,\widetilde{X}(t_n))\right]=0\\
&\int_{t_n}^{t_{n+1}}\E\left[\int_{t_n}^{t}\partial_{\gamma}^2\Phi(\widetilde{Z}_k(s))f_k(X_n)\cdot\dd B(s)\bigl(f_k\cdot g_\ell\bigr)(X_n)\partial_{x_k x_\ell}^2u(T-t,\widetilde{X}(t_n))\right]\dd t=0.
\end{align*}
Therefore, since the mappings $f_k^0$, $f_k$ and $g_\ell$ are bounded on $\mathcal{D}$, recalling that $\|\Phi\|_{3,\beta}\le \|\Phi\|_{4,\beta}<\infty$ for some $\beta\in(0,\infty)$, one obtains the upper bounds
\begin{align*}
|\epsilon_{n,k}^{(1,2,2)}|&\le C\int_{t_n}^{t_{n+1}}\int_{t_n}^{t}\E[e^{Ch_k(X_n)(s-t_n)+|f_k(X_n)\cdot(B(s)-B(t_n))|}|\partial_{x_k}u(T-t,\widetilde{X}(t_n))|]\dd s\dd t,\\
|\epsilon_{n,k,\ell}^{(2,2,2)}|&\le C\int_{t_n}^{t_{n+1}}\int_{t_n}^{t}\E[e^{Ch_k(X_n)(s-t_n)+|f_k(X_n)\cdot(B(s)-B(t_n))|}|\partial_{x_kx_\ell}^2u(T-t,\widetilde{X}(t_n))|]\dd s\dd t.
\end{align*}
The first-order spatial partial derivative $\partial_{x_k}u$ and the second-order spatial partial derivative $\partial_{x_kx_\ell}^2u$ of the solution $u$ to the Kolmogorov equation are bounded, see Equation~\eqref{eq:boundKolmogorov}. Applying the inequality~\eqref{eq:lemauxexpo} from Lemma~\ref{lem:auxexpo} (with a conditional expectation argument) then yields the following upper bound for the error terms $\epsilon_{n,k}^{(1,2,2)}$ and $\epsilon_{n,k}^{(2,2,2)}$ defined by~\eqref{eq:epsilon_n122} and~\eqref{eq:epsilon_n222}: there exists $C(T,\theta)\in(0,\infty)$ such that for all $n\in\{0,\ldots,N-1\}$ one has
\begin{equation}\label{eq:boundepsilon_n122-222}
\sum_{k=1}^{d}|\epsilon_{n,k}^{(1,2,2)}|+\sum_{k,\ell=1}^{d}|\epsilon_{n,k,\ell}^{(2,2,2)}|\le C(T,\theta)\dt^2.
\end{equation}

\paragraph{\it Treatment of the error term $\epsilon_{n,k}^{(3,2)}$.}

The mapping $\psi$ (defined by~\eqref{eq:psi}) is of class $\mathcal{C}^2$ and the growth of its first and second order derivatives is controlled since one has assumed that $\|\Phi\|_{4,\beta}<\infty$ for some $\beta\in(0,\infty)$, see~\eqref{eq:regPhi} and~\eqref{eq:regPhipsi-weak}. Owing to the definition~\eqref{eq:Ztilde} of the auxiliary process $\widetilde{Z}_k$, applying the It\^o formula, for all $t\in[t_n,t_{n+1})$ one has
\begin{align*}
\psi(\widetilde{Z}_k(t))-\psi(\widetilde{Z}_k(t_n))&=\int_{t_n}^{t}\partial_{\gamma}\psi(\widetilde{Z}_k(s))f_k^0(X_n)\dd s\\
&+\int_{t_n}^{t}\partial_{\gamma}\psi(\widetilde{Z}_k(s))f_k(X_n)\cdot\dd B(s)\\
&+\int_{t_n}^{t}\bigl(\partial_{s}\psi+\frac12\partial_{\gamma}^2\psi\bigr)(\widetilde{Z}_k(s))\|f_k(X_n)\|^2\dd s.
\end{align*}
By a conditional expectation argument, one has
\[\E\left[\int_{t_n}^{t}\partial_{\gamma}\psi(\widetilde{Z}_k(s))f_k(X_n)\cdot\dd B(s)\|f_k(X_n)\|^2\partial_{x_k}u(T-t,\widetilde{X}(t_n))\right]=0.
\]
Therefore, since the mappings $f_k^0$ and $f_k$ are bounded on $\mathcal{D}$, recalling that $\|\Phi\|_{4,\beta}<\infty$, one obtains the upper bound
\[
|\epsilon_{n,k}^{(3,2)}|\le C\int_{t_n}^{t_{n+1}}\int_{t_n}^{t}\E[e^{Ch_k(X_n)(s-t_n)+|f_k(X_n)\cdot(B(s)-B(t_n))|}|\partial_{x_k}u(T-t,\widetilde{X}(t_n))|]\dd s\dd t.
\]
The first-order spatial partial derivative $\partial_{x_k}u$ of the solution $u$ to the Kolmogorov equation is bounded, see Equation~\eqref{eq:boundKolmogorov}. Applying the inequality~\eqref{eq:lemauxexpo} from Lemma~\ref{lem:auxexpo} (with a conditional expectation argument) then yields the following upper bound for the error term $\epsilon_{n,k}^{(3,2)}$ defined by~\eqref{eq:epsilon_n32}: there exists $C(T,\theta)\in(0,\infty)$ such that for all $n\in\{0,\ldots,N-1\}$ one has
\begin{equation}\label{eq:boundepsilon_n32}
\sum_{k=1}^{d}|\epsilon_{n,k}^{(3,2)}|\le C(T,\theta)\dt^2.
\end{equation}

\paragraph{\it Conclusion}

Combining the upper bounds~\eqref{eq:boundepsilon_n121-221-31},~\eqref{eq:boundepsilon_n122-222} and~\eqref{eq:boundepsilon_n32} with the decompositions~\eqref{eq:decomp_epsilon_n12},~\eqref{eq:decomp_epsilon_n22} and~\eqref{eq:decomp_epsilon_n3} of the error terms $\epsilon_{n,k}^{(1,2)}$, $\epsilon_{n,k,\ell}^{(2,2)}$ and $\epsilon_{n,k}^{(3)}$, one obtains the inequality~\eqref{eq:lem_epsilonI} and the proof of Lemma~\ref{lem:epsilonI} is thus completed.

\end{proof}

\begin{proof}[Proof of Lemma~\ref{lem:epsilonII}]

The objective is to prove the error bounds~\eqref{eq:lem_epsilonII} on the error terms $\epsilon_{n,k}^{(1,1)}$ and $\epsilon_{n,k,\ell}^{(2,1)}$ defined by~\eqref{eq:epsilon_n11} and~\eqref{eq:epsilon_n21}.

Let $k,\ell\in\{1,\ldots,d\}$ and $n\in\{0,\ldots,N-1\}$.

First, the error term $\epsilon_{n,k}^{(1,1)}$ is decomposed as
\begin{equation}\label{eq:decomp_epsilon_n11}
\epsilon_{n,k}^{(1,1)}=\epsilon_{n,k}^{(1,1,1)}+\epsilon_{n,k}^{(1,1,2)},
\end{equation}
where the error terms $\epsilon_{n,k}^{(1,1,1)}$ and $\epsilon_{n,k}^{(1,1,2)}$ are defined as
\begin{align}
\epsilon_{n,k}^{(1,1,1)}&=\int_{t_n}^{t_{n+1}}\E\left[\left(g_k^0(\widetilde{X}(t_n))-g_k^0(\widetilde{X}(t))\right)\delta_{n,k}^1(t)\right]\dd t,\label{eq:epsilon_n111}\\
\epsilon_{n,k}^{(1,1,2)}&=\int_{t_n}^{t_{n+1}}\E\left[\left(g_k^0(\widetilde{X}(t_n))-g_k^0(\widetilde{X}(t))\right)\partial_{x_k}u(T-t,\widetilde{X}(t_n))\right]\dd t,\label{eq:epsilon_n112}
\end{align}
with the auxiliary error term $\delta_{n,k}^1(t)$ defined by~\eqref{eq:delta1}.

Second, the error term $\epsilon_{n,k,\ell}^{(2,1)}$ is decomposed as
\begin{equation}\label{eq:decomp_epsilon_n21}
\epsilon_{n,k,\ell}^{(2,1)}=\epsilon_{n,k,\ell}^{(2,1,1)}+\epsilon_{n,k,\ell}^{(2,1,2)},
\end{equation}
where the error terms $\epsilon_{n,k,\ell}^{(2,1,1)}$ and $\epsilon_{n,k,\ell}^{(2,1,2)}$ are defined as
\begin{align}
\epsilon_{n,k,\ell}^{(2,1,1)}&=\int_{t_n}^{t_{n+1}}\E\left[\left(\bigl(g_k\cdot g_\ell\bigr)(\widetilde{X}(t_n))-\bigl(g_k\cdot g_\ell\bigr)(\widetilde{X}(t))\right)\delta_{n,k,\ell}^2(t)\right]\dd t,\label{eq:epsilon_n211}\\
\epsilon_{n,k,\ell}^{(2,1,2)}&=\int_{t_n}^{t_{n+1}}\E\left[\left(\bigl(g_k\cdot g_\ell\bigr)(\widetilde{X}(t_n))-\bigl(g_k\cdot g_\ell\bigr)(\widetilde{X}(t))\right)\partial_{x_k x_\ell}^2u(T-t,\widetilde{X}(t_n))\right]\dd t,\label{eq:epsilon_n212}
\end{align}
with the auxiliary error term $\delta_{n,k,\ell}^2(t)$ defined by~\eqref{eq:delta2}.

\paragraph{\it Treatment of the error terms $\epsilon_{n,k}^{(1,1,1)}$ and $\epsilon_{n,k,\ell}^{(2,1,1)}$.}

The mapping $g_k^0$ is Lipschitz continuous on $\mathcal{D}$. In addition, the mappings $g_k$ and $g_\ell$ are bounded and Lipschiz continuous on $\mathcal{D}$, the mapping $g_k\cdot g_\ell$ is thus Lipschitz continuous on $\mathcal{D}$. Applying the Cauchy--Schwarz inequality, one has
\[
|\epsilon_{n,k}^{(1,1,1)}|+|\epsilon_{n,k,\ell}^{(2,1,1)}|\le C\int_{t_n}^{t_{n+1}}\bigl(\E[\|\widetilde{X}(t)-\widetilde{X}(t_n)\|^2]\bigr)^{\frac12}\left(\bigl(\E[|\delta_{n,k}^{1}(t)|^2]\bigr)^{\frac12}+\bigl(\E[|\delta_{n,k,\ell}^{2}(t)|^2]\bigr)^{\frac12}\right)\dd t.
\]
Applying the regularity property~\eqref{eq:auxXtilde2} from Lemma~\ref{lem:auxXtilde} and the upper bound~\eqref{eq:bounddelta12} for $\delta_{n,k}^{1}(t)$ and $\delta_{n,k,\ell}^{2}(t)$, one obtains the upper bound
\begin{equation}\label{eq:boundepsilon_n111-211}
|\epsilon_{n,k}^{(1,1,1)}|+|\epsilon_{n,k,\ell}^{(2,1,1)}|\le C(T,\theta)\dt^2.
\end{equation}

\paragraph{\it Treatment of the error terms $\epsilon_{n,k}^{(1,1,2)}$ and $\epsilon_{n,k,\ell}^{(2,1,2)}$.}

Owing to Lemma~\ref{lem:auxXtilde}, on the interval $[t_n,t_{n+1})$ the stochastic process $\widetilde{X}_k$ is solution to the stochastic differential equation~\eqref{eq:auxXtilde1}. For a function $\rho\colon\mathcal{D}\to\R$ of class $\mathcal{C}^2$, applying the It\^o formula one has
\begin{align*}
\rho(\widetilde{X}(t))-\rho(\widetilde{X}(t_n))&=\sum_{\ell=1}^{d}\int_{t_n}^{t}\partial_{x_\ell} \rho(\widetilde{X}(s))\partial_\gamma\Phi(\widetilde{Z}_\ell(s))f_\ell^0(X_n)\dd s\\
&+\sum_{\ell=1}^{d}\int_{t_n}^{t}\partial_{x_\ell} \rho(\widetilde{X}(s))\partial_\gamma\Phi(\widetilde{Z}_\ell(s))f_\ell(X_n)\cdot \dd B(s)\\
&+\frac12\sum_{\ell_1,\ell_2=1}^{d}\int_{t_n}^{t}\partial_{x_{\ell_1}x_{\ell_2}}\rho(\widetilde{X}(s))\partial_\gamma\Phi(\widetilde{Z}_{\ell_1}(s))\partial_\gamma\Phi(\widetilde{Z}_{\ell_2}(s))f_{\ell_1}(X_n)\cdot f_{\ell_2}(X_n)\dd s\\
&+\sum_{\ell=1}^{d}\int_{t_n}^{t}\partial_{x_\ell} \rho(\widetilde{X}(s))\psi(\widetilde{Z}_\ell(s))\|f_\ell(X_n)\|^2\dd s. 
\end{align*}
The expression above is applied for $\rho=g_k^0$ to deal with the error term $\epsilon_{n,k}^{(1,1,2)}$ and for $\rho=g_k\cdot g_\ell$ to deal with the error term $\epsilon_{n,k,\ell}^{(2,1,2)}$. By a conditional expectation argument, one has
\begin{align*}
&\sum_{\ell=1}^{d}\E\left[\int_{t_n}^{t}\partial_{x_\ell} g_k^0(\widetilde{X}(s))\partial_\gamma\Phi(\widetilde{Z}_\ell(s))f_\ell(X_n)\cdot \dd B(s)\partial_{x_k}u(T-t,\widetilde{X}(t_n))\right]=0\\
&\sum_{\ell=1}^{d}\E\left[\int_{t_n}^{t}\partial_{x_\ell}\bigl(g_k\cdot g_\ell\bigr)(\widetilde{X}(s))\partial_\gamma\Phi(\widetilde{Z}_\ell(s))f_\ell(X_n)\cdot \dd B(s)\partial_{x_k x_\ell}^2u(T-t,\widetilde{X}(t_n))\right]=0.
\end{align*}
The mappings $f_k^0$ and $f_k$ are bounded. Moreover, the mappings $g_k^0$ and $g_k\cdot g_\ell$ have bounded first and second order derivatives. Finally, the mappings $\partial_\gamma\Phi$ and $\psi$ (defined by~\eqref{eq:psi}) are continuous and their growth is controlled since one has assumed that $\|\Phi\|_{2,\beta}\le\|\Phi\|_{4,\beta}<\infty$ for some $\beta\in(0,\infty)$, see~\eqref{eq:regPhi} and~\eqref{eq:regPhipsi-weak}. One thus obtains the upper bounds
\begin{align*}
|\epsilon_{n,k}^{(1,1,2)}|&\le C\int_{t_n}^{t_{n+1}}\int_{t_n}^{t}\E[e^{Ch_k(X_n)(s-t_n)+|f_k(X_n)\cdot(B(s)-B(t_n))|}|\partial_{x_k}u(T-t,\widetilde{X}(t_n))|]\dd s\dd t,\\
|\epsilon_{n,k,\ell}^{(2,1,2)}|&\le C\int_{t_n}^{t_{n+1}}\int_{t_n}^{t}\E[e^{Ch_k(X_n)(s-t_n)+|f_k(X_n)\cdot(B(s)-B(t_n))|}|\partial_{x_kx_\ell}^2u(T-t,\widetilde{X}(t_n))|]\dd s\dd t.
\end{align*}
The first-order spatial partial derivative $\partial_{x_k}u$ and the second-order spatial partial derivative $\partial_{x_kx_\ell}^2u$ of the solution $u$ to the Kolmogorov equation are bounded, see Equation~\eqref{eq:boundKolmogorov}. Applying the inequality~\eqref{eq:lemauxexpo} from Lemma~\ref{lem:auxexpo} (with a conditional expectation argument) then yields the following upper bound for the error terms $\epsilon_{n,k}^{(1,1,2)}$ and $\epsilon_{n,k}^{(2,1,2)}$ defined by~\eqref{eq:epsilon_n112} and~\eqref{eq:epsilon_n212}: there exists $C(T,\theta)\in(0,\infty)$ such that for all $n\in\{0,\ldots,N-1\}$ one has
\begin{equation}\label{eq:boundepsilon_n112-212}
\sum_{k=1}^{d}|\epsilon_{n,k}^{(1,1,2)}|+\sum_{k,\ell=1}^{d}|\epsilon_{n,k,\ell}^{(2,1,2)}|\le C(T,\theta)\dt^2.
\end{equation}

\paragraph{\it Conclusion}

Combining the upper bounds~\eqref{eq:boundepsilon_n111-211} and~\eqref{eq:boundepsilon_n112-212} with the decompositions~\eqref{eq:decomp_epsilon_n11} and~\eqref{eq:decomp_epsilon_n21} of the error terms $\epsilon_{n,k}^{(1,1)}$ and $\epsilon_{n,k,\ell}^{(2,1)}$, one obtains the inequality~\eqref{eq:lem_epsilonII} and the proof of Lemma~\ref{lem:epsilonII} is thus completed.

\end{proof}

\begin{proof}[Proof of Lemma~\ref{lem:epsilonIII}]

Let $k,\ell\in\{1,\ldots,d\}$ and $n\in\{0,\ldots,N-1\}$, and recall that the error term $\epsilon_{n,k,\ell}^{(2,3)}$ is defined by~\eqref{eq:epsilon_n23}. The mappings $f_k$ and $f_\ell$ are bounded on $\mathcal{D}$. Moreover, the second-order spatial partial derivative $\partial_{x_kx_\ell}^2u$ of the solution $u$ to the Kolmogorov equation is bounded, see Equation~\eqref{eq:boundKolmogorov}. Applying the Cauchy--Schwarz inequality and the inequality~\eqref{eq:lemDelta_nk} from Lemma~\ref{lem:Delta_nk} for the error terms $\Delta_{n,k}(t)$ and $\Delta_{n,\ell}(t)$, one then obtains
\[
|\epsilon_{n,k,\ell}^{(2,3)}|\le C(T,\theta)\int_{t_n}^{t_{n+1}}\bigl(\E[|\Delta_{n,k}(t)|^2]\bigr)^{\frac12}\bigl(\E[|\Delta_{n,\ell}(t)|^2]\bigr)^{\frac12}\dd t\le C(T,\theta)\dt^2.
\]
This concludes the proof of the error bounds~\eqref{eq:lem_epsilonIII} and the proof of Lemma~\ref{lem:epsilonIII} is thus completed.
\end{proof}

\subsection{Conclusion and proof of Theorem~\ref{theo:weak}}\label{sec:conclusionproofweak}

Theorem~\ref{theo:weak} is obtained as a straightforward consequence of Lemmas~\ref{lem:epsilonI}~\ref{lem:epsilonII}~and~\ref{lem:epsilonIII}.

\begin{proof}[Proof of Theorem~\ref{theo:weak}]
Recall the expression~\eqref{eq:weakerror} of the weak error and the decomposition~\eqref{eq:decomp_epsilon_n} of $\epsilon_n$. Combining the three error bounds~\eqref{eq:lem_epsilonI},~\eqref{eq:lem_epsilonII} and~\eqref{eq:lem_epsilonIII} from Lemmas~\ref{lem:epsilonI},~\ref{lem:epsilonII} and~\ref{lem:epsilonIII}, one obtains the following upper bound:
\[
\underset{0\le n\le N-1}\sup~|\epsilon_n|\le \sum_{k=1}^{d}\underset{0\le n\le N-1}\sup~|\epsilon_{n,k}^{(1)}|+\sum_{k,\ell=1}^{d}\underset{0\le n\le N-1}\sup~|\epsilon_{n,k,\ell}^{(2)}|+\sum_{k,\ell=1}^{d}\underset{0\le n\le N-1}\sup~|\epsilon_{n,k}^{(3)}|\le C(T,\theta)\dt^2.
\]
Summing for $n\in\{0,\ldots,N-1\}$ then yields the weak error estimate~\eqref{eq:weak} and the proof of Theorem~\ref{theo:weak} is completed.
\end{proof}

\section{Proof of Theorem~\ref{theo:stronghigherorder}}\label{sec:proof3}

This section is devoted to the proof of Theorem~\ref{theo:stronghigherorder}. In this section, $\bigl(X_n\bigr)_{0\le n\le N}$ denotes the solution to the numerical scheme~\eqref{eq:schemeX1}. Recall that the system is driven by a one-dimensional Brownian motion ($M=1$) and that Assumptions~\ref{ass:integratorDP}~and~\ref{ass:integrator-higher} are satisfied. The proof is substantially more involved than the proof of Theorem~\ref{theo:strong} presented in Section~\ref{sec:proof1}. 

This section is organized as follows. Section~\ref{sec:aux3} first provides results on auxiliary processes $\widehat{Z}_k$ and $\widehat{X}_k$. Section~\ref{sec:decomp3} presents the decomposition of the error and states auxiliary error estimates. Finally, these auxiliary error estimates are proved in Section~\ref{sec:error3}. The proof of Theorem~\ref{theo:stronghigherorder} is finally given in Section~\ref{sec:conclusionproofstronghigher}.

\subsection{Auxiliary results}\label{sec:aux3}

As in Subsection~\ref{sec:aux1}, let us define the auxiliary processes $\widehat{Z}_k$ for all $k\in\{1,\ldots,d\}$: for all $n\in\{0,\ldots,N-1\}$ and $t\in[t_n,t_{n+1})$, set
\begin{equation}\label{eq:Zhat}
\left\lbrace
\begin{aligned}
\widehat{Z}_k(t)&=\Bigl(f_k(X_n)^2(t-t_n),\widehat{\Gamma}_{k}(t),X_{n,k}\Bigr)\\
\widehat{\Gamma}_{k}(t)&=f_k^0(X_n)(t-t_n)+f_k(X_n)(B(t)-B(t_n))+\bigl(g\cdot\nabla\bigr)f_k(X_n)\int_{t_n}^{t}\bigl[B(s)-B(t_n)\bigr]\dd B(s),
\end{aligned}
\right.
\end{equation}
where we recall that $X_{n,k}$ is defined by~\eqref{eq:schemeX1}. Note that one has $\widehat{Z}_k(t)\in\R^+\times\R\times[-1,1]$ for all $t\in[0,T]$. Furthermore, for all $n\in\{0,\ldots,N\}$, one has
$\widehat{Z}_k(t_n)=\bigl(0,0,X_{n,k}\bigr)$. Observe that the process $\widehat{Z}_k$ is discontinuous at the grid times $t_n$.

Next, for all $k\in\{1,\ldots,d\}$, the auxiliary process $\widehat{X}_k$ is defined as follows: for all $n\in\{0,\ldots,N-1\}$ and $t\in[t_n,t_{n+1})$, set
\begin{equation}\label{eq:Xhat}
\widehat{X}_k(t)=\Phi\bigl(\widehat{Z}_k(t)\bigr).
\end{equation}
Note that $\widehat{X}_k(t)\in[-1,1]$ for all $t\ge 0$, owing to the condition~\eqref{eq:condPhiDP} from Assumption~\ref{ass:integrator} on the integrator $\Phi$. In addition, for all $n\in\{0,\ldots,N\}$, one has 
$\widehat{X}_k(t_n)=\Phi(\widehat{Z}_k(t_n))=\Phi(0,0,X_{n,k})=X_{n,k}$. Finally, by the definition~\eqref{eq:schemeX} of the domain preserving scheme, for all $k\in\{1,\ldots,d\}$ the process $\widehat{X}_k$ is continuous on $[0,T]$: for all $n\in\{0,\ldots,N-1\}$ one has $\underset{t\to t_{n+1}^{-}}\lim~\widehat{X}_k(t)=X_{n+1,k}=\widehat{X}_k(t_{n+1})$.

First, Lemma~\ref{lem:auxZhat} provides properties of the process $\widehat{Z}_k$, see Lemma~\ref{lem:auxZtilde} for a similar statement for the process $\widetilde{Z}_k$ related to the numerical scheme~\eqref{eq:schemeX} (see Section~\ref{sec:aux1}).
\begin{lemma}\label{lem:auxZhat}
For all $p\in[1,\infty)$ and $T\in(0,\infty)$, there exists $C_p(T)\in(0,\infty)$ such that for all $\dt\in(0,1)$, one has the increment bounds
\begin{equation}\label{eq:auxZhat1}
\underset{1\le k\le d}\sup~\underset{0\le n\le N-1}\sup~\underset{t\in[t_n,t_{n+1})}\sup~\E[\|\widehat{Z}_k(t)-\widehat{Z}_k(t_n)\|^{2p}]\le C_p(T)\dt^p,
\end{equation}
and the moment bounds
\begin{equation}\label{eq:auxZhat2}
\underset{1\le k\le d}\sup~\underset{t\in[0,T]}\sup~\E[\|\widehat{Z}_k(t)\|^{2p}]\leq C_p(T).
\end{equation}
\end{lemma}

\begin{proof}[Proof of Lemma~\ref{lem:auxZhat}]
The proof is similar to the proof of Lemma~\ref{lem:auxZtilde} given in Section~\ref{sec:aux1}.

First, let us prove the inequality~\eqref{eq:auxZhat1}. By definition~\eqref{eq:Zhat} of the auxiliary process $\widehat{Z}_k$, for all $n\in\{0,\ldots,N-1\}$, for all $k\in\{1,\ldots,d\}$, one has for all $t\in[t_n,t_{n+1})$
\[
\widehat{Z}_k(t)-\widehat{Z}_k(t_n)=\Bigl(f_k(X_n)^2(t-t_n),\widehat{\Gamma}_{k}(t),0\Bigr),
\]
with $\widehat{\Gamma}_{k}(t)$ given in~\eqref{eq:Zhat}. Therefore, one has 
almost surely 
\begin{align*}
\|\widetilde{Z}_k(t)-\widetilde{Z}_k(t_n)\|^2&\le f_k(X_n)^4(t-t_n)^2+|\widehat{\Gamma}_{k}(t)|^2\\
&\le f_k(X_n)^4(t-t_n)^2+3|f_k^0(X_n)|^2(t-t_n)^2\\
&+3|f_k(X_n)|^2|B(t)-B(t_n))|^2+3\big|\bigl(g\cdot\nabla\bigr)f_k(X_n)\big|^2\Big|\int_{t_n}^{t}\bigl[B(s)-B(t_n)\bigr]\dd B(s)\Big|^2.    
\end{align*}
Recall that $X_n\in\mathcal{D}$ almost surely owing to Proposition~\ref{propo:dpX}. Furthermore, the mappings $f_k^0$, $f_k$, $\nabla f_k$ and $g$ are bounded on $\mathcal{D}$. Therefore, for all $p\in\N$, there exists $C_p(T)\in(0,\infty)$ such that one has
\[
\E[\|\widetilde{Z}_k(t)-\widetilde{Z}_k(t_n)\|^{2p}]\le C_p(T)\Bigl(\bigl(t-t_n\bigr)^{2p}+\E[|B(t)-B(t_n)|^{2p}]+\E\left[\left|\int_{t_n}^{t}\bigl[B(s)-B(t_n)\bigr]\dd B(s)\right|^{2p}\right]\Bigr).
\]
Since $B(t)-B(t_n)$ is a Gaussian random variable, one has $\E[|B(t)-B(t_n)|^{2p}]=c_p(t-t_n)^p$, for some $c_p\in(0,\infty)$. Moreover, applying the Burkholder--Davis--Gundy inequality, there exists $C_p\in(0,\infty)$ such that
\[
\E\left[\left|\int_{t_n}^{t}\bigl[B(s)-B(t_n)\bigr]\dd B(s)\right|^{2p}\right]\le C_p\int_{t_n}^{t}\E[\big|B(s)-B(t_n)|^{2p}]\dd s\le C_p|t-t_n|^p.
\]
As a result, one obtains the inequality~\eqref{eq:auxZhat1}.

It remains to prove the inequality~\eqref{eq:auxZhat2}. Note that for all $t\in[t_n,t_{n+1})$, one has
\[
\widehat{Z}_k(t)=\widehat{Z}_k(t)-\widehat{Z}_k(t_n)+\widehat{Z}_k(t_n),
\]
with $\|\widehat{Z}_k(t_n)\|=|X_{n,k}|\le 1$, since $X_n\in\mathcal{D}$ almost surely. It suffices then to apply the inequality~\eqref{eq:auxZhat1} to obtain the upper bound~\eqref{eq:auxZhat2}.

The proof of Lemma~\ref{lem:auxZhat} is thus completed.
\end{proof}

Lemma~\ref{lem:auxexpo2} below is a variant of Lemma~\ref{lem:auxexpo}, which is adapted to the treatment of exponential moment bounds for $\widehat{\Gamma}_k(t)$, and which is applied in the proof of Lemma~\ref{lem:auxXhat} and of the error estimates. 

\begin{lemma}\label{lem:auxexpo2}
Let $a,b,c\colon\mathcal{D}\to\R^+$ be continuous functions. There exists $\dt_0\in(0,1)$ and $C\in(0,\infty)$ such that for all $\dt\in(0,\dt_0)$ and all $n\in\{0,\ldots,N-1\}$, one has 
\begin{equation}\label{eq:lemauxexpo2}
\underset{t\in[t_n,t_{n+1}]}\sup~\underset{x\in\mathcal{D}}\sup~\E\left[e^{a(x)(t-t_n)}e^{b(x)|B(t)-B(t_n)|}e^{c(x)|B(t)-B(t_n)|^2}\right]\le C.
\end{equation}
\end{lemma}

\begin{proof}[Proof of Lemma~\ref{lem:auxexpo2}]
The mappings $a,b,c$ are continuous on the compact domain $\mathcal{D}$, hence they are bounded on $\mathcal{D}$. Therefore, applying the Cauchy--Schwarz inequality, there exists $C\in(0,\infty)$ such that for all $x\in\mathcal{D}$, under the conditions $0\le t-t_n\le \dt\le 1$, one has
\[
\E[e^{a(x)(t-t_n)}e^{b(x)|B(t)-B(t_n)|}e^{c(x)|B(t)-B(t_n)|^2}]\le C\bigl(\E[e^{2b(x)|B(t)-B(t_n)|}]\bigr)^{\frac12}\bigl(\E[e^{2c(x)|B(t)-B(t_n)|^2}]\bigr)^{\frac12}.
\]
Let $Z$ be a standard centered real-valued Gaussian random variable, then $B(t)-B(t_n)$ is equal in distribution to $\sqrt{t-t_n}Z$. Furthermore, for any $\varsigma\in(0,1/2)$, one has
\[
\E[e^{\varsigma Z^2}]=\frac{1}{\sqrt{1-2\varsigma}}.
\]
On the one hand, applying the Young inequality, if $\dt\in(0,1)$ one has
\[
\E[e^{2b(x)|B(t)-B(t_n)|}]\le e^{16b(x)^2}\E[e^{\frac14|B(t)-B(t_n)|^2}]\le \frac{C}{\sqrt{1-\frac12(t-t_n)}}\le C.
\]
On the other hand, define
\[
\dt_0=\frac{1}{8\underset{x\in\mathcal{D}}\sup~c(x)+1}.
\]
If the time-step size $\dt$ satisfies the condition $\dt\in(0,\dt_0)$, and if $0\le t-t_n\le \dt<\dt_0$, one has
\[
\E[e^{2c(x)|B(t)-B(t_n)|^2}]=\frac{1}{\sqrt{1-4c(x)(t-t_n)}}\le C.
\]
Combining the inequalities above yields the inequality~\eqref{eq:lemauxexpo2} and concludes the proof of Lemma~\ref{lem:auxexpo2}.
\end{proof}

Lemma~\ref{lem:auxXhat} below is a variant of Lemma~\ref{lem:auxXtilde} 
related to the domain preserving scheme~\eqref{eq:schemeX} (see Section~\ref{sec:aux1}). First, it provides the stochastic differential equations~\eqref{eq:auxXhat1} satisfied by the auxiliary processes $\widehat{X}_k$ on each interval $[t_n,t_{n+1})$, which contains additional terms compared to~\eqref{eq:auxXtilde1} (see Lemma~\ref{lem:auxXtilde}). Moreover, it states the increment bounds~\eqref{eq:auxXhat2}, which are similar to~\eqref{eq:auxXtilde2}.

\begin{lemma}\label{lem:auxXhat}
For all $n\in\{0,\ldots,N-1\}$, for all $k\in\{1,\ldots,d\}$, the stochastic process $t\in[t_n,t_{n+1})\mapsto \widehat{X}_k(t)$ is the solution to the stochastic differential equation
\begin{equation}\label{eq:auxXhat1}
\begin{aligned}
\dd \widehat{X}_k(t)
&=\partial_\gamma\Phi(\widehat{Z}_k(t))f_k^0(X_n)\dd t+\partial_\gamma\Phi(\widehat{Z}_k(t))f_k(X_n)\dd B(t)\\
&+\partial_\gamma\Phi(\widehat{Z}_k(t))\bigl(g\cdot\nabla\bigr)f_k(X_n)\bigl[B(t)-B(t_n)\bigr]\dd B(t)\\
&+\partial_{\gamma}^2\Phi(\widehat{Z}_k(t))\bigl(g\cdot\nabla\bigr)f_k(X_n)f_k(X_n)\bigl[B(t)-B(t_n)]\dd t\\
&+\frac12\partial_{\gamma}^2\Phi(\widehat{Z}_k(t))\bigl(\bigl(g\cdot\nabla\bigr)f_k(X_n)\bigr)^2\bigl[B(t)-B(t_n)]^2\dd t\\
&+\psi(\widehat{Z}_k(t))f_k(X_n)^2\dd t.
\end{aligned}
\end{equation}
In addition, for all $T\in(0,\infty)$, there exists $C(T)\in(0,\infty)$ such that 
for all $\dt\in(0,\overline{\dt})$, one has the regularity property
\begin{equation}\label{eq:auxXhat2}
\underset{1\le k\le d}\sup~\underset{0\le n\le N-1}\sup~\underset{t\in[t_n,t_{n+1})}\sup~\E[|\widehat{X}_k(t)-X_{n,k}|^2]\leq C(T)\dt.
\end{equation}
\end{lemma}

\begin{proof}[Proof of Lemma~\ref{lem:auxXhat}]

Let $k\in\{1,\ldots,d\}$ and $n\in\{0,\ldots,N-1\}$. Note that on the interval $t\in[t_n,t_{n+1})$ the process $t\mapsto \widehat{\Gamma}_{k}(t)$ is the solution to the stochastic differential equation
\[
\dd\widehat{\Gamma}_{k}(t)=f_k^0(X_n)\dd t+\left(f_k(X_n)+ \bigl(g\cdot\nabla\bigr)f_k(X_n)\bigl(B(t)-B(t_n)\right)\dd B(t),
\]
and its quadratic variation thus satisfies
\[
\dd \langle\widehat{\Gamma}_k\rangle(t)=\left[f_k(X_n)+ \bigl(g\cdot\nabla\bigr)f_k(X_n)\bigl(B(t)-B(t_n)\right]^2 \dd t.
\]
Applying the It\^o formula, owing to the definition~\eqref{eq:Xhat} of $\widehat{X}_k(t)$, on the interval $[t_n,t_{n+1})$ one has
\begin{align*}
\dd\widehat{X}_k(t)&=\partial_s\Phi(\widehat{Z}_k(t))f_k(X_n)^2\dd t+\partial_\gamma\Phi(\widehat{Z}_k(t))\dd \widehat{\Gamma}_k(t)+\frac12\partial_\gamma^2\Phi(\widehat{Z}_k(t))\dd \langle\widehat{\Gamma}_k\rangle(t)\\
&=\partial_s\Phi(\widehat{Z}_k(t))f_k(X_n)^2\dd t\\
&+\partial_\gamma\Phi(\widehat{Z}_k(t))f_k^0(X_n)\dd t+\partial_\gamma\Phi(\widehat{Z}_k(t))f_k(X_n)\dd B(t)\\
&+\partial_\gamma\Phi(\widehat{Z}_k(t))\bigl(g\cdot\nabla\bigr)f_k(X_n)\bigl[B(t)-B(t_n)\bigr]\dd B(t)\\
&+\frac12\partial_{\gamma}^2\Phi(\widehat{Z}_k(t))\left[f_k(X_n)+ \bigl(g\cdot\nabla\bigr)f_k(X_n)\bigl(B(t)-B(t_n)\right]^2 \dd t.
\end{align*}
Finally, expanding the square in the last line above and recalling the definition~\eqref{eq:psi} of $\psi=\partial_s\Phi+\frac12\partial_\gamma^2\Phi$, one obtains the stochastic differential equation~\eqref{eq:auxXhat1}.

It remains to prove the increment bounds~\eqref{eq:auxXhat1}. Using the SDE~\eqref{eq:auxXhat1}, integrating from $t_n$ to $t\in[t_n,t_{n+1})$ and recalling that $\widehat{X}_{k}(t_n)=X_{n,k}$, one has  
\begin{equation*}
\begin{aligned}
\widehat{X}_k(t)-X_{n,k}&=\widehat{X}_k(t)-\widehat{X}_k(t_n)\\
&=\int_{t_n}^t\partial_\gamma\Phi(\widehat{Z}_k(s))f_k^0(X_n)\dd s+\int_{t_n}^t\partial_\gamma\Phi(\widehat{Z}_k(s))f_k(X_n)\dd B(s)\\
&+\int_{t_n}^t\partial_\gamma\Phi(\widehat{Z}_k(s))\bigl(g\cdot\nabla\bigr)f_k(X_n)\bigl[B(s)-B(t_n)\bigr]\dd B(s)\\
&+\int_{t_n}^t\partial_{\gamma}^2\Phi(\widehat{Z}_k(s))\bigl(g\cdot\nabla\bigr)f_k(X_n)f_k(X_n)\bigl[B(s)-B(t_n)]\dd s\\
&+\frac12\int_{t_n}^t\partial_{\gamma}^2\Phi(\widehat{Z}_k(s))\bigl(\bigl(g\cdot\nabla\bigr)f_k(X_n)\bigr)^2\bigl[B(s)-B(t_n)]^2\dd s\\
&+\int_{t_n}^t\psi(\widehat{Z}_k(s))\|f_k(X_n)\|^2\dd s.
\end{aligned}
\end{equation*}
Applying the Cauchy--Schwarz inequality and the It\^o isometry formula, and recalling that the mappings $f_k^0$, $f_k$ and $g$ are bounded on the domain $\mathcal D$, one then obtains
\[
\E[|\widehat{X}_k(t)-X_{n,k}|^2]\le C(T)\int_{t_n}^{t}\Bigl(\E[|\partial_\gamma\Phi(\widehat{Z}_k(s))|^2]+\E[|\partial^2_{\gamma}\Phi(\widehat{Z}_k(s))|^2]+\E[|\psi(\widehat{Z}_k(s))|^2]\Bigr)\dd s.
\]
The growth of the mappings $\partial_\gamma\Phi$, $\partial_\gamma^2\Phi$  and $\psi$ follows from assuming that the integrator $\Phi$ satisfies $\|\Phi\|_{2,\beta}<\infty$ for some $\beta\in(0,\infty)$: there exists $C\in(0,\infty)$ such that one has
\begin{align*}
\E&[|\widehat{X}_k(t)-X_{n,k}|^2]\\
&\le C(T)\int_{t_n}^{t}\E[e^{C(f_k(X_n)^2}e^{C|\widehat{\Gamma}_k(s)|}]\dd s\\
&\le C(T)\int_{t_n}^{t}\E[e^{C(f_k(X_n)^2+|f_k^0(X_n)|)(s-t_n)}e^{C|f_k(X_n)||B(s)-B(t_n)|}e^{\frac{C|(g\cdot \nabla)f_k(X_n)|^2}{2}|B(s)-B(t_n)|^2}]\dd s.
\end{align*}
Assume that $\overline{\dt}$ is sufficiently small, to ensure that Lemma~\ref{lem:auxexpo2} can be applied with $\overline{\dt}\le \dt_0$. Applying a conditional expectation argument and the inequality~\eqref{eq:lemauxexpo2} from Lemma~\ref{lem:auxexpo2}, if $\dt\in(0,\overline{\dt})$ one has the upper bounds
\[
\E[|\widehat{X}_k(t)-X_{n,k}|^2]\le C(T)(t-t_n)\leq C(T)\dt.
\]
The proof of Lemma~\ref{lem:auxXhat} is thus completed.
\end{proof}

\subsection{Decomposition of the error}\label{sec:decomp3}

Compared with the error analysis presented in Section~\ref{sec:error1} for the general domain preserving scheme~\eqref{eq:schemeX}, in order to exhibit the first-order mean-square convergence in the case of the stochastic differential equation~\eqref{eq:sdeNagumog} driven by one dimensional Brownian motion, more effort is necessary, and the decomposition of the error is more delicate.

For all $n\in\{0,\ldots,N-1\}$, set 
\begin{align}
e_n&=X_n-X(t_n),\label{eq:def_en}\\
e_{n,k}&=X_{n,k}-X_k(t_n),\qquad \forall k\in\{1,\ldots,d\}.\label{eq:def_enk}
\end{align}
Note that $e_0=0$ and $e_{0,k}=0$ for all $k\in\{1,\ldots,d\}$. The objective of this section is to obtain the decomposition~\eqref{eq:decomp_enk} for the error terms $e_{n,k}$. This requires first to analyze separately the components $X_k(t_n)$ and $X_{n,k}$ of the exact and numerical solutions respectively.

On the one hand, the exact solution is solution to the stochastic differential equation~\eqref{eq:sdeNagumog} with $M=1$. For all $k\in\{1,\ldots,d\}$ and for all $n\in\{1,\ldots,N\}$, one has 
\begin{align*}
X_k(t_n)-x_{0,k}&=\int_{0}^{t_n}g_k^0(X(t))\dd t+\int_{0}^{t_n}g_k(X(t))\dd B(t)\\
&=\sum_{j=0}^{n-1}\int_{t_j}^{t_{j+1}}g_k^0(X(t))\dd t+\sum_{j=0}^{n-1}\int_{t_j}^{t_{j+1}}g_k(X(t))\dd B(t)\\
&=\sum_{j=0}^{n-1}\int_{t_j}^{t_{j+1}}g_k^0(X(t_j))\dd t+\varepsilon_{n,k}^{(1,1)}\\
&+\sum_{j=0}^{n-1}\int_{t_j}^{t_{j+1}}\Bigl[g_k(X(t_j))+\bigl(g\cdot\nabla\bigr)g_k(X(t_j))\bigl[B(t)-B(t_j)\bigr]\Bigr]\dd B(t)+\varepsilon_{n,k}^{(1,2)},
\end{align*}
where the error terms $\varepsilon_{n,k}^{(1,1)}$ and $\varepsilon_{n,k}^{(1,2)}$ are defined as
\begin{align}
\varepsilon_{n,k}^{(1,1)}&=\sum_{j=0}^{n-1}\int_{t_j}^{t_{j+1}}\bigl[g_k^0(X(t))-g_k^0(X(t_j))\bigr]\dd t,\label{eq:def_varepsnk11}\\
\varepsilon_{n,k}^{(1,2)}&=\sum_{j=0}^{n-1}\int_{t_j}^{t_{j+1}}\Bigl[g_k(X(t))-g_k(X(t_j))-\bigl(g\cdot\nabla\bigr)g_k(X(t_j))\bigl[B(t)-B(t_j)\bigr]\Bigr]\dd B(t).\label{eq:def_varepsnk12}
\end{align}

On the other hand, the numerical solution is given by~\eqref{eq:schemeX1}. Recall that the mappings $g_k^0$ and $g_k$ are related to $f_k^0$ and $f_k$ by~\eqref{eq:sigmaF}. Moreover, the integrator $\Phi$ is assumed to satisfy Assumption~\ref{ass:integrator-higher}. Applying a telescoping sum argument and the stochastic differential equation~\eqref{eq:auxXhat1} for the auxiliary process $\widehat{X}_k$ on each interval $[t_j,t_{j+1})$, for all $n\in\{1,\ldots,N\}$, one has
\begin{align*}
X_{n,k}-x_{0,k}&=\sum_{j=0}^{n-1}\bigl(X_{j+1,k}-X_{j,k}\bigr)=\sum_{j=0}^{n-1}\bigl(\widehat{X}_{k}(t_{j+1})-\widehat{X}_{k}(t_j)\bigr)\\
&=\sum_{j=0}^{n-1}\int_{t_j}^{t_{j+1}}g_k^0(X_j)\dd t+\varepsilon_{n,k}^{(2,1)}\\
&+\sum_{j=0}^{n-1}\int_{t_j}^{t_{j+1}}\Bigl[g_k(X_j)+\sigma'(X_{j,k})g_k(X_j)f_k(X_j)\bigl[B(t)-B(t_j)\bigr]\Bigr]\dd B(t)+\varepsilon_{n,k}^{(2,2)}\\
&+\sum_{j=0}^{n-1}\int_{t_j}^{t_{j+1}}\sigma(X_{j,k})\bigl(g\cdot\nabla\bigr) f_k(X_j)\bigl[B(t)-B(t_j)\bigr]\dd B(t)+\varepsilon_{n,k}^{(2,3)}\\
&+\varepsilon_{n,k}^{(2,4)}+\varepsilon_{n,k}^{(2,5)}+\varepsilon_{n,k}^{(2,6)},
\end{align*}
where the error terms $\varepsilon_{n,k}^{(2,1)}$, $\varepsilon_{n,k}^{(2,2)}$ and $\varepsilon_{n,k}^{(2,3)}$ are given by
\begin{align}
\varepsilon_{n,k}^{(2,1)}&=\sum_{j=0}^{n-1}\int_{t_j}^{t_{j+1}}\bigl[\partial_\gamma\Phi(\widehat{Z}_k(t))-\sigma(X_{j,k})\bigr]f_k^0(X_j)\dd t\nonumber\\
&=\sum_{j=0}^{n-1}\int_{t_j}^{t_{j+1}}\bigl[\partial_\gamma\Phi(\widehat{Z}_k(t))-\partial_\gamma\Phi(\widehat{Z}_k(t_j))\bigr]f_k^0(X_j)\dd t,\label{eq:def_varepsnk21}\\
\varepsilon_{n,k}^{(2,2)}&=\sum_{j=0}^{n-1}\int_{t_j}^{t_{j+1}}\Bigl[\partial_\gamma\Phi(\widehat{Z}_k(t))f_k(X_j)-g_k(X_j)-\sigma'(X_{j,k})g_k(X_j)f_k(X_j)\bigl[B(t)-B(t_j)\bigr]\Bigr]\dd B(t)\nonumber\\
&=\sum_{j=0}^{n-1}\int_{t_j}^{t_{j+1}}f_k(X_j)\Bigl[ \partial_\gamma\Phi(\widehat{Z}_k(t))-\partial_\gamma\Phi(\widehat{Z}_k(t_j))-\partial_{\gamma}^2\Phi(\widehat{Z}_k(t_j))f_k(X_j)\bigl[B(t)-B(t_j)\bigr]\Bigr]\dd B(t),\label{eq:def_varepsnk22}\\
\varepsilon_{n,k}^{(2,3)}&=\sum_{j=0}^{n-1}\int_{t_j}^{t_{j+1}}\Bigl[\partial_\gamma\Phi(\widehat{Z}_k(t))-\sigma(X_{j,k})\Bigr]\bigl(g\cdot\nabla\bigr) f_k(X_j)\bigl[B(t)-B(t_j)\bigr]\dd B(t)\nonumber\\
&=\sum_{j=0}^{n-1}\int_{t_j}^{t_{j+1}}\Bigl[\partial_\gamma\Phi(\widehat{Z}_k(t))-\partial_\gamma\Phi(\widehat{Z}_k(t_j))\Bigr]\bigl(g\cdot \nabla\bigr) f_k(X_j)\bigl[B(t)-B(t_j)\bigr]\dd B(t),\label{eq:def_varepsnk23}
\end{align}
and the error terms $\varepsilon_{n,k}^{(2,4)}$, $\varepsilon_{n,k}^{(2,5)}$ and $\varepsilon_{n,k}^{(2,6)}$ are defined as
\begin{align}
\varepsilon_{n,k}^{(2,4)}&=\sum_{j=0}^{n-1}\int_{t_j}^{t_{j+1}}\partial_{\gamma}^2\Phi(\widehat{Z}_k(t))\bigl(g\cdot\nabla\bigr)f_k(X_j)f_k(X_j)\bigl[B(t)-B(t_j)\bigr]\dd t,\label{eq:def_varepsnk24}\\
\varepsilon_{n,k}^{(2,5)}&=\frac12\sum_{j=0}^{n-1}\int_{t_j}^{t_{j+1}}\partial_{\gamma}^2\Phi(\widehat{Z}_k(t))\bigl(\bigl(g\cdot\nabla\bigr)f_k(X_j)\bigr)^2\bigl[B(t)-B(t_j)\bigr]^2\dd t,\label{eq:def_varepsnk25}\\
\varepsilon_{n,k}^{(2,6)}&=\sum_{j=0}^{n-1}\int_{t_j}^{t_{j+1}}\bigl[\psi(\widehat{Z}_k(t))-\psi(\widehat{Z}_k(t_j))\bigr]f_k(X_j)^2\dd t.\label{eq:def_varepsnk26}
\end{align}
Note that the condition~\eqref{eq:condPhi1} from Assumption~\ref{ass:integrator} is applied to obtain the expression~\eqref{eq:def_varepsnk21} for the error term $\varepsilon_{n,k}^{(2,1)}$. To obtain the expression~\eqref{eq:def_varepsnk22} for the error term $\varepsilon_{n,k}^{(2,2)}$, one also applies the condition~\eqref{eq:condPhi-higher} from Assumption~\ref{ass:integrator-higher}. Finally, the condition~\eqref{eq:condPhi2} is applied to obtain the expression~\eqref{eq:def_varepsnk26} for the error term $\varepsilon_{n,k}^{(2,6)}$.

Next, for all $k\in\{1,\ldots,d\}$, observe that due to the expression $g_k(x)=f_k(x)\sigma(x_k)$, see~\eqref{eq:sigmaF}, one has the following property: for all $x\in\mathcal{D}$
\[
\bigl(g_k\cdot\nabla\bigr)g_k(x)=\sigma'(x_k)g_k(x)f_k(x)+\sigma(x_k)\bigl(g\cdot\nabla\bigr)f_k(x).
\]
Applying that property, the following decomposition of the error term $e_{n,k}$ is obtained: for all $n\in\{1,\ldots,N\}$, one has 
\begin{equation}\label{eq:decomp_enk}
e_{n,k}=X_{n,k}-X_{k}(t_n)=-\sum_{\iota=1}^{2}\varepsilon_{n,k}^{(1,\iota)}+\sum_{\iota=1}^{6}\varepsilon_{n,k}^{(2,\iota)}+\sum_{\iota=1}^{3}\varepsilon_{n,k}^{(3,\iota)},
\end{equation}
where the error terms $\varepsilon_{n,k}^{(3,1)}$, $\varepsilon_{n,k}^{(3,2)}$ and $\varepsilon_{n,k}^{(3,2)}$ are defined as
\begin{align}
\varepsilon_{n,k}^{(3,1)}&=\sum_{j=0}^{n-1}\int_{t_j}^{t_{j+1}}\Bigl[g_k^0(X_j)-g_k^0(X(t_j))\Bigr]\dd t,\label{eq:def_varepsnk31}\\
\varepsilon_{n,k}^{(3,2)}&=\sum_{j=0}^{n-1}\int_{t_j}^{t_{j+1}}\Bigl[g_k(X_j)-g_k(X(t_j))\Bigr]\dd B(t),\label{eq:def_varepsnk32}\\
\varepsilon_{n,k}^{(3,3)}&=\sum_{j=0}^{n-1}\int_{t_j}^{t_{j+1}}\Bigl[\bigl(g\cdot \nabla\bigr)g_k(X_j)-\bigl(g\cdot\nabla\bigr)g_k(X(t_j))\Bigr]\bigl[B(t)-B(t_j)\bigr]\dd B(t).\label{eq:def_varepsnk33}
\end{align}

Theorem~\ref{theo:stronghigherorder} is then a straightforward corollary of Lemmas~\ref{lem:varepsilon1},~\ref{lem:varepsilon2} and~\ref{lem:varepsilon3} stated below on the auxiliary error terms appearing in the right-hand side of the decomposition~\eqref{eq:decomp_enk}.

\begin{lemma}\label{lem:varepsilon1}
There exists $C(T)\in(0,\infty)$ such that
\begin{equation}\label{eq:lem_varepsilon1}
\sum_{\iota=1}^{2}\sum_{k=1}^{d}\underset{1\le n\le N}\sup~\E[|\varepsilon_{n,k}^{(1,\iota)}|^2]\le C(T)\dt^2.
\end{equation}
\end{lemma}

\begin{lemma}\label{lem:varepsilon2}
There exists $C(T)\in(0,\infty)$ such that
\begin{equation}\label{eq:lem_varepsilon2}
\sum_{\iota=1}^{6}\sum_{k=1}^{d}\underset{1\le n\le N}\sup~\E[|\varepsilon_{n,k}^{(2,\iota)}|^2]\le C(T)\dt^2.
\end{equation}
\end{lemma}

\begin{lemma}\label{lem:varepsilon3}
There exists $C(T)\in(0,\infty)$ such that for all $n\in\{1,\ldots,N\}$ one has
\begin{equation}\label{eq:lem_varepsilon3}
\sum_{\iota=1}^{3}\sum_{k=1}^{d}\E[|\varepsilon_{n,k}^{(3,\iota)}|^2]\le C(T)\dt \sum_{j=0}^{n-1}\E[\|e_j\|^2].
\end{equation}
\end{lemma}
The proofs of Lemmas~\ref{lem:varepsilon1},~\ref{lem:varepsilon2} and~\ref{lem:varepsilon3} are postponed to Section~\ref{sec:error3}.

\subsection{Proof of Lemmas~\ref{lem:varepsilon1},~\ref{lem:varepsilon2} and~\ref{lem:varepsilon3}}\label{sec:error3}

In this section, we prove Lemmas~\ref{lem:varepsilon1},~\ref{lem:varepsilon2} and~\ref{lem:varepsilon3}.

\subsubsection{Proof of Lemma~\ref{lem:varepsilon1}}

The objective of this section is to prove the upper bounds~\eqref{eq:lem_varepsilon1} for the error terms $\varepsilon_{n,k}^{(1,1)}$ and $\varepsilon_{n,k}^{(1,2)}$, defined by~\eqref{eq:def_varepsnk11} and~\eqref{eq:def_varepsnk12}, related to the exact solution. The proof is divided into two parts where each error term is treated separately.

\begin{proof}[Proof of the upper bound~\eqref{eq:lem_varepsilon1} for the error term $\varepsilon_{n,k}^{(1,1)}$]

As a preliminary, a decomposition of the error term $\varepsilon_{n,k}^{(1,1)}$ defined by~\eqref{eq:def_varepsnk11} is provided.

Let $k\in\{1,\ldots,d\}$. Applying the It\^o formula, for all $j\in\{0,\ldots,N-1\}$ and all $t\in(t_j,t_{j+1})$ one has
\begin{align*}
g_k^0(X(t))-g_k^0(X(t_j))&=\int_{t_j}^{t}(g^0\cdot\nabla)g_k^0(X(s))\dd s\\
&+\int_{t_j}^{t}(g\cdot\nabla)g_k^0(X(s))\dd B(s)\\
&+\frac12\int_{t_j}^{t}\bigl(\nabla^2g_k^0:gg^T\bigr)(X(s))\dd s.
\end{align*}
As a result, for all $n\in\{1,\ldots,N\}$, the error term $\varepsilon_{n,k}^{(1,1)}$ is decomposed as
\begin{equation}\label{eq:decomp_varepsnk11}
\varepsilon_{n,k}^{(1,1)}=\varepsilon_{n,k}^{(1,1,1)}+\varepsilon_{n,k}^{(1,1,2)}+\varepsilon_{n,k}^{(1,1,3)},
\end{equation}
where the error terms $\varepsilon_{n,k}^{(1,1,1)}$, $\varepsilon_{n,k}^{(1,1,2)}$ and $\varepsilon_{n,k}^{(1,1,3)}$ are defined as
\begin{align}
\varepsilon_{n,k}^{(1,1,1)}
&=\sum_{j=0}^{n-1}\int_{t_j}^{t_{j+1}}\int_{t_j}^{t}(g^0\cdot\nabla)g_k^0(X(s))\dd s\dd t,\label{eq:def_varepsnk111}
\\
\varepsilon_{n,k}^{(1,1,2)}
&=\sum_{j=0}^{n-1}\int_{t_j}^{t_{j+1}}\int_{t_j}^{t}(g\cdot\nabla)g_k^0(X(s))\dd B(s)\dd t,\label{eq:def_varepsnk112}
\\
\varepsilon_{n,k}^{(1,1,3)}
&=\frac12\sum_{j=0}^{n-1}\int_{t_j}^{t_{j+1}}\int_{t_j}^{t}\bigl(\nabla^2g_k^0:gg^T\bigr)(X(s))\dd s\dd t.\label{eq:def_varepsnk113}
\end{align}

In the sequel, the error terms $\varepsilon_{n,k}^{(1,1,1)}$, $\varepsilon_{n,k}^{(1,1,2)}$ and $\varepsilon_{n,k}^{(1,1,3)}$ are treated separately. 

$\bullet$ {\bf Treatment of the error term $\varepsilon_{n,k}^{(1,1,1)}$}.

The mappings $g^0$ and $\nabla g_k^0$ are bounded on $\mathcal{D}$ and $X(s)\in\mathcal{D}$ almost surely for all $s\in[0,T]$. Therefore one obtains
\begin{align}
\E[|\varepsilon_{n,k}^{(1,1,1)}|^2]
&\le C(T)\sum_{j=0}^{n-1}\int_{t_j}^{t_{j+1}}\E[\bigl(\int_{t_j}^{t}\big|(g^0\cdot\nabla)g_k^0(X(s))\big|\dd s\bigr)^2]\dd t\nonumber\\
&\le C(T)\sum_{j=0}^{n-1}\int_{t_j}^{t_{j+1}}\dt^2\dd t\nonumber\\
&\le C(T)\dt^2.\label{eq:bound_varepsnk111}
\end{align}

$\bullet$ {\bf Treatment of the error term $\varepsilon_{n,k}^{(1,1,2)}$}.

Applying the stochastic Fubini theorem, the error term $\varepsilon_{n,k}^{(1,1,2)}$ is expressed as
\[
\varepsilon_{n,k}^{(1,1,2)}
=\sum_{j=0}^{n-1}\int_{t_j}^{t_{j+1}}(t_{j+1}-s)(g\cdot\nabla)g_k^0(X(s))\dd B(s).
\]
The mappings $g$ and $\nabla g_k^0$ are bounded on $\mathcal{D}$ and $X(s)\in\mathcal{D}$ almost surely for all $s\in[0,T]$. Applying the It\^o isometry formula, one obtains
\begin{align}
\E[|\varepsilon_{n,k}^{(1,1,2)}|^2]&=\sum_{j=0}^{n-1}\int_{t_j}^{t_{j+1}}(t_{j+1}-s)^2\E[|(g\cdot\nabla)g_k^0(X(s))|^2]\dd s\nonumber\\
&\le C(T)\dt^2.\label{eq:bound_varepsnk112}
\end{align}

$\bullet$ {\bf Treatment of the error term $\varepsilon_{n,k}^{(1,1,3)}$}.

The mappings $g$ and $\nabla^2 g_k^0$ are bounded on $\mathcal{D}$ and $X(s)\in\mathcal{D}$ almost surely for all $s\in[0,T]$. Therefore one obtains
\begin{align}
\E[|\varepsilon_{n,k}^{(1,1,3)}|^2]
&\le C(T)\sum_{j=0}^{n-1}\int_{t_j}^{t_{j+1}}\E[\bigl(\int_{t_j}^{t}\big|\bigl(\nabla^2g_k^0:gg^T\bigr)(X(s))\big|\dd s\bigr)^2]\dd t\nonumber\\
&\le C(T)\sum_{j=0}^{n-1}\int_{t_j}^{t_{j+1}}\dt^2\dd t\nonumber\\
&\le C(T)\dt^2.\label{eq:bound_varepsnk113}
\end{align}

$\bullet$ {\bf Conclusion.}

Recalling the decomposition~\eqref{eq:decomp_varepsnk11} of the error term $\varepsilon_{n,k}^{(1,1)}$ and combining the upper bounds~\eqref{eq:bound_varepsnk111},~\eqref{eq:bound_varepsnk112} and~\eqref{eq:bound_varepsnk113} for the error terms defined by~\eqref{eq:def_varepsnk111},~\eqref{eq:def_varepsnk112} and~\eqref{eq:bound_varepsnk113}, one finally obtains the following result: there exists $C(T)\in(0,\infty)$ such that
\begin{equation}\label{eq:bound_varepsnk11}
\sum_{k=1}^{d}\underset{1\le n\le N}\sup~\E[|\varepsilon_{n,k}^{(1,1)}|^2]\le C(T)\dt^2.
\end{equation}
This concludes the proof of the upper bound~\eqref{eq:lem_varepsilon1} for the error term $\varepsilon_{n,k}^{(1,1)}$.
\end{proof}

\begin{proof}[Proof of the upper bound~\eqref{eq:lem_varepsilon1} for the error term $\varepsilon_{n,k}^{(1,2)}$]

As a preliminary, a decomposition of the error term $\varepsilon_{n,k}^{(1,2)}$ defined by~\eqref{eq:def_varepsnk12} is provided.

Let $k\in\{1,\ldots,d\}$. Applying the It\^o formula, for all $j\in\{0,\ldots,N-1\}$ and all $t\in(t_j,t_{j+1})$ one has
\begin{align*}
g_k(X(t))-g_k(X(t_j))&=\int_{t_j}^{t}(g^0\cdot\nabla)g_k(X(s))\dd s\\
&+\int_{t_j}^{t}(g\cdot\nabla)g_k(X(s))\dd B(s)\\
&+\frac12\int_{t_j}^{t}\bigl(\nabla^2g_k:gg^T\bigr)(X(s))\dd s.
\end{align*}

Note that
\begin{align*}
\int_{t_j}^{t}(g\cdot\nabla)g_k(X(s))\dd B(s)
&=(g\cdot\nabla)g_k(X(t_j))\bigl[B(t)-B(t_j)]\\
&+\int_{t_j}^{t}\bigl[(g\cdot\nabla)g_k(X(s))-(g\cdot\nabla)g_k(X(t_j))\bigr]\dd B(s).
\end{align*}

As a result, for all $n\in\{1,\ldots,N\}$, the error term $\varepsilon_{n,k}^{(1,2)}$ is decomposed as
\begin{equation}\label{eq:decomp_varepsnk12}
\varepsilon_{n,k}^{(1,2)}=\varepsilon_{n,k}^{(1,2,1)}+\varepsilon_{n,k}^{(1,2,2)}+\varepsilon_{n,k}^{(1,2,3)},
\end{equation}
where the error terms $\varepsilon_{n,k}^{(1,2,1)}$, $\varepsilon_{n,k}^{(1,2,2)}$ and $\varepsilon_{n,k}^{(1,2,3)}$ are defined as
\begin{align}
\varepsilon_{n,k}^{(1,2,1)}
&=\sum_{j=0}^{n-1}\int_{t_j}^{t_{j+1}}\int_{t_j}^{t}(g^0\cdot\nabla)g_k(X(s))\dd s\dd B(t),\label{eq:def_varepsnk121}
\\
\varepsilon_{n,k}^{(1,2,2)}
&=\sum_{j=0}^{n-1}\int_{t_j}^{t_{j+1}}\int_{t_j}^{t}\bigl[(g\cdot\nabla)g_k(X(s))-(g\cdot\nabla)g_k(X(t_j))\bigr]\dd B(s)\dd B(t),
\label{eq:def_varepsnk122}
\\
\varepsilon_{n,k}^{(1,2,3)}
&=\frac12\sum_{j=0}^{n-1}\int_{t_j}^{t_{j+1}}\int_{t_j}^{t}\bigl(\nabla^2g_k:gg^T\bigr)(X(s))\dd s\dd B(t).\label{eq:def_varepsnk123}
\end{align}

In the sequel, the error terms $\varepsilon_{n,k}^{(1,2,1)}$, $\varepsilon_{n,k}^{(1,2,2)}$ and $\varepsilon_{n,k}^{(1,2,3)}$ are treated separately.

$\bullet$ {\bf Treatment of the error term $\varepsilon_{n,k}^{(1,2,1)}$}.

The mappings $g^0$ and $\nabla g_k$ are bounded on $\mathcal{D}$ and $X(s)\in\mathcal{D}$ almost surely for all $s\in[0,T]$. Applying the It\^o isometry formula, one has
\begin{align}
\E[|\varepsilon_{n,k}^{(1,2,1)}|^2]
&\le\sum_{j=0}^{n-1}\int_{t_j}^{t_{j+1}}\E[\bigl(\int_{t_j}^{t}\big|(g^0\cdot\nabla)g_k(X(s))\big|\dd s\bigr)^2]\dd t\nonumber\\
&\le C(T)\sum_{j=0}^{n-1}\int_{t_j}^{t_{j+1}}\dt^2\dd t\nonumber\\
&\le C(T)\dt^2.\label{eq:bound_varepsnk121}
\end{align}

$\bullet$ {\bf Treatment of the error term $\varepsilon_{n,k}^{(1,2,2)}$}.

Applying the It\^o isometry twice, one has
\begin{align*}
\E[|\varepsilon_{n,k}^{(1,2,2)}|^2]
&=\sum_{j=0}^{n-1}\int_{t_j}^{t_{j+1}}\E\Bigl[\Big|\int_{t_j}^{t}\bigl[(g\cdot\nabla)g_k(X(s))-(g\cdot\nabla)g_k(X(t_j))\bigr]\dd B(s)\Big|^2\Bigr]\dd t\\
&=\sum_{j=0}^{n-1}\int_{t_j}^{t_{j+1}}\int_{t_j}^{t}\E\Bigl[\big|(g\cdot\nabla)g_k(X(s))-(g\cdot\nabla)g_k(X(t_j))\big|^2\Bigr]\dd s\dd t.
\end{align*}
Note again that the mappings $g$ and $\nabla g_k$ are bounded on $\mathcal{D}$ and $X(s)\in\mathcal{D}$ almost surely for all $s\in[0,T]$.
Using the regularity property~\eqref{eq:reg_exact} of 
the mild solution $X$, one then obtains
\begin{align}
\E[|\varepsilon_{n,k}^{(1,2,2)}|^2]&\le C(T)\sum_{j=0}^{n-1}\int_{t_j}^{t_{j+1}}\int_{t_j}^{t}(s-t_j)\dd s\dd t\nonumber\\
&\le C(T)\dt^2.\label{eq:bound_varepsnk122} 
\end{align}

$\bullet$ {\bf Treatment of the error term $\varepsilon_{n,k}^{(1,2,3)}$}.

Applying the It\^o isometry property and recalling that the mappings $\nabla^2g_k$ 
and $g$ are bounded on $\mathcal D$ and that $X(s)\in\mathcal D$ almost surely for all $s\in[0,T]$, one obtains
\begin{align}
\E[|\varepsilon_{n,k}^{(1,2,3)}|^2]
&\le C(T)\sum_{j=0}^{n-1}\int_{t_j}^{t_{j+1}}\E[\bigl(\int_{t_j}^{t}\big|\bigl(\nabla^2g_k:gg^T\bigr)(X(s))\big|\dd s\bigr)^2]\dd t\nonumber\\
&\le C(T)\sum_{j=0}^{n-1}\int_{t_j}^{t_{j+1}}\dt^2\dd t\nonumber\\
&\le C(T)\dt^2.\label{eq:bound_varepsnk123}
\end{align}

$\bullet$ {\bf Conclusion.}

Recalling the decomposition~\eqref{eq:decomp_varepsnk11} of the error term $\varepsilon_{n,k}^{(1,2)}$ and combining the upper bounds~\eqref{eq:bound_varepsnk121},~\eqref{eq:bound_varepsnk122} and~\eqref{eq:bound_varepsnk123} for the error terms defined by~\eqref{eq:def_varepsnk121},~\eqref{eq:def_varepsnk122} and~\eqref{eq:bound_varepsnk123}, one thus obtains the following result: there exists $C(T)\in(0,\infty)$ such that
\begin{equation}\label{eq:bound_varepsnk12}
\sum_{k=1}^{d}\underset{1\le n\le N}\sup~\E[|\varepsilon_{n,k}^{(1,2)}|^2]\le C(T)\dt^2.
\end{equation}
This concludes the proof of the upper bound~\eqref{eq:lem_varepsilon1} for the error term $\varepsilon_{n,k}^{(1,2)}$.
\end{proof}

\begin{proof}[Proof of Lemma~\ref{lem:varepsilon1}]
It suffices to combine the upper bounds~\eqref{eq:bound_varepsnk11} and~\eqref{eq:bound_varepsnk12} for the error terms $\varepsilon_{n,k}^{(1,1)}$ and $\varepsilon_{n,k}^{(1,2)}$ defined by~\eqref{eq:def_varepsnk11} and~\eqref{eq:def_varepsnk12} to obtain the inequality~\eqref{eq:lem_varepsilon1}. This concludes the proof of Lemma~\ref{lem:varepsilon1}.
\end{proof}

\subsubsection{Proof of Lemma~\ref{lem:varepsilon2}}

The objective of this section is to prove the upper bounds~\eqref{eq:lem_varepsilon2} for the error terms $\varepsilon_{n,k}^{(2,\iota)}$ for $\iota=1,\ldots,6$, defined by~\eqref{eq:def_varepsnk21},\eqref{eq:def_varepsnk22},~\eqref{eq:def_varepsnk23},~\eqref{eq:def_varepsnk24},~\eqref{eq:def_varepsnk25} and~\eqref{eq:def_varepsnk26}, related to the numerical solution. The proof is divided into six parts where each error term is treated separately.

Before proceeding, it is convenient to state and prove the following auxiliary result which is employed several times in this section.
\begin{lemma}\label{lem:auxexpo3}
Let $\rho\colon\R^+\times\R\times[-1,1]\to[-1,1]$ be a continuous mapping, which satisfies the condition $\|\rho\|_{0,\beta}<\infty$ for some $\beta\in(0,\infty)$. If $\overline{\dt}$ is sufficiently small, there exists $C\in(0,\infty)$ such that for all $\dt\in(0,\overline{\dt})$, all $j\in\{0,\ldots,N-1\}$ and all $t\in(t_j,t_{j+1})$ one has
\begin{equation}\label{eq:lemauxexpo3}
\int_{t_j}^{t}\bigl(\E[|\rho(\widehat{Z}_k(s))|^{2}]\bigr)^{\frac12}\dd s+\int_{t_j}^{t}\bigl(\E[|\rho(\widehat{Z}_k(s))|^{4}]\bigr)^{\frac14}\dd s\le C(t-t_j).
\end{equation}
\end{lemma}

\begin{proof}[Proof of Lemma~\ref{lem:auxexpo3}]

Let $p\in\{1,2\}$. Due to the condition $\|\rho\|_{0,\beta}<\infty$, there exists $C\in(0,\infty)$ such that for all $s\in(t_j,t_{j+1})$ one has 
\begin{align*}
\E[|\partial_{\gamma}^2\Phi (\widehat{Z}_k(s))|^{2p}]
&\le C\E[e^{Cpf_k(X_j))^2(s-t_j)}e^{Cp|\widehat{\Gamma}_k(s)|}]\\
&\le C\bigl(\E[e^{Cp(f_k(X_j))^2+|f_k^0(X_j)|)(s-t_j)}e^{Cp|f_k(X_j)||B(s)-B(t_j)|}e^{\frac{Cp|(g\cdot \nabla)f_k(X_j)|^2}{2}|B(s)-B(t_j)|^2}].
\end{align*}
If $\overline{\dt}$ is sufficiently small, Lemma~\ref{lem:auxexpo2} can be applied with $a=Cp(f_k^2+|f_k^0|)$, $b=Cp|f_k|$ and $c=\frac{Cp|(g\cdot \nabla)f_k|^2}{2}$. As a result, there exists $C_k\in(0,\infty)$ such that
\[
\E[|\partial_{\gamma}^2\Phi (\widehat{Z}_k(s))|^{2p}]\le C_k.
\]
Integrating for $s\in(t_j,t)$ with $t\in(t_j,t_{j+1})$ yields the inequality~\eqref{eq:lemauxexpo3} with $C=\sup~\{C_k;~1\le k\le d\}$ and the proof of Lemma~\ref{lem:auxexpo3} is thus completed.
\end{proof}
Lemma~\ref{lem:auxexpo3} is applied below with $\rho=\partial_\gamma^2\Phi$, $\rho=\partial_\gamma^3\Phi$, $\rho=\psi=\partial_s\Phi+\frac12\partial_\gamma^2\Phi$, $\rho=\partial_\gamma\psi=\partial_s\partial_\gamma\Phi+\frac12\partial_\gamma^3\Phi$, $\rho=\partial_\gamma^2\psi=\partial_s\partial_\gamma^2\Phi+\frac12\partial_\gamma^4\Phi$ and $\rho=\partial_s\psi+\frac12\partial_\gamma^2\psi=\partial_s^2\Phi+\partial_s\partial_\gamma^2\Phi+\frac14\partial_\gamma^4\Phi$. The growth of all those functions is controlled assuming that there exists $\beta\in(0,\infty)$ such that $\|\Phi\|_{4,\beta}<\infty$, see~\eqref{eq:regPhi}.

\begin{proof}[Proof of the upper bound~\eqref{eq:lem_varepsilon2} for the error term $\varepsilon_{n,k}^{(2,1)}$]

As a preliminary, a decomposition of the error term $\varepsilon_{n,k}^{(2,1)}$ defined by~\eqref{eq:def_varepsnk21} is provided.

Let $k\in\{1,\ldots,d\}$. The mapping $\partial_\gamma\Phi$ is of class $\mathcal{C}^2$ and the growth of its first and second order derivatives is controlled since one has assumed that $\|\Phi\|_{3,\beta}<\infty$ for some $\beta\in(0,\infty)$, see~\eqref{eq:regPhi}.

Applying the It\^o formula for $t\mapsto \rho(\widehat{Z}_k(t))$ with $\rho=\partial_\gamma\Phi$, for all $j\in\{0,\ldots,N-1\}$ and all $t\in(t_j,t_{j+1})$, the error term $\varepsilon_{n,k}^{(2,1)}$ is decomposed as
\begin{equation}\label{eq:decomp_varepsnk21}
\varepsilon_{n,k}^{(2,1)}=\varepsilon_{n,k}^{(2,1,1)}+\varepsilon_{n,k}^{(2,1,2)}+\varepsilon_{n,k}^{(2,1,3)}+\varepsilon_{n,k}^{(2,1,4)}+\varepsilon_{n,k}^{(2,1,5)}+\varepsilon_{n,k}^{(2,1,6)},
\end{equation}
where the error terms $\varepsilon_{n,k}^{(2,1,\iota)}$ for $\iota=1,\ldots,6$ are defined as
\begin{align}
\varepsilon_{n,k}^{(2,1,1)}
&=\sum_{j=0}^{n-1}\int_{t_j}^{t_{j+1}}\int_{t_j}^{t}\partial_{\gamma}^2\Phi (\widehat{Z}_k(s)) \dd s f_k^0(X_j)\dd t,\label{eq:def_varepsnk211}
\\
\varepsilon_{n,k}^{(2,1,2)}
&=\sum_{j=0}^{n-1}\int_{t_j}^{t_{j+1}} \int_{t_j}^{t}\partial_{\gamma}^2 \Phi(\widehat{Z}_k(s))f_k(X_j)\dd B(s)f_k^0(X_j)\dd t,\label{eq:def_varepsnk212}
\\
\varepsilon_{n,k}^{(2,1,3)}
&=\sum_{j=0}^{n-1}\int_{t_j}^{t_{j+1}} \int_{t_j}^{t}\partial_{\gamma}^2\Phi(\widehat{Z}_k(s))(g\cdot\nabla)f_k(X_j)[B(s)-B(t_j)]\dd B(s) f_k^0(X_j)\dd t,\label{eq:def_varepsnk213}
\\
\varepsilon_{n,k}^{(2,1,4)}
&=\sum_{j=0}^{n-1}\int_{t_j}^{t_{j+1}} \int_{t_j}^{t}\partial_{\gamma}^3\Phi(\widehat{Z}_k(t))f_k(X_j)(g\cdot\nabla)f_k(X_j)(B(s)-B(t_j))\dd s f_k^0(X_j)\dd t,\label{eq:def_varepsnk214}
\\
\varepsilon_{n,k}^{(2,1,5)}
&=\frac12\sum_{j=0}^{n-1}\int_{t_j}^{t_{j+1}}\int_{t_j}^{t}\partial_{\gamma}^3\Phi(\widehat{Z}_k(s))[(g\cdot\nabla)f_k(X_j)]^2[B(s)-B(t_j)]^2\dd s f_k^0(X_j)\dd t,\label{eq:def_varepsnk215}
\\
\varepsilon_{n,k}^{(2,1,6)}
&=\sum_{j=0}^{n-1}\int_{t_j}^{t_{j+1}}\int_{t_j}^{t}\partial_\gamma\psi(\widehat{Z}_k(s))f_k(X_j)^2\dd s f_k^0(X_j)\dd t.\label{eq:def_varepsnk216}
\end{align}

In the sequel, the error terms $\varepsilon_{n,k}^{(2,1,\iota)}$ for $\iota=1,\ldots,6$ are treated separately.

$\bullet$ {\bf Treatment of the error term $\varepsilon_{n,k}^{(2,1,1)}$}.

Recall that $X_j\in\mathcal{D}$ almost surely, and that the mapping $f_k^0$ is bounded on $\mathcal{D}$. Therefore applying the Cauchy--Schwarz inequality, one has
\begin{align*}
\E[|\varepsilon_{n,k}^{(2,1,1)}|^2]&\le C(T)\sum_{j=0}^{n-1}\int_{t_j}^{t_{j+1}}\E[\bigl(\int_{t_j}^{t}|\partial_{\gamma}^2\Phi (\widehat{Z}_k(s))|\dd s\bigr)^2]\dd t\\
&\le C(T)\sum_{j=0}^{n-1}\int_{t_j}^{t_{j+1}}(t-t_j)\int_{t_j}^{t}\E[|\partial_{\gamma}^2\Phi (\widehat{Z}_k(s))|^2]\dd s\dd t.
\end{align*}
Applying the inequality~\eqref{eq:lemauxexpo3} from Lemma~\ref{lem:auxexpo3} with $\rho=\partial_\gamma^2\Phi$, one then obtains the upper bounds
\begin{align}
\E[|\varepsilon_{n,k}^{(2,1,1)}|^2]&\le C(T)\sum_{j=0}^{n-1}\int_{t_j}^{t_{j+1}}(t-t_j)^2\dd t\nonumber\\
&\le C(T)\dt^2.\label{eq:bound_varepsnk211}
\end{align}

$\bullet$ {\bf Treatment of the error term $\varepsilon_{n,k}^{(2,1,2)}$}.

Applying the stochastic Fubini theorem, one has
\begin{align*}
\varepsilon_{n,k}^{(2,1,2)}&=\sum_{j=0}^{n-1}\int_{t_j}^{t_{j+1}} \int_{t_j}^{t}\partial_{\gamma}^2 \Phi(\widehat{Z}_k(s))f_k(X_j)f_k^0(X_j)\dd t\dd B(s)\\
&=\sum_{j=0}^{n-1}\int_{t_j}^{t_{j+1}} (t_{j+1}-s)\partial_{\gamma}^2 \Phi(\widehat{Z}_k(s))f_k(X_j)f_k^0(X_j)\dd B(s).
\end{align*}
Recall that $X_j\in\mathcal{D}$ almost surely, and that the mappings $f_k$ and $f_k^0$ are bounded on $\mathcal{D}$.
Applying the It\^o isometry formula, one has
\begin{align*}
\E[|\varepsilon_{n,k}^{(2,1,2)}|^2]&=\sum_{j=0}^{n-1}\int_{t_j}^{t_{j+1}}(t_{j+1}-s)^2\E[|\partial_{\gamma}^2 \Phi(\widehat{Z}_k(s))f_k(X_j)f_k^0(X_j)|^2]\dd s\\
&\le C\dt^2\sum_{j=0}^{n-1}\int_{t_j}^{t_{j+1}}\E[|\partial_{\gamma}^2 \Phi(\widehat{Z}_k(s))|^2]\dd s.
\end{align*}
Applying the inequality~\eqref{eq:lemauxexpo3} from Lemma~\ref{lem:auxexpo3}  with $\rho=\partial_\gamma^2\Phi$, for $\dt\in(0,\overline{\dt})$ one thus obtains the upper bound
\begin{equation}\label{eq:bound_varepsnk212}
\E[|\varepsilon_{n,k}^{(2,1,2)}|^2]\le C(T)\dt^2.
\end{equation}

$\bullet$ {\bf Treatment of the error term $\varepsilon_{n,k}^{(2,1,3)}$}.

Applying the stochastic Fubini theorem, one has
\begin{align*}
\varepsilon_{n,k}^{(2,1,3)}&=\sum_{j=0}^{n-1}\int_{t_j}^{t_{j+1}} \int_{t_j}^{t}\partial_{\gamma}^2\Phi(\widehat{Z}_k(s))(g\cdot\nabla)f_k(X_j)[B(s)-B(t_j)]f_k^0(X_j)\dd t\dd B(s)\\
&=\sum_{j=0}^{n-1}\int_{t_j}^{t_{j+1}} (t_{j+1}-s)\partial_{\gamma}^2\Phi(\widehat{Z}_k(s))(g\cdot\nabla)f_k(X_j)[B(s)-B(t_j)]f_k^0(X_j)\dd B(s).
\end{align*}
Recall that $X_j\in\mathcal{D}$ almost surely, and that the mappings $\bigl(g\cdot\nabla\bigr)f_k$ and $f_k^0$ are bounded on $\mathcal{D}$.
Applying the It\^o isometry formula and the Cauchy--Schwarz inequality, one obtains
\begin{align*}
\E[|\varepsilon_{n,k}^{(2,1,3)}|^2]&=\sum_{j=0}^{n-1}\int_{t_j}^{t_{j+1}}(t_{j+1}-s)^2\E[|\partial_{\gamma}^2\Phi(\widehat{Z}_k(s))(g\cdot\nabla)f_k(X_j)[B(s)-B(t_j)]f_k^0(X_j)|^2]\dd s\\
&\le C\dt^2\sum_{j=0}^{n-1}\int_{t_j}^{t_{j+1}}\E[|\partial_{\gamma}^2\Phi(\widehat{Z}_k(s))|^2|B(s)-B(t_j)|^2]\dd s\\
&\le C\dt^2\sum_{j=0}^{n-1}\int_{t_j}^{t_{j+1}}\bigl(\E[|\partial_{\gamma}^2\Phi(\widehat{Z}_k(s))|^4]\bigr)^{\frac12}\bigl(\E[|B(s)-B(t_j)|^4]\bigr)^{\frac12}\dd s.
\end{align*}
Applying the inequality~\eqref{eq:lemauxexpo3} from Lemma~\ref{lem:auxexpo3} with $\rho=\partial_\gamma^2\Phi$, for $\dt\in(0,\overline{\dt})$ one obtains the upper bounds
\begin{align}
\E[|\varepsilon_{n,k}^{(2,1,3)}|^2]&\le C\dt^2\sum_{j=0}^{n-1}\int_{t_j}^{t_{j+1}}\bigl(\E[|\partial_{\gamma}^2\Phi(\widehat{Z}_k(s))|^4]\bigr)^{\frac12}(s-t_j)\dd s\nonumber\\
&\le C\dt^3\sum_{j=0}^{n-1}\int_{t_j}^{t_{j+1}}\bigl(\E[|\partial_{\gamma}^2\Phi(\widehat{Z}_k(s))|^4]\bigr)^{\frac12}\dd s\nonumber\\
&\le C(T)\dt^3.\label{eq:bound_varepsnk213}
\end{align}

$\bullet$ {\bf Treatment of the error term $\varepsilon_{n,k}^{(2,1,4)}$}.

Recall that $X_j\in\mathcal{D}$ almost surely, and that the mappings $f_k$, $\bigl(g\cdot\nabla\bigr)f_k$ and $f_k^0$ are bounded on $\mathcal{D}$. Applying the Cauchy--Schwarz inequality, one has
\begin{align*}
\E[|\varepsilon_{n,k}^{(2,1,4)}|^2]&\le C(T)\sum_{j=0}^{n-1}\int_{t_j}^{t_{j+1}}\E[\Big|\int_{t_j}^{t}\partial_{\gamma}^3\Phi(\widehat{Z}_k(t))f_k(X_j)(g\cdot\nabla)f_k(X_j)(B(s)-B(t_j))f_k^0(X_j)\dd s\Big|^2]\dd t\\
&\le C(T)\sum_{j=0}^{n-1}\int_{t_j}^{t_{j+1}}(t-t_j)\int_{t_j}^{t}\E[|\partial_{\gamma}^3\Phi(\widehat{Z}_k(t))|^2|B(s)-B(t_j)|^2]\dd s\dd t\\
&\le C(T)\sum_{j=0}^{n-1}\int_{t_j}^{t_{j+1}}(t-t_j)\int_{t_j}^{t}\bigl(\E[|\partial_{\gamma}^3\Phi(\widehat{Z}_k(t))|^4\bigr)^{\frac12}\bigl(\E[|B(s)-B(t_j)|^4]\bigr)^{\frac12}\dd s\dd t.
\end{align*}
Applying the inequality~\eqref{eq:lemauxexpo3} from Lemma~\ref{lem:auxexpo3} with $\rho=\partial_\gamma^3\Phi$, one obtains the upper bounds
\begin{align}
\E[|\varepsilon_{n,k}^{(2,1,4)}|^2]&\le C(T)\sum_{j=0}^{n-1}\int_{t_j}^{t_{j+1}}(t-t_j)^2\int_{t_j}^{t}\bigl(\E[|\partial_{\gamma}^3\Phi(\widehat{Z}_k(t))|^4\bigr)^{\frac12}\dd s\dd t\nonumber\\
&\le C(T)\sum_{j=0}^{n-1}\int_{t_j}^{t_{j+1}}(t-t_j)^3\dd t\nonumber\\
&\le C(T)\dt^3.\label{eq:bound_varepsnk214}
\end{align}

$\bullet$ {\bf Treatment of the error term $\varepsilon_{n,k}^{(2,1,5)}$}.

Recall that $X_j\in\mathcal{D}$ almost surely, and that the mappings $\bigl(g\cdot\nabla\bigr)f_k$ and $f_k^0$ are bounded on $\mathcal{D}$. Applying the Cauchy--Schwarz inequality, one has
\begin{align*}
\E[|\varepsilon_{n,k}^{(2,1,5)}|^2]&\le C(T)\sum_{j=0}^{n-1}\int_{t_j}^{t_{j+1}}\E[\Big|\int_{t_j}^{t}\partial_{\gamma}^3\Phi(\widehat{Z}_k(s))[(g\cdot\nabla)f_k(X_j)]^2[B(s)-B(t_j)]^2f_k^0(X_j)\dd s\Big|^2]\dd t\\
&\le C(T)\sum_{j=0}^{n-1}\int_{t_j}^{t_{j+1}}(t-t_j)\int_{t_j}^{t}\E[|\partial_{\gamma}^3\Phi(\widehat{Z}_k(s))|^2|B(s)-B(t_j)|^4]\dd s\dd t\\
&\le C(T)\sum_{j=0}^{n-1}\int_{t_j}^{t_{j+1}}(t-t_j)\int_{t_j}^{t}\bigl(\E[|\partial_{\gamma}^3\Phi(\widehat{Z}_k(t))|^4\bigr)^{\frac12}\bigl(\E[|B(s)-B(t_j)|^8]\bigr)^{\frac12}\dd s\dd t.
\end{align*}
Applying the inequality~\eqref{eq:lemauxexpo3} from Lemma~\ref{lem:auxexpo3} with $\rho=\partial_\gamma^3\Phi$, for $\dt\in(0,\overline{\dt})$, one obtains the upper bounds
\begin{align}
\E[|\varepsilon_{n,k}^{(2,1,5)}|^2]&\le C(T)\sum_{j=0}^{n-1}\int_{t_j}^{t_{j+1}}(t-t_j)^3\int_{t_j}^{t}\bigl(\E[|\partial_{\gamma}^3\Phi(\widehat{Z}_k(t))|^4\bigr)^{\frac12}\dd s\dd t\nonumber\\
&\le C(T)\sum_{j=0}^{n-1}\int_{t_j}^{t_{j+1}}(t-t_j)^4\dd t\nonumber\\
&\le C(T)\dt^4.\label{eq:bound_varepsnk215}
\end{align}

$\bullet$ {\bf Treatment of the error term $\varepsilon_{n,k}^{(2,1,6)}$}.

Recall that $X_j\in\mathcal{D}$ almost surely, and that the mappings $f_k$ and $f_k^0$ are bounded on $\mathcal{D}$. Applying the Cauchy--Schwarz inequality, one has
\begin{align*}
\E[|\varepsilon_{n,k}^{(2,1,6)}|^2]&\le C(T)\sum_{j=0}^{n-1}\int_{t_j}^{t_{j+1}}\E[\Big|\int_{t_j}^{t}\partial_\gamma\psi(\widehat{Z}_k(s))f_k(X_j)^2f_k^0(X_j)\dd s\Big|^2]\dd t\\
&\le C(T)\sum_{j=0}^{n-1}\int_{t_j}^{t_{j+1}}(t-t_j)\int_{t_j}^{t}\E[|\partial_\gamma\psi(\widehat{Z}_k(s))|^2]\dd s\dd t.
\end{align*}
Applying the inequality~\eqref{eq:lemauxexpo3} from Lemma~\ref{lem:auxexpo3} with $\rho=\partial_\gamma\psi$, one obtains the upper bounds
\begin{align}
\E[|\varepsilon_{n,k}^{(2,1,6)}|^2]&\le C(T)\sum_{j=0}^{n-1}\int_{t_j}^{t_{j+1}}(t-t_j)^2\dd s\dd t\nonumber\\
&\le C(T)\dt^2.\label{eq:bound_varepsnk216}
\end{align}

$\bullet$ {\bf Conclusion.}

Recalling the decomposition~\eqref{eq:decomp_varepsnk21} of the error term $\varepsilon_{n,k}^{(2,1)}$ and combining the upper bounds~\eqref{eq:bound_varepsnk211},~\eqref{eq:bound_varepsnk212},~\eqref{eq:bound_varepsnk213},~\eqref{eq:bound_varepsnk214},~\eqref{eq:bound_varepsnk215} and~\eqref{eq:bound_varepsnk216} for the error terms defined by~\eqref{eq:def_varepsnk211},~\eqref{eq:def_varepsnk212},~\eqref{eq:def_varepsnk213},~\eqref{eq:def_varepsnk214},~\eqref{eq:def_varepsnk215} and~\eqref{eq:def_varepsnk216}, one finally obtains the following result: there exists $C(T)\in(0,\infty)$ such that
\begin{equation}\label{eq:bound_varepsnk21}
\sum_{k=1}^{d}\underset{1\le n\le N}\sup~\E[|\varepsilon_{n,k}^{(2,1)}|^2]\le C(T)\dt^2.
\end{equation}
This concludes the proof of the upper bound~\eqref{eq:lem_varepsilon2} for the error term $\varepsilon_{n,k}^{(2,1)}$.
\end{proof}

\begin{proof}[Proof of the upper bound~\eqref{eq:lem_varepsilon2} for the error term $\varepsilon_{n,k}^{(2,2)}$]

As a preliminary, a decomposition of the error term $\varepsilon_{n,k}^{(2,2)}$ defined by~\eqref{eq:def_varepsnk22} is provided.

Let $k\in\{1,\ldots,d\}$. The mapping $\partial_\gamma\Phi$ is of class $\mathcal{C}^2$ and the growth of its first and second order derivatives is controlled since one has assumed that $\|\Phi\|_{3,\beta}<\infty$ for some $\beta\in(0,\infty)$, see~\eqref{eq:regPhi}.

Applying the It\^o formula for $t\mapsto \rho(\widehat{Z}_k(t))$ with $\rho=\partial_\gamma\Phi$, for all $j\in\{0,\ldots,N-1\}$ and all $t\in(t_j,t_{j+1})$, the error term $\varepsilon_{n,k}^{(2,2)}$ is decomposed as
\begin{equation}\label{eq:decomp_varepsnk22}
\varepsilon_{n,k}^{(2,2)}=\varepsilon_{n,k}^{(2,2,1)}+\varepsilon_{n,k}^{(2,2,2)}+\varepsilon_{n,k}^{(2,2,3)}+\varepsilon_{n,k}^{(2,2,4)}+\varepsilon_{n,k}^{(2,2,5)}+\varepsilon_{n,k}^{(2,2,6)},
\end{equation}
where the error terms $\varepsilon_{n,k}^{(2,2,\iota)}$ for $\iota=1,\ldots,6$ are defined as
\begin{align}
\varepsilon_{n,k}^{(2,2,1)}
&=\sum_{j=0}^{n-1}\int_{t_j}^{t_{j+1}} \int_{t_j}^{t}\partial_{\gamma}^2\Phi(\widehat{Z}_k(s)) \dd s f_k(X_j)\dd B(t),\label{eq:def_varepsnk221}
\\
\varepsilon_{n,k}^{(2,2,2)}
&=\sum_{j=0}^{n-1}\int_{t_j}^{t_{j+1}} \int_{t_j}^{t}\bigl[\partial_{\gamma}^2\Phi(\widehat{Z}_k(s))-\partial_{\gamma}^2\Phi(\widehat{Z}_k(t_j))\bigr]f_k(X_j)\dd B(s) f_k(X_j)\dd B(t),\label{eq:def_varepsnk222}
\\
\varepsilon_{n,k}^{(2,2,3)}
&=\sum_{j=0}^{n-1}\int_{t_j}^{t_{j+1}} \int_{t_j}^{t}\partial_{\gamma}^2\Phi(\widehat{Z}_k(s))(g\cdot\nabla)f_k(X_j)[B(s)-B(t_j)]\dd B(s) f_k(X_j)\dd B(t),\label{eq:def_varepsnk223}
\\
\varepsilon_{n,k}^{(2,2,4)}
&=\sum_{j=0}^{n-1}\int_{t_j}^{t_{j+1}} \int_{t_j}^{t}\partial_{\gamma}^3\Phi(\widehat{Z}_k(t))f_k(X_j)(g\cdot\nabla)f_k(X_j)(B(s)-B(t_j))\dd s f_k(X_j)\dd B(t),\label{eq:def_varepsnk224}
\\
\varepsilon_{n,k}^{(2,2,5)}
&=\sum_{j=0}^{n-1}\int_{t_j}^{t_{j+1}} \int_{t_j}^{t} \partial_{\gamma}^3\Phi(\widehat{Z}_k(s))[(g\cdot\nabla)f_k(X_j)]^2[B(s)-B(t_j)]^2\dd s f_k(X_j)\dd B(t),\label{eq:def_varepsnk225}
\\
\varepsilon_{n,k}^{(2,2,6)}
&=\sum_{j=0}^{n-1}\int_{t_j}^{t_{j+1}} \int_{t_j}^{t}\partial_\gamma\psi(\widehat{Z}_k(s))f_k(X_j)^2\dd s f_k(X_j)\dd B(t).\label{eq:def_varepsnk226}
\end{align}

In the sequel, the error terms $\varepsilon_{n,k}^{(2,2,\iota)}$ for $\iota=1,\ldots,6$ are treated separately.

$\bullet$ {\bf Treatment of the error term $\varepsilon_{n,k}^{(2,2,1)}$}.

Recall that $X_j\in\mathcal{D}$ almost surely, and that the mapping $f_k$ is bounded on $\mathcal{D}$. Applying the It\^o isometry formula and the Cauchy--Schwarz inequality, one has
\begin{align*}
\E[|\varepsilon_{n,k}^{(2,2,1)}|^2]&=\sum_{j=0}^{n-1}\int_{t_j}^{t_{j+1}} \E[\big|\int_{t_j}^{t}|\partial_{\gamma}^2\Phi(\widehat{Z}_k(s))|\dd s\big|^2f_k(X_j)^2]\dd t\\
&\le \sum_{j=0}^{n-1}\int_{t_j}^{t_{j+1}} (t-t_j)\int_{t_j}^{t}\E[|\partial_{\gamma}^2\Phi(\widehat{Z}_k(s))|^2]\dd s\dd t.
\end{align*}
Applying the inequality~\eqref{eq:lemauxexpo3} from Lemma~\ref{lem:auxexpo3} with $\rho=\partial_{\gamma}^2\Phi$, one obtains the upper bounds
\begin{align}
\E[|\varepsilon_{n,k}^{(2,2,1)}|^2]&\le C\sum_{j=0}^{n-1}\int_{t_j}^{t_{j+1}}(t-t_j)^2\dd t\nonumber\\
&\le C(T)\dt^2.\label{eq:bound_varepsnk221}
\end{align}

$\bullet$ {\bf Treatment of the error term $\varepsilon_{n,k}^{(2,2,2)}$}.

Recall that $X_j\in\mathcal{D}$ almost surely, and that the mapping $f_k$ is bounded on $\mathcal{D}$. Applying the It\^o isometry formula twice, one has
\begin{align*}
\E[|\varepsilon_{n,k}^{(2,2,2)}|^2]
&=\sum_{j=0}^{n-1}\int_{t_j}^{t_{j+1}}\E[\Big|\int_{t_j}^{t}\bigl[\partial_{\gamma}^2\Phi(\widehat{Z}_k(s))-\partial_{\gamma}^2\Phi(\widehat{Z}_k(t_j))\bigr]f_k(X_j)^2\dd B(s)\Big|^2]\dd t\\
&=\sum_{j=0}^{n-1}\int_{t_j}^{t_{j+1}}\int_{t_j}^{t}\E[\bigl[\partial_{\gamma}^2\Phi(\widehat{Z}_k(s))-\partial_{\gamma}^2\Phi(\widehat{Z}_k(t_j))\bigr]^2f_k(X_j)^4]\dd s\dd t\\
&\le C\sum_{j=0}^{n-1}\int_{t_j}^{t_{j+1}}\int_{t_j}^{t}\E[\bigl[\partial_{\gamma}^2\Phi(\widehat{Z}_k(s))-\partial_{\gamma}^2\Phi(\widehat{Z}_k(t_j))\bigr]^2]\dd s\dd t.
\end{align*}
The growth of the derivatives of $\partial_\gamma^2\Phi$ are controlled since it is assumed that $\|\Phi\|_{3,\beta}<\infty$ for some $\beta\in(0,\infty)$. As a result, applying the Cauchy--Schwarz inequality and recalling the definition~\eqref{eq:Zhat} of $\widehat{Z}_k(t)$ and of $\widehat{\Gamma}_k(t)$, for all $t\in(t_j,t_{j+1})$, one has
\begin{align*}
\E\bigl[\big|\partial_\gamma\Phi(\widehat{Z}_k(t))-\partial_\gamma\Phi(\widehat{Z}_k(t_j))\big|^2\big]&\le C\E[e^{C|f_k(X_j)|^2(t-t_j)}e^{C|\Gamma_k(t_j)|}\|\widehat{Z}_k(t)-\widehat{Z}_k(t_j)\|^2]\\
&\le C\bigl(\E[e^{2C|f_k(X_j)|^2(t-t_j)}e^{2C|\Gamma_k(t_j)|}]\bigr)^{\frac12}\bigl(\E[\|\widehat{Z}_k(t)-\widehat{Z}_k(t_j)\|^4]\bigr)^{\frac12}.
\end{align*}
If $\overline{\dt}$ is sufficiently small, Lemma~\ref{lem:auxexpo2} is applicable. Applying also the inequality~\eqref{eq:auxZhat1} from Lemma~\ref{lem:auxZhat}, one obtains the upper bounds
\begin{align}
\E[|\varepsilon_{n,k}^{(2,2,2)}|^2]&\le C\sum_{j=0}^{n-1}\int_{t_j}^{t_{j+1}}\int_{t_j}^{t}(s-t_j)\dd s\dd t\nonumber\\
&\le C\sum_{j=0}^{n-1}\int_{t_j}^{t_{j+1}}(t-t_j)^2\dd t\nonumber\\
&\le C(T)\dt^2.\label{eq:bound_varepsnk222}
\end{align}

$\bullet$ {\bf Treatment of the error term $\varepsilon_{n,k}^{(2,2,3)}$}.

Recall that $X_j\in\mathcal{D}$ almost surely, and the mappings $f_k$ and $\bigl(g\cdot \nabla\bigr)f_k$ are bounded on $\mathcal{D}$. Applying the It\^o isometry formula twice and the Cauchy--Schwarz inequality, one has
\begin{align*}
\E[|\varepsilon_{n,k}^{(2,2,3)}|^2]&=\sum_{j=0}^{n-1}\int_{t_j}^{t_{j+1}} \E[\Big|\int_{t_j}^{t}\partial_{\gamma}^2\Phi(\widehat{Z}_k(s))(g\cdot\nabla)f_k(X_j)f_k(X_j)[B(s)-B(t_j)]\dd B(s)\Big|^2]\dd t\\
&=\sum_{j=0}^{n-1}\int_{t_j}^{t_{j+1}} \int_{t_j}^{t}\E[\big|\partial_{\gamma}^2\Phi(\widehat{Z}_k(s))\big|^2 \big|(g\cdot\nabla)f_k(X_j)f_k(X_j)\big|^2|B(s)-B(t_j)|^2]\dd s\dd t\\
&\le C\sum_{j=0}^{n-1}\int_{t_j}^{t_{j+1}} \int_{t_j}^{t}\bigl(\E[\big|\partial_{\gamma}^2\Phi(\widehat{Z}_k(s))\big|^4]\bigr)^{\frac12}\bigl(\E[|B(s)-B(t_j)|^2]\bigr)^{\frac12}\dd s\dd t.
\end{align*}
Applying the inequality~\eqref{eq:lemauxexpo3} from Lemma~\ref{lem:auxexpo3} with $\rho=\partial_{\gamma}^2\Phi$, one obtains the upper bounds
\begin{align}
\E[|\varepsilon_{n,k}^{(2,2,3)}|^2]&\le C\sum_{j=0}^{n-1}\int_{t_j}^{t_{j+1}} (t-t_j)\int_{t_j}^{t}\bigl(\E[\big|\partial_{\gamma}^2\Phi(\widehat{Z}_k(s))\big|^4]\bigr)^{\frac12}\dd s\dd t\nonumber\\
&\le C\sum_{j=0}^{n-1}\int_{t_j}^{t_{j+1}} (t-t_j)^2\dd t\nonumber\\
&\le C(T)\dt^2.\label{eq:bound_varepsnk223}
\end{align}

$\bullet$ {\bf Treatment of the error term $\varepsilon_{n,k}^{(2,2,4)}$}.

Recall that $X_j\in\mathcal{D}$ almost surely, and the mappings $f_k$ and $\bigl(g\cdot \nabla\bigr)f_k$ are bounded on $\mathcal{D}$. Applying the It\^o isometry formula and the Cauchy--Schwarz inequality, one has
\begin{align*}
\E[|\varepsilon_{n,k}^{(2,2,4)}|^2]
&=\sum_{j=0}^{n-1}\int_{t_j}^{t_{j+1}} \E[\Big|\int_{t_j}^{t}\partial_{\gamma}^3\Phi(\widehat{Z}_k(t))f_k(X_j)^2(g\cdot\nabla)f_k(X_j)[B(s)-B(t_j)]\dd s\Big|^2]\dd t\\
&\le \sum_{j=0}^{n-1}\int_{t_j}^{t_{j+1}}(t-t_j)\int_{t_j}^{t}\E[|\partial_{\gamma}^3\Phi(\widehat{Z}_k(t))|^2f_k(X_j)^4[(g\cdot\nabla)f_k(X_j)]^2|B(s)-B(t_j))|^2]\dd s\dd t\\
&\le \sum_{j=0}^{n-1}\int_{t_j}^{t_{j+1}}(t-t_j)\int_{t_j}^{t}\bigl(\E[|\partial_{\gamma}^3\Phi(\widehat{Z}_k(t))|^4]\bigr))^{\frac12}\bigl(\E[|B(s)-B(t_j))|^4]\bigr)^{\frac12}\dd s\dd t.
\end{align*}
Applying the inequality~\eqref{eq:lemauxexpo3} from Lemma~\ref{lem:auxexpo3} with $\rho=\partial_{\gamma}^3\Phi$, one obtains the upper bounds
\begin{align}
\E[|\varepsilon_{n,k}^{(2,2,4)}|^2]&\le C\sum_{j=0}^{n-1}\int_{t_j}^{t_{j+1}} (t-t_j)^2\int_{t_j}^{t}\bigl(\E[\big|\partial_{\gamma}^3\Phi(\widehat{Z}_k(s))\big|^4]\bigr)^{\frac12}\dd s\dd t\nonumber\\
&\le C\sum_{j=0}^{n-1}\int_{t_j}^{t_{j+1}} (t-t_j)^3\dd t\nonumber\\
&\le C(T)\dt^3.\label{eq:bound_varepsnk224}
\end{align}

$\bullet$ {\bf Treatment of the error term $\varepsilon_{n,k}^{(2,2,5)}$}.

Recall that $X_j\in\mathcal{D}$ almost surely, and the mappings $f_k$ and $\bigl(g\cdot \nabla\bigr)f_k$ are bounded on $\mathcal{D}$. Applying the It\^o isometry formula and the Cauchy--Schwarz inequality, one has
\begin{align*}
\E[|\varepsilon_{n,k}^{(2,2,5)}|^2]
&=\sum_{j=0}^{n-1}\int_{t_j}^{t_{j+1}}\E[\Big|\int_{t_j}^{t} \partial_{\gamma}^3\Phi(\widehat{Z}_k(s))f_k(X_j)[(g\cdot\nabla)f_k(X_j)]^2[B(s)-B(t_j)]^2\dd s\Big|^2]\dd t\\
&\le \sum_{j=0}^{n-1}\int_{t_j}^{t_{j+1}}(t-t_j)\int_{t_j}^{t}\E[|\partial_{\gamma}^3\Phi(\widehat{Z}_k(s))|^2f_k(X_j)^2[(g\cdot\nabla)f_k(X_j)]^4|B(s)-B(t_j)|^4]\dd s\dd t\\
&\le C\sum_{j=0}^{n-1}\int_{t_j}^{t_{j+1}}(t-t_j)\int_{t_j}^{t}\bigl(\E[|\partial_{\gamma}^3\Phi(\widehat{Z}_k(s))|^4]\bigr)^{\frac12}\bigl(\E[|B(s)-B(t_j)|^8]\bigr)^{\frac12}\dd s\dd t.
\end{align*}
Applying the inequality~\eqref{eq:lemauxexpo3} from Lemma~\ref{lem:auxexpo3} with $\rho=\partial_{\gamma}^3\Phi$, one obtains the upper bounds
\begin{align}
\E[|\varepsilon_{n,k}^{(2,2,5)}|^2]&\le C\sum_{j=0}^{n-1}\int_{t_j}^{t_{j+1}} (t-t_j)^3\int_{t_j}^{t}\bigl(\E[\big|\partial_{\gamma}^3\Phi(\widehat{Z}_k(s))\big|^4]\bigr)^{\frac12}\dd s\dd t\nonumber\\
&\le C\sum_{j=0}^{n-1}\int_{t_j}^{t_{j+1}} (t-t_j)^4\dd t\nonumber\\
&\le C(T)\dt^4.\label{eq:bound_varepsnk225}
\end{align}

$\bullet$ {\bf Treatment of the error term $\varepsilon_{n,k}^{(2,2,6)}$}.

Recall that $X_j\in\mathcal{D}$ almost surely, and that the mapping $f_k$ is bounded on $\mathcal{D}$. Applying the It\^o isometry formula and the Cauchy--Schwarz inequality, one has
\begin{align*}
\E[|\varepsilon_{n,k}^{(2,2,6)}|^2]
&=\sum_{j=0}^{n-1}\int_{t_j}^{t_{j+1}} \E[\Big|\int_{t_j}^{t}\partial_\gamma\psi(\widehat{Z}_k(s))f_k(X_j)^3\dd s\Big|^2] \dd t\\
&\le C\sum_{j=0}^{n-1}\int_{t_j}^{t_{j+1}}(t-t_j)\int_{t_j}^{t}\E[|\partial_\gamma\psi(\widehat{Z}_k(s))|^2f_k(X_j)^6]\dd s\dd t\\
&\le C\sum_{j=0}^{n-1}\int_{t_j}^{t_{j+1}}(t-t_j)\int_{t_j}^{t}\E[|\partial_\gamma\psi(\widehat{Z}_k(s))|^2]\dd s\dd t.
\end{align*}
Applying the inequality~\eqref{eq:lemauxexpo3} from Lemma~\ref{lem:auxexpo3} with $\rho=\partial_{\gamma}\psi$, one obtains the upper bounds
\begin{align}
\E[|\varepsilon_{n,k}^{(2,2,6)}|^2]&\le C\sum_{j=0}^{n-1}\int_{t_j}^{t_{j+1}} (t-t_j)^2\dd t\nonumber\\
&\le C(T)\dt^2.\label{eq:bound_varepsnk226}
\end{align}

$\bullet$ {\bf Conclusion.}

Recalling the decomposition~\eqref{eq:decomp_varepsnk22} of the error term $\varepsilon_{n,k}^{(2,2)}$ and combining the upper bounds~\eqref{eq:bound_varepsnk221},~\eqref{eq:bound_varepsnk222},~\eqref{eq:bound_varepsnk223},~\eqref{eq:bound_varepsnk224},~\eqref{eq:bound_varepsnk225} and~\eqref{eq:bound_varepsnk226} for the error terms defined by~\eqref{eq:def_varepsnk221},~\eqref{eq:def_varepsnk222},~\eqref{eq:def_varepsnk223},~\eqref{eq:def_varepsnk224},~\eqref{eq:def_varepsnk225} and~\eqref{eq:def_varepsnk226}, one finally obtains the following result: there exists $C(T)\in(0,\infty)$ such that
\begin{equation}\label{eq:bound_varepsnk22}
\sum_{k=1}^{d}\underset{1\le n\le N}\sup~\E[|\varepsilon_{n,k}^{(2,2)}|^2]\le C(T)\dt^2.
\end{equation}
This concludes the proof of the upper bound~\eqref{eq:lem_varepsilon2} for the error term $\varepsilon_{n,k}^{(2,2)}$.
\end{proof}

\begin{proof}[Proof of the upper bound~\eqref{eq:lem_varepsilon2} for the error term $\varepsilon_{n,k}^{(2,3)}$]

The error term $\varepsilon_{n,k}^{(2,3)}$ is defined by~\eqref{eq:def_varepsnk23}. Recall that $X_j\in\mathcal{D}$ almost surely, and that the mapping $\bigl(g\cdot\nabla\bigr)f_k$ is bounded on $\mathcal{D}$.
Applying the It\^o isometry formula and the Cauchy--Schwarz inequality, one has
\begin{align*}
\E[|\varepsilon_{n,k}^{(2,3)}|^2]&=\sum_{j=0}^{n-1}\int_{t_j}^{t_{j+1}}\E\left[\big|\partial_\gamma\Phi(\widehat{Z}_k(t))-\partial_\gamma\Phi(\widehat{Z}_k(t_j))\big|^2\big|\bigl(g\cdot \nabla\bigr) f_k(X_j)\big|^{2}\big|B(t)-B(t_j)\big|^2 \right]\dd t\\
&\le \sum_{j=0}^{n-1}\int_{t_j}^{t_{j+1}}\bigl(\E\bigl[\big|\partial_\gamma\Phi(\widehat{Z}_k(t))-\partial_\gamma\Phi(\widehat{Z}_k(t_j))\big|^4\bigr]\bigr)^{\frac12}\bigl(\E\bigl[\big|B(t)-B(t_j)|^4\bigr]\bigr)^{\frac12} \dd t\\
&\le \sum_{j=0}^{n-1}\int_{t_j}^{t_{j+1}}(t-t_j)\bigl(\E\bigl[\big|\partial_\gamma\Phi(\widehat{Z}_k(t))-\partial_\gamma\Phi(\widehat{Z}_k(t_j))\big|^4\bigr]\bigr)^{\frac12} \dd t.
\end{align*}
The growth of the derivatives of $\partial_\gamma\Phi$ are controlled since it is assumed that $\|\Phi\|_{2,\beta}<\infty$ for some $\beta\in(0,\infty)$. As a result, applying the Cauchy--Schwarz inequality and recalling the definition~\eqref{eq:Zhat} of $\widehat{Z}_k(t)$ and of $\widehat{\Gamma}_k(t)$, for all $t\in(t_j,t_{j+1})$, one has
\begin{align*}
\E\bigl[\big|\partial_\gamma\Phi(\widehat{Z}_k(t))-\partial_\gamma\Phi(\widehat{Z}_k(t_j))\big|^4\big]&\le C\E[e^{C|f_k(X_j)|^2(t-t_j)}e^{C|\Gamma_k(t_j)|}\|\widehat{Z}_k(t)-\widehat{Z}_k(t_j)\|^4]\\
&\le C\bigl(\E[e^{2C|f_k(X_j)|^2(t-t_j)}e^{2C|\Gamma_k(t_j)|}]\bigr)^{\frac12}\bigl(\E[\|\widehat{Z}_k(t)-\widehat{Z}_k(t_j)\|^8]\bigr)^{\frac12}.
\end{align*}
If $\overline{\dt}$ is sufficiently small, Lemma~\ref{lem:auxexpo2} is applicable. Applying also the inequality~\eqref{eq:auxZhat1} from Lemma~\ref{lem:auxZhat}, one obtains the upper bound
\[
\E[|\varepsilon_{n,k}^{(2,3)}|^2]\le C(T)\sum_{j=0}^{n-1}\int_{t_j}^{t_{j+1}}(t-t_j)^2 \dd t.
\]
Therefore one obtains the following result: there exists $C(T)\in(0,\infty)$ such that
\begin{equation}\label{eq:bound_varepsnk23}
\sum_{k=1}^{d}\underset{1\le n\le N}\sup~\E[|\varepsilon_{n,k}^{(2,3)}|^2]\le C(T)\dt^2.
\end{equation}
This concludes the proof of the upper bound~\eqref{eq:lem_varepsilon2} for the error term $\varepsilon_{n,k}^{(2,3)}$.
\end{proof}

\begin{proof}[Proof of the upper bound~\eqref{eq:lem_varepsilon2} for the error term $\varepsilon_{n,k}^{(2,4)}$]

The error term $\varepsilon_{n,k}^{(2,4)}$ is defined by~\eqref{eq:def_varepsnk24} and is decomposed as
\begin{equation}\label{eq:decomp_varepsnk24}
\varepsilon_{n,k}^{(2,4)}=\varepsilon_{n,k}^{(2,4,1)}+\varepsilon_{n,k}^{(2,4,2)},
\end{equation}
where the error terms $\varepsilon_{n,k}^{(2,4,1)}$ and $\varepsilon_{n,k}^{(2,4,2)}$ are defined as
\begin{align}
\varepsilon_{n,k}^{(2,4,1)}
&=\sum_{j=0}^{n-1}\int_{t_j}^{t_{j+1}}\bigl[\partial_{\gamma}^2\Phi(\widehat{Z}_k(t))-\partial_{\gamma}^2\Phi(\widehat{Z}_k(t_j))\bigr]\bigl(g\cdot\nabla\bigr)f_k(X_j)f_k(X_j)\bigl[B(t)-B(t_j)\bigr]\dd t,\label{eq:def_varepsnk241}
\\
\varepsilon_{n,k}^{(2,4,2)}
&=\sum_{j=0}^{n-1}\int_{t_j}^{t_{j+1}}\partial_{\gamma}^2\Phi(\widehat{Z}_k(t_j))\bigl(g\cdot\nabla\bigr)f_k(X_j)f_k(X_j)\bigl[B(t)-B(t_j)\bigr]\dd t.\label{eq:def_varepsnk242}
\end{align}

$\bullet$ {\bf Treatment of the error term $\varepsilon_{n,k}^{(2,4,1)}$}.

Recall that $X_j\in\mathcal{D}$ almost surely, and the mappings $f_k$ and $\bigl(g\cdot \nabla\bigr)f_k$ are bounded on $\mathcal{D}$. Applying the Cauchy--Schwarz inequality, one has
\begin{align*}
\E[|\varepsilon_{n,k}^{(2,4,1)}|^2]
&\le C(T)\sum_{j=0}^{n-1}\int_{t_j}^{t_{j+1}}\E[\big|\partial_{\gamma}^2\Phi(\widehat{Z}_k(t))-\partial_{\gamma}^2\Phi(\widehat{Z}_k(t_j))\big|^2[\bigl(g\cdot\nabla\bigr)f_k(X_j)]^2f_k(X_j)^2|B(t)-B(t_j)|^2]\dd t\\
&\le C(T)\sum_{j=0}^{n-1}\int_{t_j}^{t_{j+1}}\bigl(\E[\big|\partial_{\gamma}^2\Phi(\widehat{Z}_k(t))-\partial_{\gamma}^2\Phi(\widehat{Z}_k(t_j))\big|^4]\bigr)^{\frac12}\bigl(\E[|B(t)-B(t_j)|^4]\bigr)^{\frac12}\dd t.
\end{align*}
The growth of the derivatives of $\partial_\gamma^2\Phi$ are controlled since it is assumed that $\|\Phi\|_{3,\beta}<\infty$ for some $\beta\in(0,\infty)$. As a result, applying the Cauchy--Schwarz inequality and recalling the definition~\eqref{eq:Zhat} of $\widehat{Z}_k(t)$ and of $\widehat{\Gamma}_k(t)$, for all $t\in(t_j,t_{j+1})$, one has
\begin{align*}
\E\bigl[\big|\partial_\gamma\Phi(\widehat{Z}_k(t))-\partial_\gamma\Phi(\widehat{Z}_k(t_j))\big|^4\big]&\le C\E[e^{C|f_k(X_j)|^2(t-t_j)}e^{C|\Gamma_k(t_j)|}\|\widehat{Z}_k(t)-\widehat{Z}_k(t_j)\|^4]\\
&\le C\bigl(\E[e^{2C|f_k(X_j)|^2(t-t_j)}e^{2C|\Gamma_k(t_j)|}]\bigr)^{\frac12}\bigl(\E[\|\widehat{Z}_k(t)-\widehat{Z}_k(t_j)\|^8]\bigr)^{\frac12}.
\end{align*}
If $\overline{\dt}$ is sufficiently small, Lemma~\ref{lem:auxexpo2} is applicable. Applying also the inequality~\eqref{eq:auxZhat1} from Lemma~\ref{lem:auxZhat}, one obtains the upper bounds
\begin{align}
\E[|\varepsilon_{n,k}^{(2,4,1)}|^2]
&\le C(T)\sum_{j=0}^{n-1}\int_{t_j}^{t_{j+1}}\bigl(\E[\big|\partial_{\gamma}^2\Phi(\widehat{Z}_k(t))-\partial_{\gamma}^2\Phi(\widehat{Z}_k(t_j))\big|^4]\bigr)^{\frac12}(t-t_j)\dd t\nonumber\\
&\le C(T)\sum_{j=0}^{n-1}\int_{t_j}^{t_{j+1}}(t-t_j)^2\dd t\nonumber\\
&\le C(T)\dt^2.\label{eq:bound_varepsnk241}
\end{align}

$\bullet$ {\bf Treatment of the error term $\varepsilon_{n,k}^{(2,4,2)}$}.

Applying the stochastic Fubini theorem, one has
\begin{align*}
\varepsilon_{n,k}^{(2,4,2)}
&=\sum_{j=0}^{n-1}\int_{t_j}^{t_{j+1}}\int_{t_j}^{t}\partial_{\gamma}^2\Phi(\widehat{Z}_k(t_j))\bigl(g\cdot\nabla\bigr)f_k(X_j)f_k(X_j)\dd B(s)\dd t\\
&=\sum_{j=0}^{n-1}\int_{t_j}^{t_{j+1}}(t_{j+1}-s)\partial_{\gamma}^2\Phi(\widehat{Z}_k(t_j))\bigl(g\cdot\nabla\bigr)f_k(X_j)f_k(X_j)\dd B(s).
\end{align*}
Recall that $X_j\in\mathcal{D}$ almost surely, and the mappings $f_k$ and $\bigl(g\cdot \nabla\bigr)f_k$ are bounded on $\mathcal{D}$. Applying the It\^o isometry formula, one has
\begin{align*}
\E[|\varepsilon_{n,k}^{(2,4,2)}|^2]
&=\sum_{j=0}^{n-1}\int_{t_j}^{t_{j+1}}(t_{j+1}-s)^2\E[\big|\partial_{\gamma}^2\Phi(\widehat{Z}_k(t_j))\bigl(g\cdot\nabla\bigr)f_k(X_j)f_k(X_j)\big|^2]\dd s\\
&\le C\dt^2\sum_{j=0}^{n-1}\int_{t_j}^{t_{j+1}}\E[|\partial_{\gamma}^2\Phi(\widehat{Z}_k(t_j))|^2]\dd s.
\end{align*}
Applying the inequality~\eqref{eq:lemauxexpo3} from Lemma~\ref{lem:auxexpo3} with $\rho=\partial_{\gamma}^2\Phi$, one obtains the upper bound
\begin{equation}\label{eq:bound_varepsnk242}
\E[|\varepsilon_{n,k}^{(2,4,2)}|^2]\le C(T)\dt^2.
\end{equation}

Recalling the decomposition~\eqref{eq:decomp_varepsnk24} of the error term $\varepsilon_{n,k}^{(2,4)}$ and combining the upper bounds~\eqref{eq:bound_varepsnk241} and~\eqref{eq:bound_varepsnk242} for the error terms defined by~\eqref{eq:def_varepsnk241} and~\eqref{eq:def_varepsnk242}, one obtains the following result: there exists $C(T)\in(0,\infty)$ such that
\begin{equation}\label{eq:bound_varepsnk24}
\sum_{k=1}^{d}\underset{1\le n\le N}\sup~\E[|\varepsilon_{n,k}^{(2,4)}|^2]\le C(T)\dt^2.
\end{equation}
This concludes the proof of the upper bound~\eqref{eq:lem_varepsilon2} for the error term $\varepsilon_{n,k}^{(2,4)}$.
\end{proof}

\begin{proof}[Proof of the upper bound~\eqref{eq:lem_varepsilon2} for the error term $\varepsilon_{n,k}^{(2,5)}$]

The error term is defined by $\varepsilon_{n,k}^{(2,5)}$. Recall that $X_j\in\mathcal{D}$ almost surely, and that the mapping $\bigl(g\cdot\nabla\bigr)f_k$ is bounded on $\mathcal{D}$. Applying the Cauchy--Schwarz inequality, one has
\begin{align*}
\E[|\varepsilon_{n,k}^{(2,5)}|^2]
&\le C(T)\sum_{j=0}^{n-1}\int_{t_j}^{t_{j+1}}\E[|\partial_{\gamma}^2\Phi(\widehat{Z}_k(t))|^2|B(t)-B(t_j)|^4]\dd t\\
&\le C(T)\sum_{j=0}^{n-1}\int_{t_j}^{t_{j+1}}\bigl(\E[|\partial_{\gamma}^2\Phi(\widehat{Z}_k(t))|^4\bigr)^{\frac12}\bigl(\E[|B(t)-B(t_j)|^8]\bigr)^{\frac12}\dd t\\
&\le C(T)\dt^2\sum_{j=0}^{n-1}\int_{t_j}^{t_{j+1}}\bigl(\E[|\partial_{\gamma}^2\Phi(\widehat{Z}_k(t))|^4\bigr)^{\frac12}\dd t.
\end{align*}
Applying the inequality~\eqref{eq:lemauxexpo3} from Lemma~\ref{lem:auxexpo3} with $\rho=\partial_\gamma^2\Phi$, one obtains the following result: there exists $C(T)\in(0,\infty)$ such that
\begin{equation}\label{eq:bound_varepsnk25}
\sum_{k=1}^{d}\underset{1\le n\le N}\sup~\E[|\varepsilon_{n,k}^{(2,5)}|^2]\le C(T)\dt^2.
\end{equation}
This concludes the proof of the upper bound~\eqref{eq:lem_varepsilon2} for the error term $\varepsilon_{n,k}^{(2,5)}$.
\end{proof}

\begin{proof}[Proof of the upper bound~\eqref{eq:lem_varepsilon2} for the error term $\varepsilon_{n,k}^{(2,6)}$]

As a preliminary, a decomposition of the error term $\varepsilon_{n,k}^{(2,6)}$ defined by~\eqref{eq:def_varepsnk26} is provided.

Let $k\in\{1,\ldots,d\}$. The mapping $\psi=\partial_s\Phi+\frac12\partial_\gamma^2\Phi$ is of class $\mathcal{C}^2$ and the growth of its first and second order derivatives is controlled since one has assumed that $\|\Phi\|_{4,\beta}<\infty$ for some $\beta\in(0,\infty)$, see~\eqref{eq:regPhi}.

Applying the It\^o formula for $t\mapsto \rho(\widehat{Z}_k(t))$ with $\rho=\psi$, for all $j\in\{0,\ldots,N-1\}$ and all $t\in(t_j,t_{j+1})$, the error term $\varepsilon_{n,k}^{(2,6)}$ is decomposed as
\begin{equation}\label{eq:decomp_varepsnk26}
\varepsilon_{n,k}^{(2,6)}=\varepsilon_{n,k}^{(2,6,1)}+\varepsilon_{n,k}^{(2,6,2)}+\varepsilon_{n,k}^{(2,6,3)}+\varepsilon_{n,k}^{(2,6,4)}+\varepsilon_{n,k}^{(2,6,5)}+\varepsilon_{n,k}^{(2,6,6)},
\end{equation}
where the error terms $\varepsilon_{n,k}^{(2,6,\iota)}$ for $\iota=1,\ldots,6$ are defined as
\begin{align}
\varepsilon_{n,k}^{(2,6,1)}
&=\sum_{j=0}^{n-1}\int_{t_j}^{t_{j+1}} \int_{t_j}^{t}\partial_\gamma \psi (\widehat{Z}_k(s)) \dd s f_k(X_j)^2\dd t,\label{eq:def_varepsnk261}
\\
\varepsilon_{n,k}^{(2,6,2)}
&=\sum_{j=0}^{n-1}\int_{t_j}^{t_{j+1}} \int_{t_j}^{t}\partial_\gamma \psi(\widehat{Z}_k(s))f_k(X_j)\dd B(s) f_k(X_j)^2\dd t,\label{eq:def_varepsnk262}
\\
\varepsilon_{n,k}^{(2,6,3)}
&=\sum_{j=0}^{n-1}\int_{t_j}^{t_{j+1}} \int_{t_j}^{t}\partial_\gamma\psi(\widehat{Z}_k(s))(g\cdot\nabla)f_k(X_j)[B(s)-B(t_j)]\dd B(s) f_k(X_j)^2\dd t,\label{eq:def_varepsnk263}
\\
\varepsilon_{n,k}^{(2,6,4)}
&=\sum_{j=0}^{n-1}\int_{t_j}^{t_{j+1}} \int_{t_j}^{t}\partial_{\gamma}^2\psi(\widehat{Z}_k(t))f_k(X_j)(g\cdot\nabla)f_k(X_j)(B(s)-B(t_j))\dd s f_k(X_j)^2\dd t,\label{eq:def_varepsnk264}
\\
\varepsilon_{n,k}^{(2,6,5)}
&=\sum_{j=0}^{n-1}\int_{t_j}^{t_{j+1}} \int_{t_j}^{t} \partial_{\gamma}^2\psi(\widehat{Z}_k(s))[(g\cdot\nabla)f_k(X_j)]^2[B(s)-B(t_j)]^2\dd s f_k(X_j)^2\dd t,\label{eq:def_varepsnk265}
\\
\varepsilon_{n,k}^{(2,6,6)}
&=\sum_{j=0}^{n-1}\int_{t_j}^{t_{j+1}} \int_{t_j}^{t}\bigl(\partial_s\psi+\frac12\partial_{\gamma}^2\psi\bigr)(\widehat{Z}_k(s))f_k(X_j)^2\dd s f_k(X_j)^2\dd t.\label{eq:def_varepsnk266}
\end{align}

In the sequel, the error terms $\varepsilon_{n,k}^{(2,6,\iota)}$ for $\iota=1,\ldots,6$ are treated separately. The arguments are similar to those employed to deal with the error terms $\varepsilon_{n,k}^{(2,1,\iota)}$ for $\iota=1,\ldots,6$, the details are provided for completeness.

$\bullet$ {\bf Treatment of the error term $\varepsilon_{n,k}^{(2,6,1)}$}.

Recall that $X_j\in\mathcal{D}$ almost surely, and that the mapping $f_k$ is bounded on $\mathcal{D}$. Therefore applying the Cauchy--Schwarz inequality, one has
\begin{align*}
\E[|\varepsilon_{n,k}^{(2,6,1)}|^2]&\le C(T)\sum_{j=0}^{n-1}\int_{t_j}^{t_{j+1}}\E[\bigl(\int_{t_j}^{t}|\partial_{\gamma}\psi (\widehat{Z}_k(s))|\dd s\bigr)^2]\dd t\\
&\le C(T)\sum_{j=0}^{n-1}\int_{t_j}^{t_{j+1}}(t-t_j)\int_{t_j}^{t}\E[|\partial_{\gamma}\psi (\widehat{Z}_k(s))|^2]\dd s\dd t.
\end{align*}
Applying the inequality~\eqref{eq:lemauxexpo3} from Lemma~\ref{lem:auxexpo3} with $\rho=\partial_\gamma\psi$, one obtains the upper bounds
\begin{align}
\E[|\varepsilon_{n,k}^{(2,6,1)}|^2]&\le C(T)\sum_{j=0}^{n-1}\int_{t_j}^{t_{j+1}}(t-t_j)^2\dd t\nonumber\\
&\le C(T)\dt^2.\label{eq:bound_varepsnk261}
\end{align}

$\bullet$ {\bf Treatment of the error term $\varepsilon_{n,k}^{(2,6,2)}$}.

Applying the stochastic Fubini theorem, one has
\begin{align*}
\varepsilon_{n,k}^{(2,6,2)}&=\sum_{j=0}^{n-1}\int_{t_j}^{t_{j+1}} \int_{t_j}^{t}\partial_{\gamma} \psi(\widehat{Z}_k(s))f_k(X_j)^3\dd t\dd B(s)\\
&=\sum_{j=0}^{n-1}\int_{t_j}^{t_{j+1}} (t_{j+1}-s)\partial_{\gamma} \psi(\widehat{Z}_k(s))f_k(X_j)^3\dd B(s).
\end{align*}
Recall that $X_j\in\mathcal{D}$ almost surely, and that the mapping $f_k$ is bounded on $\mathcal{D}$.
Applying the It\^o isometry formula, one has 
\begin{align*}
\E[|\varepsilon_{n,k}^{(2,6,2)}|^2]&=\sum_{j=0}^{n-1}\int_{t_j}^{t_{j+1}}(t_{j+1}-s)^2\E[|\partial_{\gamma}^2 \psi(\widehat{Z}_k(s))|^2|f_k(X_j)|^6]\dd s\\
&\le C\dt^2\sum_{j=0}^{n-1}\int_{t_j}^{t_{j+1}}\E[|\partial_{\gamma}^2\psi(\widehat{Z}_k(s))|^2]\dd s.
\end{align*}
Applying the inequality~\eqref{eq:lemauxexpo3} from Lemma~\ref{lem:auxexpo3} from Lemma~\ref{lem:auxexpo3} with $\rho=\partial_\gamma^2\psi$, for $\dt\in(0,\overline{\dt})$ one obtains the upper bound
\begin{equation}\label{eq:bound_varepsnk262}
\E[|\varepsilon_{n,k}^{(2,6,2)}|^2]\le C(T)\dt^2.
\end{equation}

$\bullet$ {\bf Treatment of the error term $\varepsilon_{n,k}^{(2,6,3)}$}.

Applying the stochastic Fubini theorem, one has
\begin{align*}
\varepsilon_{n,k}^{(2,6,3)}&=\sum_{j=0}^{n-1}\int_{t_j}^{t_{j+1}} \int_{t_j}^{t}\partial_{\gamma}\psi(\widehat{Z}_k(s))(g\cdot\nabla)f_k(X_j)[B(s)-B(t_j)]f_k(X_j)^2\dd t\dd B(s)\\
&=\sum_{j=0}^{n-1}\int_{t_j}^{t_{j+1}} (t_{j+1}-s)\partial_{\gamma}\psi(\widehat{Z}_k(s))(g\cdot\nabla)f_k(X_j)[B(s)-B(t_j)]f_k(X_j)^2\dd B(s).
\end{align*}
Recall that $X_j\in\mathcal{D}$ almost surely, and that the mappings $\bigl(g\cdot\nabla\bigr)f_k$ and $f_k$ are bounded on $\mathcal{D}$.
Applying the It\^o isometry formula and the Cauchy--Schwarz inequality, one obtains
\begin{align*}
\E[|\varepsilon_{n,k}^{(2,6,3)}|^2]&=\sum_{j=0}^{n-1}\int_{t_j}^{t_{j+1}}(t_{j+1}-s)^2\E[|\partial_{\gamma}\psi(\widehat{Z}_k(s))(g\cdot\nabla)f_k(X_j)[B(s)-B(t_j)]f_k(X_j)^2|^2]\dd s\\
&\le C\dt^2\sum_{j=0}^{n-1}\int_{t_j}^{t_{j+1}}\E[|\partial_{\gamma}\psi(\widehat{Z}_k(s))|^2|B(s)-B(t_j)|^2]\dd s\\
&\le C\dt^2\sum_{j=0}^{n-1}\int_{t_j}^{t_{j+1}}\bigl(\E[|\partial_{\gamma}\psi(\widehat{Z}_k(s))|^4]\bigr)^{\frac12}\bigl(\E[|B(s)-B(t_j)|^4]\bigr)^{\frac12}\dd s.
\end{align*}
Applying the inequality~\eqref{eq:lemauxexpo3} from Lemma~\ref{lem:auxexpo3} with $\rho=\partial_\gamma\psi$, for $\dt\in(0,\overline{\dt})$ one obtains the upper bounds
\begin{align}
\E[|\varepsilon_{n,k}^{(2,6,3)}|^2]&\le C\dt^2\sum_{j=0}^{n-1}\int_{t_j}^{t_{j+1}}\bigl(\E[|\partial_{\gamma}\psi(\widehat{Z}_k(s))|^4]\bigr)^{\frac12}(s-t_j)\dd s\nonumber\\
&\le C\dt^3\sum_{j=0}^{n-1}\int_{t_j}^{t_{j+1}}\bigl(\E[|\partial_{\gamma}\psi(\widehat{Z}_k(s))|^4]\bigr)^{\frac12}\dd s\nonumber\\
&\le C(T)\dt^3.\label{eq:bound_varepsnk263}
\end{align}

$\bullet$ {\bf Treatment of the error term $\varepsilon_{n,k}^{(2,6,4)}$}.

Recall that $X_j\in\mathcal{D}$ almost surely, and that the mappings $f_k$ and $\bigl(g\cdot\nabla\bigr)f_k$ are bounded on $\mathcal{D}$. Applying the Cauchy--Schwarz inequality, one has
\begin{align*}
\E[|\varepsilon_{n,k}^{(2,6,4)}|^2]&\le C(T)\sum_{j=0}^{n-1}\int_{t_j}^{t_{j+1}}\E[\Big|\int_{t_j}^{t}\partial_{\gamma}^2\psi(\widehat{Z}_k(t))f_k(X_j)^3(g\cdot\nabla)f_k(X_j)(B(s)-B(t_j))\dd s\Big|^2]\dd t\\
&\le C(T)\sum_{j=0}^{n-1}\int_{t_j}^{t_{j+1}}(t-t_j)\int_{t_j}^{t}\E[|\partial_{\gamma}^2\psi(\widehat{Z}_k(t))|^2|B(s)-B(t_j)|^2]\dd s\dd t\\
&\le C(T)\sum_{j=0}^{n-1}\int_{t_j}^{t_{j+1}}(t-t_j)\int_{t_j}^{t}\bigl(\E[|\partial_{\gamma}^2\psi(\widehat{Z}_k(t))|^4\bigr)^{\frac12}\bigl(\E[|B(s)-B(t_j)|^4]\bigr)^{\frac12}\dd s\dd t.
\end{align*}
Applying the inequality~\eqref{eq:lemauxexpo3} from Lemma~\ref{eq:lemauxexpo3} with $\rho=\partial_\gamma^2\psi$, one obtains the upper bounds
\begin{align}
\E[|\varepsilon_{n,k}^{(2,6,4)}|^2]&\le C(T)\sum_{j=0}^{n-1}\int_{t_j}^{t_{j+1}}(t-t_j)^2\int_{t_j}^{t}\bigl(\E[|\partial_{\gamma}^2\psi(\widehat{Z}_k(t))|^4\bigr)^{\frac12}\dd s\dd t\nonumber\\
&\le C(T)\sum_{j=0}^{n-1}\int_{t_j}^{t_{j+1}}(t-t_j)^3\dd t\nonumber\\
&\le C(T)\dt^3.\label{eq:bound_varepsnk264}
\end{align}

$\bullet$ {\bf Treatment of the error term $\varepsilon_{n,k}^{(2,6,5)}$}.

Recall that $X_j\in\mathcal{D}$ almost surely, and that the mappings $\bigl(g\cdot\nabla\bigr)f_k$ and $f_k^0$ are bounded on $\mathcal{D}$. Applying the Cauchy--Schwarz inequality, one has
\begin{align*}
\E[|\varepsilon_{n,k}^{(2,6,5)}|^2]&\le C(T)\sum_{j=0}^{n-1}\int_{t_j}^{t_{j+1}}\E[\Big|\int_{t_j}^{t}\partial_{\gamma}^2\psi(\widehat{Z}_k(s))[(g\cdot\nabla)f_k(X_j)]^2[B(s)-B(t_j)]^2f_k(X_j)^2\dd s\Big|^2]\dd t\\
&\le C(T)\sum_{j=0}^{n-1}\int_{t_j}^{t_{j+1}}(t-t_j)\int_{t_j}^{t}\E[|\partial_{\gamma}^2\psi(\widehat{Z}_k(s))|^2|B(s)-B(t_j)|^4]\dd s\dd t\\
&\le C(T)\sum_{j=0}^{n-1}\int_{t_j}^{t_{j+1}}(t-t_j)\int_{t_j}^{t}\bigl(\E[|\partial_{\gamma}^2\psi(\widehat{Z}_k(t))|^4\bigr)^{\frac12}\bigl(\E[|B(s)-B(t_j)|^8]\bigr)^{\frac12}\dd s\dd t.
\end{align*}
Applying the inequality~\eqref{eq:lemauxexpo3} from Lemma~\ref{lem:auxexpo3} with $\rho=\partial_\gamma^2\psi$, for $\dt\in(0,\overline{\dt})$, one obtains the upper bounds
\begin{align}
\E[|\varepsilon_{n,k}^{(2,6,5)}|^2]&\le C(T)\sum_{j=0}^{n-1}\int_{t_j}^{t_{j+1}}(t-t_j)^3\int_{t_j}^{t}\bigl(\E[|\partial_{\gamma}^2\psi(\widehat{Z}_k(t))|^4\bigr)^{\frac12}\dd s\dd t\nonumber\\
&\le C(T)\sum_{j=0}^{n-1}\int_{t_j}^{t_{j+1}}(t-t_j)^4\dd t\nonumber\\
&\le C(T)\dt^4.\label{eq:bound_varepsnk265}
\end{align}

$\bullet$ {\bf Treatment of the error term $\varepsilon_{n,k}^{(2,6,6)}$}.

Recall that $X_j\in\mathcal{D}$ almost surely, and that the mappings $f_k$ is bounded on $\mathcal{D}$. Applying the Cauchy--Schwarz inequality, one has
\begin{align*}
\E[|\varepsilon_{n,k}^{(2,1,6)}|^2]&\le C(T)\sum_{j=0}^{n-1}\int_{t_j}^{t_{j+1}}\E[\Big|\int_{t_j}^{t}\bigl(\partial_s\psi+\frac12\partial_\gamma^2\psi\bigr)(\widehat{Z}_k(s))f_k(X_j)^4\dd s\Big|^2]\dd t\\
&\le C(T)\sum_{j=0}^{n-1}\int_{t_j}^{t_{j+1}}(t-t_j)\int_{t_j}^{t}\E[|\bigl(\partial_s\psi+\frac12\partial_\gamma^2\psi\bigr)(\widehat{Z}_k(s))|^2]\dd s\dd t.
\end{align*}
Applying the inequality~\eqref{eq:lemauxexpo3} from Lemma~\ref{lem:auxexpo3} with $\rho=\partial_s\psi+\frac12\partial_\gamma^2\psi$, one obtains the upper bounds
\begin{align}
\E[|\varepsilon_{n,k}^{(2,6,6)}|^2]&\le C(T)\sum_{j=0}^{n-1}\int_{t_j}^{t_{j+1}}(t-t_j)^2\dd s\dd t\nonumber\\
&\le C(T)\dt^2.\label{eq:bound_varepsnk266}
\end{align}

$\bullet$ {\bf Conclusion.}

Recalling the decomposition~\eqref{eq:decomp_varepsnk26} of the error term $\varepsilon_{n,k}^{(2,6)}$ and combining the upper bounds~\eqref{eq:bound_varepsnk261},~\eqref{eq:bound_varepsnk262},~\eqref{eq:bound_varepsnk263},~\eqref{eq:bound_varepsnk264},~\eqref{eq:bound_varepsnk265} and~\eqref{eq:bound_varepsnk266} for the error terms defined by~\eqref{eq:def_varepsnk261},~\eqref{eq:def_varepsnk262},~\eqref{eq:def_varepsnk263},~\eqref{eq:def_varepsnk264},~\eqref{eq:def_varepsnk265} and~\eqref{eq:def_varepsnk266}, one finally obtains the following result: there exists $C(T)\in(0,\infty)$ such that
\begin{equation}\label{eq:bound_varepsnk26}
\sum_{k=1}^{d}\underset{1\le n\le N}\sup~\E[|\varepsilon_{n,k}^{(2,6)}|^2]\le C(T)\dt^2.
\end{equation}
This concludes the proof of the upper bound~\eqref{eq:lem_varepsilon2} for the error term $\varepsilon_{n,k}^{(2,6)}$.
\end{proof}

\subsubsection{Proof of Lemma~\ref{lem:varepsilon3}}

The objective of this section is to prove the upper bounds~\eqref{eq:lem_varepsilon2} for the error terms $\varepsilon_{n,k}^{(3,1)}$, $\varepsilon_{n,k}^{(3,2)}$ and $\varepsilon_{n,k}^{(3,3)}$, defined by~\eqref{eq:def_varepsnk31},~\eqref{eq:def_varepsnk32} and~\eqref{eq:def_varepsnk33}, related to the comparison of the exact and numerical solutions.

\begin{proof}[Proof of Lemma~\ref{lem:varepsilon3}]

The error terms $\varepsilon_{n,k}^{(3,1)}$, $\varepsilon_{n,k}^{(3,2)}$ and $\varepsilon_{n,k}^{(3,3)}$ are treated separately.

$\bullet$ {\bf Treatment of the error term $\varepsilon_{n,k}^{(3,1)}$.}

The error term $\varepsilon_{n,k}^{(3,1)}$ is defined by~\eqref{eq:def_varepsnk31}. he mapping $g^0_k$ is Lipschitz continuous on $\mathcal{D}$. Applying the Cauchy--Schwarz inequality, one obtains 
\begin{align}
\E[|e_{n,k}^{(3,1)}|^2]&\le C(T)\dt \sum_{j=0}^{n-1}\E[\|X_j-X(t_j)\|^2]\nonumber\\
&\le C(T)\dt \sum_{j=0}^{n-1}\E[\|e_j\|^2].\label{eq:bound_varepsnk31}
\end{align}

$\bullet$ {\bf Treatment of the error term $\varepsilon_{n,k}^{(3,2)}$.}

The error term $\varepsilon_{n,k}^{(3,2)}$ is defined by~\eqref{eq:def_varepsnk32}. The mapping $g_k$ is Lipschitz continuous on $\mathcal{D}$. Applying the It\^o isometry formula, one obtains
\begin{align}
\E[|e_{n,k}^{(3,2)}|^2]&\le C(T)\dt \sum_{j=0}^{n-1}\E[\|X_j-X(t_j)\|^2]\\
&\le C(T)\dt \sum_{j=0}^{n-1}\E[\|e_j\|^2].\label{eq:bound_varepsnk32}
\end{align}

$\bullet$ {\bf Treatment of the error term $\varepsilon_{n,k}^{(3,3)}$.}

The error term $\varepsilon_{n,k}^{(3,3)}$ is defined by~\eqref{eq:def_varepsnk33}.  Applying the It\^o isometry formula, a conditional expectation argument and the identity $\E[|B(t)-B(t_j)|^2]=t-t_j$, one obtains 
\begin{align*}
\E[|e_{n,k}^{(3,3)}|^2]&\le C(T)\sum_{j=0}^{n-1}\int_{t_j}^{t_{j+1}} 
\E\left[ \left| \Bigl[\bigl(g\cdot \nabla\bigr)g_k(X_j)-\bigl(g\cdot\nabla\bigr)g_k(X(t_j))\Bigr]\bigl[B(t)-B(t_j)\bigr]\right|^2 \right]\dd t\\
&\leq C(T)\dt^2\sum_{j=0}^{n-1} \E\left[ \left| \Bigl[\bigl(g\cdot \nabla\bigr)g_k(X_j)-\bigl(g\cdot\nabla\bigr)g_k(X(t_j))\Bigr]\right|^2 \right].
\end{align*}
Finally, the mapping $(g\cdot\nabla)g_k$ is Lipschitz continuous on $\mathcal{D}$, therefore one obtains 
\begin{align}
\E[|e_{n,k}^{(3,3)}|^2]&\le C(T)\dt^2\sum_{j=0}^{n-1} \E[\|e_j\|^2]\nonumber\\
&\le C(T)\dt\sum_{j=0}^{n-1} \E[\|e_j\|^2].\label{eq:bound_varepsnk33}
\end{align}

$\bullet$ {\bf Conclusion.}

It suffices to combine the upper bounds~\eqref{eq:bound_varepsnk31},~\eqref{eq:bound_varepsnk32}and~\eqref{eq:bound_varepsnk33} for the error terms $\varepsilon_{n,k}^{(3,1)}$, $\varepsilon_{n,k}^{(3,2)}$ and $\varepsilon_{n,k}^{(3,3)}$ defined by~\eqref{eq:def_varepsnk31},~\eqref{eq:def_varepsnk32} and~\eqref{eq:def_varepsnk33} to obtain the inequality~\eqref{eq:lem_varepsilon3}. This concludes the proof of Lemma~\ref{lem:varepsilon3}.
\end{proof}

\subsection{Proof of Theorem~\ref{theo:stronghigherorder}}\label{sec:conclusionproofstronghigher}

\begin{proof}[Proof of Theorem~\ref{theo:stronghigherorder}]
Recalling the decomposition~\eqref{eq:decomp_enk} of the error term $e_{n,k}$, and combining the inequalities~\eqref{eq:lem_varepsilon1},~\eqref{eq:lem_varepsilon2} and~\eqref{lem:varepsilon3} from Lemmas~\ref{lem:varepsilon1},~\ref{lem:varepsilon2} and~\ref{lem:varepsilon3}, one obtains the following upper bound: there exists $C(T)\in(0,\infty)$ such that for all $n\in\{1,\ldots,N\}$ one has
\[
\E[\|e_n\|^2]=\sum_{k=1}^{d}\E[|e_{n,k}|^2]\le C(T)\dt^2+C(T)\dt\sum_{j=0}^{n-1} \E[\|e_j\|^2].
\]
Applying the discrete Gr\"onwall inequality then provides the strong error estimates~\eqref{eq:stronghigherorder} and the proof of Theorem~\ref{theo:stronghigherorder} is thus completed.
\end{proof}

\begin{appendix}
\section{Proof of Lemmas~\ref{lem:examples1} and~\ref{lem:examples2}}\label{app}

The mappings $\varphi$ and $\phi$ defined by~\eqref{eq:flowvarphi} and~\eqref{eq:flowphi} are the flows of the ordinary differential equations~\eqref{eq:y} and~\eqref{eq:y2} with vector fields $\sigma$ and $-\frac12\sigma'\sigma$. These vector fields are of class $\mathcal{C}^\infty$ and one has $\sigma(-1)=\sigma(+1)=0$. Therefore, one has the following properties:
\begin{itemize}
\item for all $y\in[-1,1]$ and all $s\in\R$, one has $\varphi(s,y)\in[-1,1]$ and $\phi(s,y)\in[-1,1]$ 
\item for all $y\in[-1,1]$, one has
\begin{align*}
&\varphi(0,y)=\phi(0,y)=0,\\
&\partial_s\varphi(0,y)=\sigma(y),\quad \partial_s\phi(0,y)=-\frac12(\sigma'\sigma)(y),\\
&\partial_{s}^2\varphi(0,y)=\frac12(\sigma'\sigma)(y).
\end{align*}
\end{itemize}

\begin{proof}[Proof of Lemma~\ref{lem:examples1}]
First, consider the integrator $\Phi$ defined by~\eqref{integratorA} from Example~\ref{exA}. The condition~\eqref{eq:condPhiDP} from Assumption~\ref{ass:integratorDP} is satisfied: for all $s\ge 0$, $\gamma\in\R$ and $y\in[-1,1]$, one has
\[
\Phi(s,\gamma,y)=\varphi(\gamma,\phi(s,y))\in[-1,1].
\]
To verify that Assumptions~\ref{ass:integrator} and~\ref{ass:integrator-higher} are satisfied, one needs to check the following properties.
\begin{itemize}
\item The condition~\eqref{eq:condPhi0} is satisfied: for all $y\in[-1,1]$,
\[
\Phi(0,0,y)=\varphi(0,\phi(0,y))=\phi(0,y)=y.
\]
\item The condition~\eqref{eq:condPhi1} is satisfied: for all $y\in[-1,1]$,
\[
\partial_\gamma\Phi(0,0,y)=\partial_s\varphi(0,\phi(0,y))=\sigma(\phi(0,y))=\sigma(y).
\]
\item The condition~\eqref{eq:condPhi-higher} is satisfied: for all $y\in[-1,1]$,
\begin{align*}
&\partial_s\Phi(0,0,y)=\partial_y\varphi(0,\phi(0,y))\partial_s\phi(0,y)=-\frac12\bigl(\sigma'\sigma\bigr)(y),\\
&\partial_\gamma^2\Phi(0,0,y)=\partial_s^2\varphi(\phi(0,y))=\frac12\bigl(\sigma'\sigma\bigr)(y).
\end{align*}
\item The condition~\eqref{eq:condPhi2} is satisfied as a straightforward consequence of the condition~\eqref{eq:condPhi-higher}.
\end{itemize}

Second, consider the integrator $\Phi$ defined by~\eqref{integratorB} from Example~\ref{exB}. The condition~\eqref{eq:condPhiDP} from Assumption~\ref{ass:integratorDP} is satisfied: for all $s\ge 0$, $\gamma\in\R$ and $y\in[-1,1]$, one has
\[
\Phi(s,\gamma,y)=\varphi(\gamma-\frac{\sigma'(y)s}{2},y)\in[-1,1].
\]
To verify that Assumptions~\ref{ass:integrator} and~\ref{ass:integrator-higher} are satisfied, one needs to check the following properties.
\begin{itemize}
\item The condition~\eqref{eq:condPhi0} is satisfied: for all $y\in[-1,1]$,
\[
\Phi(0,0,y)=\varphi(0,y)=y.
\]
\item The condition~\eqref{eq:condPhi1} is satisfied: for all $y\in[-1,1]$,
\[
\partial_\gamma\Phi(0,0,y)=\partial_s\varphi(0,y)=\sigma(y).
\]
\item The condition~\eqref{eq:condPhi-higher} is satisfied: for all $y\in[-1,1]$,
\begin{align*}
&\partial_s\Phi(0,0,y)=-\partial_s\varphi(0,y)\frac{\sigma'(y)}{2}=-\frac12\bigl(\sigma'\sigma\bigr)(y),\\
&\partial_\gamma^2\Phi(0,0,y)=\partial_s^2\varphi(y)=\frac12\bigl(\sigma'\sigma\bigr)(y).
\end{align*}
\item The condition~\eqref{eq:condPhi2} is satisfied as a straightforward consequence of the condition~\eqref{eq:condPhi-higher}.
\end{itemize}
This concludes the proof of Lemma~\ref{lem:examples1}.
\end{proof}

The derivatives of $\varphi$ and $\phi$ satisfy the following property: for all $p\in\N$, there exists $C_p\in(0,\infty)$ such that for all $y\in[-1,1]$ and all $s\in\R$ one has
\begin{equation}\label{eq:reg-varphi-phi}
\sum_{|\alpha|\le p}\left|\partial_s^{\alpha_s}\partial_y^{\alpha_y}\varphi(s,y)\right|+\sum_{|\alpha|\le p}\left|\partial_s^{\alpha_s}\partial_y^{\alpha_y}\phi(s,y)\right|\le e^{C_p|s|},
\end{equation}
where $\alpha=(\alpha_s,\alpha_y)\in\N_0^2$ and $|\alpha|=\alpha_s+\alpha_y$. The property~\eqref{eq:reg-varphi-phi} is a standard property for flows of ordinary differential equations with smooth vector fields. It suffices to establish first the property recursively for derivatives with respect to the variable $y$, and second the property when derivatives to the variable $s$ appear using the ordinary differential equations~\eqref{eq:y} and~\eqref{eq:y2}. The details are omitted.

\begin{proof}[Proof of Lemma~\ref{lem:examples2}]
The integrators $\Phi$ defined by~\eqref{integratorA} or~\eqref{integratorB} are written as compositions of the mappings $\varphi$, $\phi$ and $\sigma'$, therefore these integrators are of class $\mathcal{C}^\infty$.

For all $p\in\N$, the derivatives of $\Phi$ can be expressed in terms of the derivatives of the mappings $\varphi$, $\phi$ and $\sigma'$, which all satisfy the condition~\eqref{eq:reg-varphi-phi}. The existence of $\beta\in(0,\infty)$ (depending on $p$) such that $\|\Phi\|_{p,\beta}<\infty$ (given by~\eqref{eq:regPhi}) can be established by a recursion argument or by the application of the Fa\`a di Bruno formula. The details are omitted. The property can also be checked directly for the values of $p$ which are needed for Theorem~\ref{theo:strong} ($p=2$), Theorem~\ref{theo:weak} ($p=3$) and Theorem~\ref{theo:stronghigherorder} ($p=4$). This concludes the proof of Lemma~\ref{lem:examples2}.
\end{proof}

\end{appendix}

\section*{Acknowledgments}
This work was initiated thanks to the support of the SFVE-A program.
The work of DC was partially supported by the Swedish Research Council (VR) (projects nr. $2018-04443$
and $2024-04536$). The computations were performed on resources provided by
the National Academic Infrastructure for Supercomputing in Sweden (NAISS) at Vera, Chalmers e-Commons
at Chalmers University of Technology and partially funded by the Swedish Research Council
through grant agreement no. 2022-06725.

\end{document}